\documentclass[11pt]{amsart}
\usepackage{amsmath}
\usepackage{tikz}
\usepackage{amsthm}
\usepackage{amssymb}
\usepackage{mathrsfs}
\usepackage{verbatim}
\usepackage{hyperref}
\usepackage{color}
\usepackage{esint}
\usepackage{float}
\usepackage{enumerate}
\usepackage{graphicx}
\usepackage{caption}

\usepackage[T1]{fontenc} 

\begin{document}

% define theorem environments
\newtheorem{theorem}{Theorem}    %[section]

\newtheorem{proposition}[theorem]{Proposition}
\newtheorem{conjecture}[theorem]{Conjecture}
\def\theconjecture{\unskip}
\newtheorem{corollary}[theorem]{Corollary}
\newtheorem{lemma}[theorem]{Lemma}
\newtheorem{sublemma}[theorem]{Sublemma}
\theoremstyle{definition}
\newtheorem{observation}[theorem]{Observation}
\newtheorem{remark}[theorem]{Remark}
\newtheorem{definition}[theorem]{Definition}
\newtheorem{notation}[theorem]{Notation}
\newtheorem{question}[theorem]{Question}
\newtheorem{questions}[theorem]{Questions}
\newtheorem{example}[theorem]{Example}
\newtheorem{problem}[theorem]{Problem}
\newtheorem{exercise}[theorem]{Exercise}

\numberwithin{theorem}{section} \numberwithin{theorem}{section}
\numberwithin{equation}{section}

\def\earrow{{\mathbf e}}
\def\rarrow{{\mathbf r}}
\def\uarrow{{\mathbf u}}
\def\varrow{{\mathbf V}}
\def\tpar{T_{\rm par}}
\def\apar{A_{\rm par}}

\def\reals{{\mathbb R}}
\def\torus{{\mathbb T}}
\def\heis{{\mathbb H}}
\def\integers{{\mathbb Z}}
\def\naturals{{\mathbb N}}
\def\complex{{\mathbb C}\/}
\def\distance{\operatorname{distance}\,}
\def\support{\operatorname{support}\,}
\def\dist{\operatorname{dist}\,}
\def\Span{\operatorname{span}\,}
\def\degree{\operatorname{degree}\,}
\def\kernel{\operatorname{kernel}\,}
\def\dim{\operatorname{dim}\,}
\def\codim{\operatorname{codim}}
\def\trace{\operatorname{trace\,}}
\def\Span{\operatorname{span}\,}
\def\dimension{\operatorname{dimension}\,}
\def\codimension{\operatorname{codimension}\,}
\def\nullspace{\scriptk}
\def\kernel{\operatorname{Ker}}
\def\ZZ{ {\mathbb Z} }
\def\p{\partial}
\def\rp{{ ^{-1} }}
\def\Re{\operatorname{Re\,} }
\def\Im{\operatorname{Im\,} }
\def\ov{\overline}
\def\eps{\varepsilon}
\def\lt{L^2}
\def\diver{\operatorname{div}}
\def\curl{\operatorname{curl}}
\def\etta{\eta}
\newcommand{\norm}[1]{ \|  #1 \|}
\def\expect{\mathbb E}
\def\bull{$\bullet$\ }
\def\C{\mathbb{C}}
\def\R{\mathbb{R}}
\def\Rd{{\mathbb{R}^d}}
\def\ep{\varepsilon}
\def\2{\frac{1}{2}}
\def\p{\frac{1}{p}}
\def\Sn{{{S}^{n-1}}}
\def\M{\mathbb{M}}
\def\N{\mathbb{N}}
\def\Q{{\mathbb{Q}}}
\def\Z{\mathbb{Z}}
\def\F{\mathcal{F}}

\def\S{\mathcal{S}}
\def\supp{\operatorname{supp}}
\def\dist{\operatorname{dist}}
\def\essi{\operatornamewithlimits{ess\,inf}}
\def\esss{\operatornamewithlimits{ess\,sup}}
\def\xone{x_1}
\def\xtwo{x_2}
\def\xq{x_2+x_1^2}
\def\Oz{\Omega}
\def\oz{\omega}
\newcommand{\abr}[1]{ \langle  #1 \rangle}

\def\Xint#1{\mathchoice
{\XXint\displaystyle\textstyle{#1}}%
{\XXint\textstyle\scriptstyle{#1}}%
{\XXint\scriptstyle\scriptscriptstyle{#1}}%
{\XXint\scriptscriptstyle\scriptscriptstyle{#1}}%
\!\int}
\def\XXint#1#2#3{{\setbox0=\hbox{$#1{#2#3}{\int}$ }
\vcenter{\hbox{$#2#3$ }}\kern-.6\wd0}}
\def\ddashint{\Xint=}

\newcommand{\ave}[1]{\langle #1\rangle}
\newcommand{\Norm}[1]{ \left\|  #1 \right\| }
\newcommand{\abs}[1]{ \left|  #1 \right| }
\newcommand{\set}[1]{ \left\{ #1 \right\} }
\def\one{\mathbf 1}
\def\whole{\mathbf V}
\newcommand{\modulo}[2]{[#1]_{#2}}
\renewcommand{\thefootnote}{\fnsymbol{footnote}}
%\footnote[0]{2000\textit{ Mathematics Subject Classification}
\def\scriptf{{\mathcal F}}
\def\scriptg{{\mathcal G}}
\def\scriptm{{\mathcal M}}
\def\scriptb{{\mathcal B}}
\def\scriptc{{\mathcal C}}
\def\scriptt{{\mathcal T}}
\def\scripti{{\mathcal I}}
\def\scripte{{\mathcal E}}
\def\scriptv{{\mathcal V}}
\def\scriptw{{\mathcal W}}
\def\scriptu{{\mathcal U}}
\def\scriptS{{\mathcal S}}
\def\scripta{{\mathcal A}}
\def\scriptr{{\mathcal R}}
\def\scripto{{\mathcal O}}
\def\scripth{{\mathcal H}}
\def\scriptd{{\mathcal D}}
\def\scriptl{{\mathcal L}}
\def\scriptn{{\mathcal N}}
\def\scriptp{{\mathcal P}}
\def\scriptk{{\mathcal K}}
\def\frakv{{\mathfrak V}}

\newcommand{\B}{\mathbb B}
\newcommand{\T}{\mathbb T}

\allowdisplaybreaks

\arraycolsep=1pt
%
%If a theorem-like environment should not be numbered,
%add * after \newtheorem, and delete the counter option such as [theorem].
\newtheorem*{remark0}{\indent\sc Remark}
%
%%%%% Proof %%%%%
\renewcommand{\proofname}{Proof.} % Name is small caps
%The following commands are available in the proof environment:
%\begin{proof}
%\end{proof}
%The end of a proof is marked with a square.
%%%%%%%%%%%%%%%%%%%%%%%%%%%%%%%%%%%%%%%%%

\title[Commutators of fractional integrals]
{Schatten properties of commutators of fractional integrals on spaces of homogeneous type}

% with applications to fractional Bessel operators}
%title of paper and the running head option

\author[Hyt\"onen \& Wu]{Tuomas Hyt\"onen and Lin Wu$^*$}
%%%%%%%%%%%%%%% footnote %%%%%%%%%%%%%%%%

%\footnote[0]{{2020} \textit{Mathematics Subject Classification}. 42B20, 42B35}

%In case \subjclass[2000] command is not effective
%(or the version of amsart.cls is old), write as follows instead:
%\renewcommand{\thefootnote}{\fnsymbol{footnote}}
%\footnote[0]{2000\textit{ Mathematics Subject Classification}.

%Primary 00; Secondary 00.}
%
\thanks{T.H. was supported by the Research Council of Finland (project no.\ 364208). L.W. is supported by China Scholarship Council (Grant No. 202406310137)}
\thanks{$^*$Corresponding author.}
%%%%%%%%%%%% Authors addresses %%%%%%%%%%%%%

\address{Aalto University, Department of Mathematics and Systems Analysis, P.O. Box 11100, FI-00076 Aalto, Finland} \email{tuomas.hytonen@aalto.fi}

\address{Xiamen University, School of Mathematical Sciences, Xiamen 361005, China} \email{wulin@stu.xmu.edu.cn}

%\address{}
%\email{}

%%%%%%%%%%%%%%%%%%%%%%%%%%%%%%%%%%%%%%%%%

\begin{abstract}

Extending classical results of Janson and Peetre (1988) on the Schatten class $S^p$ membership of commutators of Riesz potentials on the Euclidean space, we obtain analogous results for commutators $[b,T]$, 
where $T\in\{T_\ep,\widetilde T_\alpha\}$ belongs to either one of two natural classes of fractional integral operators on a space of homogeneous type.
Our approach is based on recent related work of Hyt\"{o}nen and Korte on singular (instead of fractional) integrals; working directly with the kernels, it differs from the Fourier analytic considerations of Janson and Peetre, covering new operators even when specialised to $\mathbb R^d$.

%We characterize the Schatten class membership of commutators of fractional integrals on space of homogeneous types, motivated by the recent related results of Hyt\"{o}nen--Korte \cite{HK} on singular integrals. This extends the results for the Riesz-potential operator on the Euclidean space. Compared with \cite{JP}'s method limited to operators with concrete Fourier transform forms, we provide a new approach in terms of the fractional oscillatory spaces. 

%The main case of our characterization in spaces of lower dimension $d\ge 2$ is as follows. For $p\in [d,\infty)$ and $ \ep\in (0,\p)$ describing the order of the fractional integral $T_\eps$, the commutator satisfies $[b,T_{\ep}] \in S^p$ if and only if $b$ belongs to the fractional homogeneous Sobolev space $\dot{B}_p^{\frac1p-\ep}(\mu)$. For $p\in (0,d)$ and a certain range of $\ep >0$,  we have $[b,T_{\ep}]\in S^p$ if and only if $b$ is constant. We also get partial results for other parameter values, including $0<d<2$.  

The cleanest case of our characterization in spaces of lower dimension $d> 2$ and satisfying a $(1,2)$-Poincar\'e inequality is as follows. For a parameter $\ep \in (0,\frac12-\frac{1}{d})$ describing the order of the fractional integral $T_\eps$, we have a dichotomy: If $\frac{d}{1+d\ep }<p<\frac{1}{\eps}$, then $[b,T_{\ep}]\in S^p$ if and only if $b$ belongs to a suitable Besov (or fractional Sobolev) space. If $0<p\leq \frac{d}{1+d\ep }$, then $[b,T_{\ep}]\in S^p$ if and only if $b$ is constant. This is analogous to the result for singular integrals, where a similar cut-off happens at $p=d$, formally corresponding to fractional order $\eps=0$. We also obtain results for other parameter values, including dimensions $0<d\leq 2$.

As an application, these results are used to show Schatten properties of commutators of fractional Bessel operators, complementing recent related results of Fan, Lacey, Li, and Xiong (2025) on commutators of singular integrals in the Bessel setting.   
\end{abstract}

\maketitle

\setcounter{tocdepth}{1}
\tableofcontents

\section{Introduction}

%\subsection{Backgroumd and motivation}

Fractional integral operators play a fundamental role in analysis due to their widespread applications in potential analysis, harmonic analysis, PDE and Sobolev embeddings. 
%In recent years, the study of fractional singular integrals and their commutators has been a topic of increasing research interest in boundedness, compactness and the membership to the Schatten class or weak Schatten class on the metric space.  

The main classical example of fractional operators on the Euclidean space $\mathbb{R}^d$ with $d\ge 1$ is the Riesz potential $(-\Delta)^{-\frac{\alpha}{2}}$, with $\alpha=d\ep>0$, given by %\cite{GP} for equivalent definitions.
%Recall that the Riesz potential $(-\Delta)^{-\frac{\alpha}{2}}$ of order $\alpha \in (0,d)$ is defined by 
\begin{equation}\label{eq:Riesz}
  (-\Delta)^{-\frac{\alpha}{2}}f(x)
=c_{d,\alpha}\int_{\R^d}\frac{f(y)}{|x-y|^{d-\alpha}}dy
=c_{d,d\eps}\int_{\R^d}\frac{f(y)}{|x-y|^{d(1-\ep)}}dy,
\end{equation}
where
\begin{equation*}
  c_{d,\alpha}
  =\frac{\Gamma(\frac{d-\alpha }{2})}{2^{\alpha}\pi^{\frac{d}{2}}\Gamma(\frac{\alpha}{2})};
\end{equation*}
see \cite[Chapter 6]{Grafakos}. We write the two equivalent formulas above, parametrised by $\alpha$ and $\ep$, since they give rise to two different classes of generalisations, as we will see below.

The topic of this paper is {\em commutators} of fractional integral operators with pointwise multipliers, namely, operators of the type
\begin{equation*}
  [b,T]f= bTf-T(bf),
%  [b,(-\Delta)^{-\frac{\alpha}{2}}]f=b(-\Delta)^{-\frac{\alpha}{2}}f-(-\Delta)^{-\frac{\alpha}{2}}(bf).
\end{equation*}
where $T=(-\Delta)^{-\frac{\alpha}{2}}$ or one of its generalisations that we shortly describe.

Some classical results in this theme are as follows:
In \cite{C1}, Chanillo showed that for any $0<\alpha<d$ and $1<p<\frac{d}{\alpha}$ as well as $ \frac{1}{q}=\frac{1}{p}-\frac{\alpha}{d}$, the commutator $[b, (-\Delta)^{-\frac{\alpha}{2}}]$ is bounded from $L^p(\mathbb{R}^d)$ to $L^q(\mathbb{R}^d)$
if and only if  $b\in BMO(\mathbb{R}^d)$, the space of functions of bounded mean oscillation.
 Later, Wang \cite{W} proved that $[b, (-\Delta)^{-\frac{\alpha}{2}}] $ is compact from $L^p(\mathbb{R}^d)$ to $L^q(\mathbb{R}^d)$
if and only if  $b\in VMO(\mathbb{R}^d)$, the $BMO$-closure of $C_c^{\infty}(\mathbb{R}^d)$.
%The characterization of compactness on Morrey spaces was proved by Chen-Ding-Wang \cite{CDW}.

In this paper, we are particularly interested in quantitative versions of compactness measured in terms of the Schatten $S^p$ norms
\begin{equation}\label{eq:Sp}
  \Norm{R}_{S^p(L^2(\mu))}
  :=\Big(\sum_{n=0}^\infty a_n(R)^p\Big)^{\frac1p},\qquad
  a_n(R):=\inf\{\Norm{R-F}_{L^2(\mu)\to L^2(\mu)}:\operatorname{rank}F\leq n\},
\end{equation}
where $a_n(R)$ is the $n$th approximation number (or singular value) of $R:L^2(\mu)\to L^2(\mu)$.
In this direction, 
%In addition, we would like to highlight the following Schatten properties of the commutator generated by the Riesz potential $(-\Delta)^{-\frac{\alpha}{2}}$, on the Euclidean space $\mathbb{R}^d\,\,(d\ge 2)$. As for the strong-type of the Riesz potential commutators,
Janson and Peetre \cite[p.~484]{JP} obtained the following results in the Euclidean space $\mathbb{R}^d$ with $d\ge 2$
 as a special case of their work on so-called ``paracommutators'':
\begin{enumerate}
\item For $p\ge 1$ and $ \big(\frac{d}{p}-1\big)_+<\alpha<\min \{\frac{d}{p},\frac{d}{2} \}$, the commutator $[b,{(-\Delta)}^{-\frac{\alpha}{2}}]$ belongs to the Schatten class $S^p(L^2(\R^d))$ if and only if $b$ belongs to the classical fractional Sobolev space $\dot{B}_{p,p}^{\frac{d}{p}-\alpha}(\mathbb{R}^d)$ (see the definition in (\ref{bRn})).
\item For $1\leq p <d$ and $0< \alpha \leq \frac{d}{p}-1$, the commutator $[b,{(-\Delta)}^{-\frac{\alpha}{2}}]$ belongs to the Schatten class $S^p(L^2(\R^d))$ if and only if $b$ is constant. 
\end{enumerate}

For $d\ge 2$ and $\ep \in (0,\frac{1}{2})$, these results can be restated as follows:
\begin{equation}\label{eqJP}
[b, (-\Delta)^{-\frac{d\ep}{2}}]\in S^p(L^2(\R^d)) \iff 
\begin{cases}
  b \in \dot{B}_{p,p}^{d(\frac{1}{p}-\ep)}(\mathbb{R}^d),& \quad \frac{d}{1+d\ep}<p<\frac{1}{\ep},\\
  b =\mathrm{constant},&\quad 0<p\leq \frac{d}{1+d\ep}.
\end{cases}	
\end{equation}

At the critical point $p=\frac{d}{1+d\ep}$, 
Frank, Sukochev, and Zanin \cite{FSZ} showed that $[b,(-\Delta)^{-\frac{d\ep}{2}}]$ belongs to the weak Schatten class $S^{p,\infty}$ if and only if $b$ belongs to the homogeneous Sobolev space $\dot{W}^{1}_{p}(\mathbb{R}^d)$.
For the corresponding Schatten $S^p$ properties of the fractional Laplacian operator $(-\Delta)^{\frac{\alpha}{2}}$ with positive order $\alpha>0$, we refer the reader to \cite{FSZ},\cite{JP},\cite{M}. These questions fall outside the scope of the present work.

Several authors have also considered related questions in the more general setting of a space of homogeneous type $(X,\rho,\mu)$, which is a set $X$ with a quasi-distance $\rho$ and a positive measure $\mu$ such that the balls defined by $B(x,r)=\{y\in X:\rho(x,y)<r\}$ satisfy a doubling condition; see Section \ref{sec:def} for a detailed definition. In this setting, two types of fractional integrals are defined by
% a ``volumic'' fractional integral is defined by
\begin{equation}\label{def}
	I_{\ep}f(x)=\int_{X} \frac{1}{V(x,y)^{1-\ep}} f(y)d\mu(y),\qquad
	\tilde I_{\alpha}f(x)=\int_X\frac{\rho(x,y)^\alpha}{V(x,y)}f(y)d\mu(y),
\end{equation}
where $V(x,y):=\mu(B(x,\rho(x,y)))$; we will refer to them as ``volumic'' and ``metric'', respectively.
When $X=\mathbb{R}^d$ with $\rho(x,y)=|x-y|$ and $d\mu=dx$, both $I_{\ep}$ and $\tilde I_\alpha=\tilde I_{d\eps}$ reduce to the classical Riesz potential $(-\Delta)^{-\frac{d\eps}{2}}$. Note that $I_\eps$ arises by interpreting the whole $|x-y|^{d(1-\eps)}$ in \eqref{eq:Riesz} as $V(x,y)^{1-\eps}$, while $\tilde I_\alpha$ is based on applying a different interpretation $|x-y|^d\sim V(x,y)$ and $|x-y|^{\alpha}=\rho(x,y)^{\alpha}$ to the two factors of $|x-y|^{d-\alpha}$.

Many works on fractional integrals over spaces of homogeneous type, like \cite{GV}, are formulated in so-called ``normal'' spaces with $V(x,y)\sim \rho(x,y)$, in which case the volumic and the metric versions coincide. Without assuming normality, volumic fractional integrals have been studied e.g.\ in \cite{Pan}, and \cite{BC} obtained the $(L^p,L^q)$ boundedness of their commutators $[b,I_{\ep}]$ for $b \in BMO(X)$. These volumic fractional integrals admit a relatively clean theory in its own right. Nevertheless, it seems that metric fractional integrals, especially with $\alpha=1$, are actually the ones that more frequently arise in applications; see e.g.\ \cite[Theorem 3.22 and Section 9.1]{Heinonen} and \cite[Eq.~(7)]{PW}, where further references to such operators in different contexts are given. In particular, fractional powers $(-\Delta_\lambda)^{-\frac{\alpha}{2}}$ of the Bessel Laplacian $\Delta_\lambda$ turn out to be of the metric form; see \cite{BDLL} and Section \ref{sec:Bessel} below. Incorporating this prominent example into our theory was a major motivation for dealing with $\tilde I_\alpha$, and we will return to this example in more detail below. More generally, we show in Section \ref{heat kernel} that a large class of fractional operators arising from heat kernels fall under the umbrella of metric fractional integrals.

% Bramanti and Cerutti \cite{BC} first obtained the $(L^p,L^q)$ boundedness of the commutators $[b,I_{\ep}]$ if $b \in BMO(X)$ on the homogenous type space $(X,\rho,\mu)$ satisfying a certain reverse doubling condition. %We refer to \cite{BC}, Definition 1.3 for more details of ``annuli''.  
%Moreover, some results on fractional integral operators remain valid even in the absence of the doubling condition; see \cite{FYY,GG,HMY}.

% belongs to a general class of fractional integral operators on a space of homogeneous type $X$. Specifically, the operator $T$ can be divided into two types: $T_{\ep}$ and $\widetilde{T_{\alpha}}$ as shown in Definition \ref{DefT}. The kernel of $T_{\ep}$ has size like that of $I_\eps$, plus a little regularity, as described in Definition \ref{D1} below.
 
The aim of this paper is to investigate the Schatten class $S^p$ properties of the commutators $[b,T]$, where $T\in\{T_{\ep},\widetilde{T_{\alpha}}\}$ belongs to one of two classes of fractional integral operators modelled after $I_\eps$ and $\tilde I_\alpha$, respectively. We will show that $[b,T]\in S^p$ if and only if $b$ belongs to a suitable Besov space, with certain fractional oscillatory spaces %$\mathrm{Osc}_{\phi}^{p,q}(\mu)$ 
as intermediate steps in proving this equivalence.

%We aim to provide a new characterization of Schatten class $S^p$ properties of the general fractional type commutator $[b,T_{\ep}] $ and $[b,\widetilde{T_{\alpha}}]$ on space of homogeneous types, in terms of the membership of $b$ in some Besov spaces $\dot{B}_p^{\ep}(\mu)$ and $ \widetilde{B}_p^{\alpha}(\mu) $, respectively, with certain fractional oscillatory spaces $\mathrm{Osc}_{\phi}^{p,q}(\mu)$ as intermediate steps in proving this equivalence.

Our approach is based on the recent works of Hyt\"{o}nen and Korte \cite{H2,HK} where, building on the work of Janson--Peetre \cite{JP} and Rochberg--Semmes \cite{RS} in the Euclidean case, they established similar results on spaces of homogeneous type for singular instead of fractional integrals, corresponding formally to the case $\ep,\alpha=0$. Even in the Euclidean setting $X=\mathbb{R}^d$, our framework complements the results of Janson and Peetre \cite{JP}, whose operators are defined on the Fourier transform side, in contrast to the more direct spatial description in our theory. While the basic case of $[b,(-\Delta)^{-\frac{\alpha}{2}}]$ is covered by both, the examples beyond that are not comparable.
  
The following corollary, with clean conclusions under somewhat stronger assumptions than our main Theorem \ref{T} and Corollary \ref{C},  serves as an illustration of our results.
%In particular, if $d>2$, we present necessary and sufficient conditions for the fractional commutators $[b,T]$ to belong to the Schatten class $S^p$ for all $p\in (0,\infty)$. 

\begin{corollary}\label{C'}
	Let $(X,\rho,\mu)$ be a space of homogeneous type with lower dimension $d> 2$ (Definition \ref{def:dim}) and satisfying the $(1,2)$-Poincar\'e inequality (Definition \ref{def:Poincare}). 
Suppose that $\ep \in (0,\frac12-\frac{1}{d})$ and $\alpha=d\ep\in(0,\frac{d}{2}-1)$.
Then the following hold for all $b \in L_{\mathrm{loc}}^1(X)$:
%\begin{equation*}
%  [b,T_\eps]\in S^p\quad\iff\quad\begin{cases} b\in \dot{B}_p^{\frac1p-\ep}(\mu), & p\in(\frac{d}{1+d\eps},\frac{1}{\eps}), \\
%  b=\textrm{const}, & p\in(0,\frac{d}{1+d\eps}],\end{cases}
%\end{equation*}
%
\begin{enumerate}[\rm(i)]

\item\label{C:p>d'} If $p\in (\frac{d}{1+d\ep},\frac{1}{\eps}) =(\frac{d}{1+\alpha},\frac{d}{\alpha})$, then 
$$\begin{cases}
		[b,I_{\ep}]\in S^p(L^2(\mu)) &\iff  \quad b \in \dot{B}_p^{\frac1p-\ep}(\mu),\\
		[b,\tilde{I}_{\alpha}]\in S^p(L^2(\mu)) &\iff \quad b \in \widetilde{B}_p^{\alpha}(\mu),
	\end{cases}$$
where the two Besov spaces on the right are defined in \eqref{b2} and \eqref{b2 variant}, respectively.

\item\label{C:p<d'} If $p\in(0,\frac{d}{1+d\ep}]=(0,\frac{d}{1+\alpha}]$,
then $[b,I_{\ep}]\in S^p(L^2(\mu))$ or $[b,\tilde{I}_{\alpha}]\in S^p(L^2(\mu))$ if and only if $b$ is constant.
\end{enumerate}
More generally, the same conclusions hold for all strongly non-degenerate $\phi$-fractional integral operators (Definition \ref{DefT}) $T_\eps$ in place of $I_\eps$ and $\widetilde T_\alpha$ in place of $\tilde I_\alpha$, where $\phi(x,y)\in\{V(x,y)^\eps,\rho(x,y)^\alpha\}$, respectively.
\end{corollary}

We note that the $(1,2)$-Poincar\'e inequality is a natural assumption, in the sense that it is the version of the Poincar\'e inequality most frequently established in concrete situations; see e.g. Baudoin et al. \cite{BBG}.

\begin{remark}
%When $(X,\rho,\mu)=(\R^d,|x-y|,dx)$ is the Euclidean space, 
For $(X,\rho,\mu)=(\R^d,|x-y|,dx)$,
both $\dot B_p^{\frac1p-\ep}(\mu)$ and $\widetilde{B}_p^{d\eps}(\mu)$ coincide with the classical Besov space $\dot B_{p,p}^{d(\frac1p-\eps)}(\R^d)$, and Corollary \ref{C'} applies in particular to $I_{\eps}=\tilde I_{d\eps}=(-\Delta)^{-\frac{d\eps}{2}}$. We see that the Corollary reproduces the classical Janson--Peetre result \eqref{eqJP}, except for the fact that our restrictions on the dimension (our $d>2$ vs.\ $d\geq 2$ in \eqref{eqJP}) and the fractional parameter (our $\eps\in(0,\frac12-\frac1d)$ vs.\ $\eps\in(0,\frac12)$ in \eqref{eqJP}) are somewhat stronger. These restrictions arise from a limitation of our method described in Remark \ref{R1.14}. On the other hand, in this smaller range, Corollary \ref{C'} not only recovers \eqref{eqJP} for $(-\Delta)^{-\frac{\alpha}{2}}$, but also covers a large range of other fractional operators as in Definition \ref{DefT}.
\end{remark}

As a more serious application of our abstract results, we characterise the Schatten properties of commutators of fractional powers of the Bessel Laplacian 
\begin{equation}\label{eq5.1}
\Delta_{\lambda}^{(n+1)}:=\frac{\partial^2}{\partial x_1^2}+\dots
 +\frac{\partial^2}{\partial x_n^2} +\frac{\partial^2}{\partial x_{n+1}^2} +\frac{2\lambda}{x_{n+1}} \cdot \frac{\partial}{\partial x_{n+1}}
\end{equation}
on $\R_+^{n+1}=\R^n\times(0,\infty)$. We give here the following illustrative result, leaving a more general statement for Corollary \ref{PB2}:

\begin{corollary}\label{PB2'}
Let $n\ge 2$, $\lambda>0$, and $0<\alpha<\frac{n-1}{2}$. Let $(-\Delta_{\lambda})^{-\alpha / 2}$ be the fractional Bessel operator in $(\mathbb{R}^{n+1}_+,|\cdot |,dm_{\lambda}^{(n+1)}) $, where $dm_\lambda^{(n+1)}(x)=x_{n+1}^{2\lambda}dx$. Then the following conclusions hold for all $b \in L_{\mathrm{loc}}^1(\mathbb{R}^{n+1}_+)$:
\begin{equation*}
  [b,(-\Delta_\lambda)^{-\frac{\alpha}{2}}]\in S^p(L^2(dm_\lambda^{(n+1)}))\quad\iff\quad
  \begin{cases} b\in \widetilde{B}_{p}^{\alpha}(dm_{\lambda}^{(n+1)}), & \text{if}\quad p\in (\frac{n+1}{\alpha+1},\frac{n+1}{\alpha}), \\
  b=\text{const}, & \text{if}\quad p\in(0, \frac{n+1}{\alpha+1}],\end{cases}
\end{equation*}
where $\widetilde{B}_{p}^{\alpha}(dm_{\lambda}^{(n+1)})$ is defined as in \eqref{b2 variant} with $(X,\rho,\mu)=(\mathbb{R}^{n+1}_+,|\cdot |,dm_{\lambda}^{(n+1)})$.
%
%	\begin{enumerate}
%
%\item\label{C8.2:p>d variant'} If $p\in (\frac{n+1}{\alpha+1},\frac{n+1}{\alpha}])$,  then $[b, (-\Delta_{\lambda})^{-\alpha/2}]\in S^p$ if and only if $b \in \widetilde{B}_{p}^{\alpha}(dm_{\lambda}) $.
%  
%
%\item\label{C8.3:const variant'} If $p\in(0, \frac{n+1}{\alpha+1}]$,
%then $[b, (-\Delta_{\lambda})^{-\alpha/2}]\in S^p$ if and only if $b$ is constant. 
%\end{enumerate}
%
\end{corollary}

\begin{proof}[Sketch of proof]
This is a direct application of Corollary \ref{C'}, once we verify the following:
\begin{enumerate}[\rm(i)]
  \item $(\mathbb{R}^{n+1}_+,|\cdot |,dm_{\lambda}^{(n+1)})$ is a space of homogeneous type of lower dimension $d=n+1>2$ that satisfies the $(1,2)$ (in fact, even the stronger $(1,1)$) Poincar\'e inequality. This is \cite[Proposition 4.2]{H2}, restated as Proposition \ref{Bessel}.
  \item\label{it:Bessel-phi} $(-\Delta_\lambda)^{-\frac{\alpha}{2}}$ is a strongly non-degenerate $\phi$-fractional integral operator on this space with $\phi(x,y)=\abs{x-y}^\alpha$. This is Proposition \ref{P5.5}.
\end{enumerate}
\end{proof}

Part \eqref{it:Bessel-phi} of the proof above is due to \cite{BDLL} in the special case $n=0$ (i.e., on $\R_+=(0,\infty)$), but seems to be unavailable in the previous literature for $n\geq 1$. We establish these properties in Section \ref{sec:Bessel}. Besides Corollary \ref{PB2'}, this may have independent interest in bringing the fractional Bessel operators $(-\Delta_{\lambda})^{-\frac{\alpha}{2}}$ under the general umbrella of fractional integrals on spaces of homogeneous type, for which other results can then be directly quoted from the literature.

There are previous results related to Corollary \ref{PB2'} by Fan, Lacey, Li, and Xiong \cite{FLL}, who  
deal with the Schatten properties of commutators associated with the Bessel Riesz transforms  
$R_{\lambda,j} = \partial_{j} (-\Delta_\lambda)^{-\2}$.  
By the kernel estimates of $R_{\lambda,j}$ obtained in \cite{DGK},  
these are operators of singular integral type,
formally corresponding to $\widetilde{T}_{\alpha}$ with $\alpha = 0$,
which are in the scope of the theory of Hyt\"onen and Korte \cite{H2,HK}.

The structure of the paper is as follows. 
In Section \ref{sec:def}, we give the main definitions and state the general form of our main results.
In Section \ref{sec:prelim}, we prove some basic lemmas supporting the definition in Section \ref{sec:def}, and we introduce the definition of dyadic cubes as preparatory tools.
Section \ref{sec:upper} is devoted to establishing  upper Schatten bounds of the commutators $[b,T]$. 
In Section \ref{sec:Besov}, we present the equivalent characterization of the fractional Sobolev norms via some fractional oscillatory norms.
In Section \ref{sec:median}, we simplify the recent complex median method of Wei and Zhang \cite{WZ} and apply it to fractional integrals.
%extend the so-called approximate weak factorisation method of \cite{H} to deal with fractional integrals on spaces of homogeneous type.
%In Section \ref{sec:awf}, we extend the so-called approximate weak factorisation method of \cite{H} to deal with fractional integrals on spaces of homogeneous type.
Building on the results from Section \ref{sec:Besov} and \ref{sec:median}, we derive the lower Schatten bounds of the commutators in Section \ref{sec:lower}.
Adapting results from Hyt\"onen and Korte \cite{HK}, we show in Section \ref{sec:const} that certain Besov spaces only consist of constants.
The proofs the main results, including Corollary \ref{C'}, are then completed by synthesizing the preceding estimates in Section \ref{sec:synth}. 

The final three sections provide examples of fractional integrals that fall under the scope of our theory.
In Section \ref{sec:regular}, we study kernels with additional regularity, which is often available in applications.
In Section \ref{heat kernel}, we show that negative fractional powers $\mathcal L^{-s}$ of generators of heat semigroups $e^{-t\mathcal L}$, under quite general assumptions on the heat kernel $p_t(x,y)$, are metric fractional integrals in the sense of our definition. In Section \ref{sec:Bessel}, we deal with the specific case of fractional Bessel operators, after verifying that they fit into our general framework.

\subsection*{Notation}  
We write \(X \lesssim Y\) to mean \(X \leq C Y\) for some constant \(C > 0\) independent of key variables, and \(X \sim Y\) when both \(X \lesssim Y\) and \(Y \lesssim X\) hold. %The symbol \(l(Q)\) denotes the side length of the cube \(Q\).

\section{Definitions and main results}\label{sec:def}

We now provide the full set of relevant definitions and then state the general form of our main results, a special case of which was formulated in Corollary \ref{C'} in the Introduction.

Firstly, we recall that $(X,\rho,\mu)$ is a space of homogeneous type if $\rho:X \times X \rightarrow [0,+\infty)$ is a quasi-metric on the set $X$ satisfying the following properties:
$\mathrm{(\romannumeral1)}$ $\rho(x,y)=\rho(y,x)\ge 0$ for all $x,y\in X$,
$\mathrm{(\romannumeral2)}$ $\rho(x,y)=0$ if and only if $x=y$,
$\mathrm{(\romannumeral3)}$ there exists a constant $A_0\ge 1$ such that for all $x,y,z \in X$,
\begin{equation}\label{quasitri}
\rho(x,y)\leq A_0[\rho(x,z)+\rho(z,y)],	
\end{equation}
 and $\mu$ is a positive Borel measure on $X$, satisfying the doubling condition: 
\begin{equation}\label{eq1}
0<\mu(2B)\leq C\mu(B)<\infty, \quad \text{for   all balls }B.
\end{equation}
We abbreviate
\begin{equation*}
  V(x,r):=\mu(B(x,r)),\qquad
  V(x,y):=\begin{cases} V(x,\rho(x,y)), & x\neq y, \\ \mu(\{x\}), & x=y.\end{cases}
\end{equation*}

% For simplicity, we denote the measure $V(x,y):=\mu(B(x,\rho(x,y)))$ for any $x,y \in X$. 
 
 \begin{definition}\label{def:dim}
 We say that $(X,\rho,\mu)$ has upper dimension $D>0$ if there exists a constant  $C_{\mu}\ge 1$ such that for any $x \in X$ and $0<r\leq R<\infty$, 
\begin{equation}\label{dD}
	 \frac{V(x,R)}{V(x,r)}\leq C_{\mu}(\frac{R}{r})^D,  
\end{equation}
and $(X,\rho,\mu)$ has lower dimension $d>0$ if there exists a constant  $\widetilde{C}_{\mu}\ge 1$ such that for any $x \in X$ and $0<r\leq R<\infty$, 
\begin{equation}\label{dD'}
	 \frac{V(x,R)}{V(x,r)}\ge  \widetilde{C}_{\mu}(\frac{R}{r})^d.
\end{equation}
 Note that \eqref{dD} is equivalent to the doubling condition \eqref{eq1}.
 \end{definition}

A measure $\mu$ on $X$ is said to be Ahlfors $\gamma$-regular, if there is a constant $\beta \ge 1$ such that 
$${\beta}^{-1}r^{\gamma}\leq \mu(B(x,r))\leq \beta r^{\gamma},$$
for any $0<r<\infty$ and any ball $B(x,r)$ in $X$. 
A metric space $X$ carrying an Ahlfors $\gamma$-regular measure is called an Ahlfors $\gamma$-regular space.
It has both the upper dimension $D=\gamma$ and the lower dimension $d=\gamma$.

\subsection{Fractional integrals and their commutators}
 
 To streamline the discussion of the two different versions of fractional integrals, 
for $\ep\in (0,1)$ and $\alpha\in(0,\infty)$, we denote
 \begin{equation}\label{eq:phixr}
  \phi(x,r)=\phi(B(x,r))\in\{ V(x,r)^\eps, r^\alpha\}.
\end{equation}
and, consistently with the two cases in \eqref{eq:phixr},
\begin{equation}\label{eq:phixy}
  \phi(x,y):=\phi(x,\rho(x,y))\in\{ V(x,y)^\eps, \rho(x,y)^\alpha\}.
\end{equation}

\begin{definition}\label{D1}
A function $K \in L_{\mathrm{loc}}^1 (X \times X  )$ is called
%\begin{enumerate}[\rm(i)]
% \item 
 a {\em $\phi$-fractional integral kernel} if there is a constant $C_{K}$ such that
\begin{equation}\label{1.1}
|K(x,y)|\leq C_K \frac{\phi(x,y)}{V(x,y)}\qquad\Big(\text{with}\quad\frac{0}{0}:=0\Big)
%H(x,y),
\end{equation}
 for all $x,y \in X$. % with $x \neq  y$;
% \begin{equation}\label{DH}
% 	H(x,y)=\frac{1}{V(x,y)^{1-\ep}}
% \,\, \mathrm{or} \,\, \, \frac{\rho(x,y)^{\alpha}}{V(x,y)},
% \end{equation}
%the function $\omega :[0,1]\rightarrow [0,\infty)$
% is continuous, increasing, subadditive with $\omega(0)=0$;
%  \item a {\em Dini-regular $\phi$-fractional integral kernel} if, in addition, $\omega$ satisfies the Dini condition
% $$\|\omega\|_{Dini}=\int_0^1 \omega(t)\frac{dt}{t}<\infty.$$
%\end{enumerate}
\end{definition}

 \begin{definition}\label{D1'}
 A $\phi$-fractional integral kernel $K $ is said to be %called non-degenerate
% \begin{enumerate}[\rm(i)]
  %\item
   {\em strongly non-degenerate} if there are positive constants $A,c_1$ and $\overline{C}$, and $\eta\leq\frac{\pi}{9}$, such that
  for every $x_0\in X$ and $r>0$, there exists a point $y_0 \in B(x_0,\overline{C}Ar) \setminus B(x,Ar)$ such that for some $v\in\C$ with $\abs{v}=1$, we have at least one of the following two options:
\begin{equation}\label{eq:sndg1}
|K(x,y)|\ge c_1\cdot \frac{\phi(x_0,r)}{V(x_0,r)},\qquad
|\arg(\bar v K(x,y))|\leq\eta
\end{equation} 
for all $x\in B(x_0,r)$ and $y\in B(y_0,r)$, or
\begin{equation}\label{eq:sndg2}
|K(y,x)|\ge c_1\cdot \frac{\phi(x_0,r)}{V(x_0,r)},\qquad
|\arg(\bar v K(y,x))|\leq\eta
\end{equation} 
for all $x\in B(x_0,r)$ and $y\in B(y_0,r)$.
%\end{enumerate}
 \end{definition}

%\begin{remark}\label{rem:ndg1}
%A simple sufficient condition for strong non-degeneracy is that all annuli $B(x,\overline  C R)\setminus B(x,r)$ are non-empty, and
%\begin{equation*}
%  K(x,y)\geq c\frac{\phi(x,y)}{V(x,y)}
%\end{equation*}
%for all $x,y\in X$.
%\end{remark}

 \begin{remark}\label{rem:ndg2}
Although ``strongly non-degenerate'' is the main notion of non-degeneracy that we use in this paper, we reserve the simpler name ``non-degenerate'' for another variant (Definition \ref{D1reg}), since this variant is closer to the notion of ``non-degenerate'' for singular integral kernels as defined in \cite[Eq.\ (1.7)]{H2}.
 
The existence of a (strongly) non-degenerate kernel requires in particular that $B(x,r)\neq X$ for all $r>0$, hence that $X$ is unbounded, and thus (by \cite[Lemma 1.9]{BC}) that $\mu(X)=\infty$. Hence, this assumption is implicitly in force in all results dealing with (strongly) non-degenerate kernels. A modification can be made to accommodate spaces of finite diameter; see \cite[Section 10]{H2} for related discussion in the case of singular (instead of fractional) kernels.

Moreover, the existence of a (strongly) non-degenerate kernel also implies the non-empty annulus property $B(x,\overline{C}r) \setminus B(x,r)\neq\varnothing$ for all $x\in X$ and $r>0$, which is equivalent to inequality \eqref{dD'} for some $d>0$ by \cite[Remark 1.2]{HMY'}. In particular, this implies that $\mu(\{x\})=0$ for all $x\in X$, which is hence also implicitly assumed in all results dealing with (strongly) non-degenerate kernels.
 \end{remark}
 
 In many cases, strong non-degeneracy can be deduced from a simple two-sided bound for a non-negative fractional kernel. A prominent example is the fractional Bessel kernel, see \eqref{5.9}. We will give the proof of Lemma \ref{lem:K>0} in Section \ref{sec:lemmas}.
 
 \begin{lemma}\label{lem:K>0}
Let $(X,\rho,\mu)$ be a space of homogeneous type with the non-empty annulus property $B(x,\overline{C}r) \setminus B(x,r)\neq\varnothing$ for all $x\in X$ and $r>0$. Let $K$ be a $\phi$-fractional integral kernel that satisfies
\begin{equation}\label{eq:Ksim}
  K(x,y)\sim\frac{\phi(x,y)}{V(x,y)}
\end{equation}
for all $x,y\in X$. Then $\phi$ is strongly non-degenerate.
 \end{lemma}

\begin{definition}\label{DefT}
Let $K$ be a $\phi$-fractional integral kernel satisfying \eqref{eq:phixy} and \eqref{1.1}.
The associated {\em $\phi$-fractional integral} is defined by
\begin{equation}\label{1.3}
Tf(x)=\int_{X} K(x,y)f(y)d\mu(y)
\end{equation}
for all $f\in L^1_{\operatorname{loc}}(X)$ and $x\in X$ for which the integral \eqref{1.3} is well defined. The operator $T$ is called strongly non-degenerate whenever its kernel $K$ has the corresponding property. We write $T\in\{T_\eps,\widetilde T_\alpha\}$ according to the two cases in \eqref{eq:phixy}.
\end{definition}

\begin{example}\label{ex:basic} 
The basic fractional integrals $I_\eps$ and $\tilde I_\alpha$ from \eqref{def} are strongly non-degenerate $\phi$-fractional integrals with $\phi(x,y)=V(x,y)^\eps$ and $\phi(x,y)=\rho(x,y)^\alpha$, respectively.
%, but not necessarily $\omega$-regular.
\end{example}

\begin{proof}
That $I_\eps$ and $\tilde I_\alpha$ from \eqref{def} are $\phi$-fractional integrals with the respective $\phi$ is clear; indeed, their kernels are equal to the upper bound defining a $\phi$-fractional integral kernel. The strong non-degeneracy follows from Lemma \ref{lem:K>0}.
%In a general space of homogeneous type, $V(x,y)$ need not be continuous, and hence $\omega$-regularity may easily fail. For example, let $n\geq 2$ and
%\begin{equation*}
%  X=\{x=(x_i)_{i=1}^n\in\R^n\mid \exists i:x_i\in\Z\}
%\end{equation*}
%with the $\ell^\infty$ metric and the $(n-1)$-dimensional Lebesgue measure. This is an Ahlors $(n-1)$-regular space of homogeneous type, but $V(x,y)$ has jumps at the points where $x$ or $y$ has more than one integer coordinate.
\end{proof}

The following basic lemma, whose proof we postpone to Section \ref{sec:lemmas}, guarantees that \eqref{1.3} is well defined for a rather rich class of functions:
 
 \begin{lemma}\label{lem:Tf}
 Let $\phi$ and $K$ satisfy \eqref{eq:phixy} and \eqref{1.1}. Then
 \begin{enumerate}[\rm(i)]
   \item\label{it:Tf(x)} If $f\in L^1(X)$ is boundedly supported, then \eqref{1.3} is well defined for a.e.\ $x\in X$. 
   \item\label{it:Tf,g} If $f\in L^p(X)$ and $g\in L^{p'}(X)$ are boundedly supported, where $p\in[1,\infty]$ and $\frac{1}{p}+\frac{1}{p'}=1$, then $\langle Tf,g\rangle=\int_X Tf(x)g(x)d\mu(x)$ is well defined, and
\begin{equation}\label{eq:Tf,g}
   |\langle Tf,g\rangle|\leq\int_X\int_X|K(x,y) f(y)g(x)|d\mu(y)d\mu(x)
   \lesssim \phi(x_0,r)\Norm{f}_{L^p(X)}\Norm{g}_{L^{p'}(X)}
\end{equation}
if the supports of $f$ and $g$ are contained in $B(x_0,r)$.
 \end{enumerate}
 \end{lemma}

\begin{remark}
If the kernel $K$ satisfies (\ref{1.1}) and (\ref{1.2}) with $\ep,\alpha =0$, then $K$ is called a standard kernel and the corresponding operator $T$ as in  (\ref{1.3}) is called a Calder\'on-Zygmund operator. Giving a meaning to the integral \eqref{1.3} is much trickier in this case, and one usually only requires that the representation \eqref{1.3} is valid for $x$ outside the support of $f$.
\end{remark}
 
%\begin{definition}\label{DefT}
%A linear operator $T$ is called a $\phi$-fractional integral with kernel $K$, if $K$ is a $\phi$-fractional integral kernel and
%%satisfying (\ref{1.1}) and (\ref{1.2}) if
%for all $f \in L^{1}(X)$ with bounded support and all $x \notin \supp f$,
%\begin{equation}\label{1.3}
%Tf(x)=\int_{X} K(x,y)f(y)d\mu(y). 	
%\end{equation}
%\end{definition}

%Throughout this paper, we assume that $\mu(X)=\infty$. Otherwise, the non-degenerate kernel condition implies that $\mu(\{x_0\})=0$ for any $x_0 \in X$.
%
%\begin{remark}
%1. The assumption ``$\mu(X)=\infty$'' is only used to make sure that even for a big enough constant $A$, there exists a small enough constant $\delta_A$  such that    
%for every ball $B=B(x_0,r)$, there is a ball $\widetilde{B}=B(y_0,r)$ with $  \rho(x_0,y_0)\ge  Ar$ in Lemma \ref{P1}.
%
%2. If $K_{\ep}$ is a non-degenerate fractional integral kernel, then for any $x\in X$,
%\begin{align*}
%	T_{\ep}f(x)
%	 &= \int_{X\setminus \{x\}} K_{\ep}(x,y)f(y)d\mu(y)+K_{\ep}(x,x)\mu(\{x\})f(x),
%\end{align*}
%and the latter term vanishes due to $\mu(\{x\})=0$.
%\end{remark}

Our main object of study is {\em commutators} associated to the fractional integral $T$:

%And the definition of the fractional commutator $[b,T]$ will be  verified by Proposition \ref{P11.8} in Section 3.  

\begin{definition}\label{DfC}
Let $b \in L_{\mathrm{loc}}^1(X)$.
%Suppose that $K$ is a fractional integral kernel as in Definition \ref{D1}.
Let $T$ be a $\phi$-fractional integral operator as in Definition \ref{DefT}, with $\phi$-fractional integral kernel.
The commutator $[b,T]$ is defined by
\begin{equation*}
[b,T]f(x) := \int_X (b(x)-b(y))K(x,y)f(y)d\mu(y)
\end{equation*}
for all $f\in L^1_{\operatorname{loc}}(X)$ and $x\in X$ such that the integral is well defined.
%We define the corresponding fractional commutator $[b,T]$ associated to the kernel $K$, 
%$$[b,T]f(x) := \int_X (b(x)-b(y))K(x,y)f(y)d\mu(y),$$ 
%for all $f \in L^{p}(X)$ (where $1\leq p<\infty)$ and $x \notin \supp f$.
\end{definition}

\begin{lemma}\label{lem:[b,T]}
\mbox{}
\begin{enumerate}[\rm(i)]
  \item\label{it:both} If both $Tf(x)$ and $T(bf)(x)$ are well defined, then so is $[b,T]f(x)$‚ and we have
\begin{equation*}
    [b,T]f(x)=b(x)Tf(x)-T(bf)(x).
\end{equation*}
  \item\label{it:fBd} Case \eqref{it:both} holds in particular if $f\in L^\infty(X)$ is boundedly supported.
  \item\label{it:f,gBd} If both $f,g\in L^\infty(X)$ are boundedly supported, then $\langle [b,T]f,g\rangle=\langle Tf,bg\rangle-\langle T(bf),g\rangle$ is well defined.
\end{enumerate}
\end{lemma}

\begin{proof}
\eqref{it:both}  is immediate from the definitions.

\eqref{it:fBd}: Under this assumption both $f,bf\in L^1(X)$ are boundedly supported, and the claim follows from Lemma \ref{lem:Tf}\eqref{it:Tf(x)}.

\eqref{it:f,gBd}: This follows from Lemma \ref{lem:Tf}\eqref{it:Tf,g} applied to both boundedly supported pairs of functions $(f,bg)\in L^\infty(X)\times L^1(X)$ and $(bf,g)\in L^1(X)\times L^\infty(X)$ in place of $(f,g)$.
\end{proof}

In particular, Lemma \ref{lem:[b,T]} shows that, under the very general assumption $b\in L^1_{\operatorname{loc}}(X)$, the commutator $[b,T]$ is well defined on a class of test functions that is dense in $L^2(X)$. Hence, the question of extending $[b,T]$ to a bounded operator on $L^2(X)$ is equivalent to estimates on this dense test class. We are mainly interested in the stronger property that $[b,T]$ is not only bounded on $L^2(X)$ but belongs to the Schatten class $S^p=S^p(L^2(X))$ of certain compact operators on $L^2(X)$. Our main results will provide sufficient and necessary conditions for this in the following sense:
\begin{enumerate}[\rm(i)]
  \item If $b$ belongs to a suitable subclass of $L^1_{\operatorname{loc}}(X)$, then $[b,T]f(x)$ is well-defined for all $f\in L^2(X)$ and a.e.\ $x\in X$, and the operator $[b,T]$ thus defined belongs to $S^p$.
  \item If $b\in L^1_{\operatorname{loc}}(X)$ and the operator $[b,T]$, first defined on boundedly supported $f\in L^\infty(X)$ only, has an extension to a bounded linear operator on $L^2(X)$ of class $S^p$, then $b$ belongs to a suitable subclass of $L^1_{\operatorname{loc}}(X)$.
\end{enumerate}

%
%
%Let $(H,K,T)$ be a triple where $K$ is a fractional kernel, $H$ is an upper bound of $K$, and $T$ is the associated fractional integral operator. We now distinguish two cases based on the form of $H$, as given in \eqref{DH}. If $H(x,y)=\frac{1}{V(x,y)^{1-\ep}}$, then we call the fractional integral ``volumic'' and denote its kernel by  $K_{\ep}(x,y):=K(x,y)$ and the corresponding operator by $T_{\ep}:=T$.
% If $H(x,y)=\frac{\rho(x,y)^{\alpha}}{V(x,y)}$, then we call the the fractional integral ``metric'' and denote its kernel by $\widetilde{K_{\alpha}}(x,y):=K(x,y)$ and the corresponding operator by $\widetilde{T_{\alpha}}:=T$.
% 

\subsection{Fractional Sobolev norms and the Poincar\'e inequality}

Suppose that $1<p<\infty$ and $0<s<1$. The classical fractional Sobolev space $\dot{B}_{p,p}^s(\R^d)$ is defined as all locally integrable function $b$ on the Euclidean space $\R^d$ such that 
\begin{equation}\label{bRn}
	\|b\|_{\dot{B}_{p,p}^s(\R^d)}=\left(\int_{\R^d} \int_{\R^d} \frac{|b(x)-b(y)|^p}{|x-y|^{d+sp}} dx\,dy\right)^{\frac{1}{p}}<\infty.
\end{equation}
For many purposes (see e.g.~\cite{GKS}), its relevant extension to spaces of homogeneous type is defined by
\begin{equation}\label{b1}
	\|b\|_{\dot{B}_{p,p}^s(\mu)}=\left(\int_X \int_X \left(\frac{|b(x)-b(y)|}{\rho(x,y)^s}\right)^p\frac{d\mu(x)d\mu(y)}{V(x,y)}\right)^{\frac{1}{p}}<\infty,
\end{equation}
where the integrand involves both $\rho$ and $V$, i.e., the factor $|x-y|^d$ in \eqref{bRn} is interpreted as a volume $V(x,y)$, but the factor $|x-y|^{sp}$ as a distance $\rho(x,y)^{sp}$. 

 However, for the study of $\phi$-fractional integrals $T$, the following variant seems more natural, and this will be confirmed by its appearance in the characterizing conditions of Theorem \ref{T} below. For $p\in(1,\infty)$ and $\phi$ as in \eqref{eq:phixy}, let
 \begin{equation}\label{eq:newBp}
  \|b\|_{\B_{p}(\phi,\mu)}:=\left(\int_X \int_X \frac{|b(x)-b(y)|^p}{V(x,y)^2}\phi(x,y)^p d\mu(x)d\mu(y)\right)^{\frac{1}{p}}.
\end{equation}
When $\phi=\rho^\alpha$, we also denote
\begin{equation}\label{b2 variant}
\|b\|_{\widetilde{B}_{p}^{\alpha}(\mu)}:=\|b\|_{\B_{p}(\rho^\alpha,\mu)}
:=\left(\int_X \int_X \frac{|b(x)-b(y)|^p}{V(x,y)^2} \cdot \rho(x,y)^{p\alpha} d\mu(x)d\mu(y)\right)^{\frac{1}{p}}.
\end{equation} 
 However, in the volumic case, we adopt a different normalisation
 \begin{equation}\label{b2}
\|b\|_{\dot{B}_{p}^\eps(\mu)}:=\|b\|_{\B_{p}(V^{\frac1p-\eps},\mu)}
:=\left(\int_X \int_X \frac{|b(x)-b(y)|^p}{V(x,y)^{1+p\ep}}d\mu(x)d\mu(y)\right)^{\frac{1}{p}},	
%=\left(\int_X \int_X \frac{|b(x)-b(y)|^p}{V(x,y)^{2-p\ep}}d\mu(x)d\mu(y)\right)^{\frac{1}{p}}.	
\end{equation}
noting that $V^{p(\frac1p-\ep)-2}=V^{1-\ep p-2}=V^{-1-\ep p}$. The motivation of this normalisation is that the parameter $\ep$ in $\dot{B}_{p}^\eps(\mu)$ plays a similar role as the classical smoothness parameter $s$ in \eqref{b1}. Notably, if $(X,\rho,\mu)$ is Ahlfors $d$-regular, then
\begin{equation*}
  \dot B^s_{p,p}(\mu)=\dot B^{\frac sd}_p(\mu)=\widetilde B^{\frac d p-s}_p(\mu)\quad \text{if}\quad V\sim \rho^d.
\end{equation*}
With this normalisation, the volumic Besov space that appears in our results about $S^p$ properties of commutators will be 
 \begin{equation*}
  \|b\|_{\dot{B}_{p}^{\frac1p-\eps}(\mu)}:=\|b\|_{\B_{p}(V^{\eps},\mu)}
:=\left(\int_X \int_X \frac{|b(x)-b(y)|^p}{V(x,y)^2} V(x,y)^{\ep p}d\mu(x)d\mu(y)\right)^{\frac{1}{p}};	
\end{equation*}
except for the dimensional factor $d$, this is similar to the form of the classical results as in \eqref{eqJP}.
 
%  $T_{\ep}$ and $\widetilde{T_{\alpha}}$, the following variant seems more natural, and this will be confirmed by its appearance in the characterizing conditions of Theorem \ref{T} below.
%For $1<p<\infty$ and $\ep,\alpha >0$, let
%\begin{equation}\label{b2}
%\|b\|_{\dot{B}_{p}^\eps(\mu)}
%:=\left(\int_X \int_X \frac{|b(x)-b(y)|^p}{V(x,y)^{1+p\ep}}d\mu(x)d\mu(y)\right)^{\frac{1}{p}};	
%%=\left(\int_X \int_X \frac{|b(x)-b(y)|^p}{V(x,y)^{2-p\ep}}d\mu(x)d\mu(y)\right)^{\frac{1}{p}}.	
%\end{equation}
%and  
%\begin{equation}\label{b2 variant}
%\|b\|_{\widetilde{B}_{p}^{\alpha}(\mu)}
%:=\left(\int_X \int_X \frac{|b(x)-b(y)|^p}{V(x,y)^2} \cdot \rho(x,y)^{p\alpha} d\mu(x)d\mu(y)\right)^{\frac{1}{p}};
%\end{equation}
%hence the distance factor $\rho(x,y)^s$ in \eqref{b1} is replaced by the volume factor $V(x,y)^\eps$ and the mixed factor $V(x,y)^{\frac{1}{p}}\cdot \rho(x,y)^{-\alpha} $ in \eqref{b2} and \eqref{b2 variant}, respectively.
%If $(X,\rho,\mu)$ is Ahlfors $d$-regular, thus $V(x,y)\sim \rho(x,y)^d$,  then for $\alpha=d\eps$ and $s=\frac{d}{p}-\alpha$, we have
%$$\|b\|_{\dot{B}_{p,p}^{s}(\mu)}
%\sim
%\|b\|_{\dot{B}_{p}^{\p-\eps}(\mu)}
%\sim 
%\|b\|_{\widetilde{B}_{p}^{\alpha}(\mu)}.
%$$

A similar space also featured in analogous results for singular integrals in \cite{H2}. The space denoted by $\dot B_p(\mu)$ in \cite{H2} corresponds to $\dot B_p^{\frac1p}(\mu)=\tilde B^0_p(\mu)=\B_p(1,\mu)$ in the present notation, taking $\phi\equiv 1$ in \eqref{eq:newBp}.

%Definition (\ref{b2}) is justified by the result of Proposition \ref{P3} and \ref{P2}: it is under this definition that the characterisation works in general. In most results, we will need to restrict the range of $\ep$.   
%Let us observe two important cases:
%
% $\mathrm{(\romannumeral1)}$ $1<p<\infty$, $(\p-\frac{1}{d} )_+<\ep<\p$:
%Especially, if $(X,\rho,\mu)$ is Ahlfors $d$-regular space ($d>1$), by taking $s= d(\frac{1}{p}-\ep) $ in (\ref{b1}) and using the inequality $V(x,y)\thickapprox \rho(x,y)^d$,  it is obvious that 
%$$\|b\|_{\dot{B}_{p}^{\ep}(\mu)}
%\thickapprox
%\|b\|_{\dot{B}_{p,p}^{d(\frac{1}{p}-\ep)}(\mu)} , $$ 
%
% $\mathrm{(\romannumeral2)}$ $1<p<d$, $0<\ep\leq \frac{1}{p}-\frac{1}{d}$: 
% Especially, our result of Theorem \ref{T} (2) shows that if the space $(X,\rho,\mu)$ is a space of homogeneous type with lower dimension $d$ and satisfies the $(1,\frac{d}{1+d\ep})$-Poincar\'e inequality, then 
% $$\dot{B}_{p}^{\ep}(\mu) = \{\mathrm{constants}\}.$$

We also recall the Poincar\'e inequality, which plays a significant role in several aspect of analysis on metric spaces (see \cite{Heinonen}), and our main result below is no exception.

\begin{definition}\label{def:Poincare}
Let $s>1$.
A space $(X,\rho,\mu)$ is said to satisfy the $(1,s)$-Poincar\'e inequality if $\rho$ is a metric (i.e., $A_0=1$ in \eqref{quasitri}), and 
 there exists $\lambda \ge 1$ and $c_P$ such that for every Lipschitz function $f$ on $X$, every  $x\in X$ and $r>0$,
\begin{equation}\label{PI}
	\fint_{B(x,r)}|f-\langle f \rangle_{B(x,r)} | d\mu \leq c_P \cdot r \cdot \left( \fint_{B(x,\lambda r)}(\mathrm{lip} f)^{s} d\mu \right)^{1/s}, 
\end{equation}
where the pointwise Lipschitz constant $\mathrm{lip}f$ is defined as   
\begin{equation}
\mathrm{lip}f(x):=\underset{r\rightarrow 0}{\lim \inf}\,\,\,\underset{\rho(x,y)\leq r }{\sup}\frac{|f(x)-f(y)|}{r}.	
\end{equation}
\end{definition}

\begin{remark}\label{R1.8}
On metric spaces, the abundance of Lipschitz functions makes the Poincar\'e inequality a useful and non-trivial condition. 
To accommodate quasi-metric spaces, we incorporate the assumption that $\rho$ is a metric into the definition of the Poincar\'e inequality. This  allows us to use the phrase ``let $X$ satisfy the Poincar\'e inequality''  as a shorthand for ``let $X$ be a metric space that satisfies the Poincar\'e inequality''.

For any $s>1$, if $X$ satisfies the $(1,s)$-Poincar\'e inequality, then $X$ satisfies the $(1,t)$-Poincar\'e inequality for every $t\ge s$. This is immediate from H\"older's inequality.

If $X$ is a complete doubling space, then it also satisfies 	the $(1,t)$-Poincar\'e inequality for {\em some} $t<s$. This is a deeper theorem from \cite{KZ}; we will not need it in the present work.
\end{remark}

\subsection{The main results}

We are now ready to state our main result. We will write simply $S^p:=S^p(L^2(\mu))$, where the space $L^2(\mu)$ is understood from the context.

\begin{theorem}\label{T}
Let $(X,\rho,\mu)$ be a space of homogeneous type with a lower dimension $d>0$.
Let $\phi$ be as in \eqref{eq:phixy} with parameter $\ep\in(0,1)$ or $\alpha\in(0,\infty)$.
Suppose that $T$ is a $\phi$-fractional integral. %, whose modulus of continuity $\omega$ satisfies $L^1$-Dini condition.
Then the following conclusions hold for all $b \in L_{\mathrm{loc}}^1(X)$:
\begin{enumerate}[\rm(1)]

\item \label{T:p>2}If $p\in [2,\infty)$, $\ep \in (0,1)$, and $\alpha\in (0,\infty)$, then 
$$\Norm{[b,T_\ep]}_{S^p}\lesssim\Norm{b}_{\dot{B}_p^{\frac1p-\ep}(\mu)}
\qquad\mathrm{and}\qquad
\Norm{[b,\widetilde{T_{\alpha}}]}_{S^p}\lesssim\Norm{b}_{\widetilde{B}_{p}^{\alpha}(\mu)}.$$
 
\item \label{T:p<1} If $p\in (1,\infty)$, $\ep \in (0,\frac{1}{p})$, $\alpha=d \ep$, and $K$ is strongly non-degenerate (Definition \ref{D1'}), then 
$$\Norm{b}_{\dot{B}_p^{\frac1p-\ep}(\mu)}\lesssim\Norm{[b,T_\ep]}_{S^p}
\qquad\mathrm{and}\qquad
\Norm{b}_{\widetilde{B}_{p}^{\alpha}(\mu)} \lesssim \Norm{[b,\widetilde{T_{\alpha}}]}_{S^p} .$$

\item \label{T:1<p<d} If $d\in (1,\infty)$, $\ep \in (0,1-\frac{1}{d})$, $\alpha=d\ep $, and $X$ satisfies the $(1,\frac{d}{1+d\ep})$-Poincar\'e inequality,
then 
\begin{equation}\label{4.4}
	\dot{B}_{\frac{d}{1+d\ep}}^{\frac1d}(\mu) =\widetilde{B}_{\frac{d}{1+\alpha}}^{\alpha}(\mu)\equiv \{\mathrm{constants}\}.
	\end{equation}                                                                                                                                                                     
\end{enumerate}
The first result \eqref{T:p>2} does not require the existence of  lower dimension $d$.
\end{theorem}

Combining the upper and lower bounds, we can further present the following corollary:

\begin{corollary}\label{C}
Let $(X,\rho,\mu)$ be a space of homogeneous type with a lower dimension $d>0$. Let $\ep \in (0,1)$ and $\alpha=d\ep$.
Suppose that $T$ is a strongly non-degenerate fractional integral operator with kernel $K$ satisfying (\ref{1.1}) through \eqref{Non-degenerate}. Then the following conclusions hold for all $b \in L_{\mathrm{loc}}^1(X)$:
\begin{enumerate}[\rm(1)]

\item\label{C:p>d} If $p\in [2,\infty) $ and $\ep\in (0,\frac{1}{p})$, then 
$$\begin{cases}
		[b,T_{\ep}]\in S^p &\iff  \,\,\, b \in \dot{B}_p^{\frac1p-\ep}(\mu),\\
		[b,\widetilde{T_{\alpha}}]\in S^p &\iff \,\,\, b \in \widetilde{B}_p^{\alpha}(\mu).
	\end{cases}$$
	
\item\label{C:p<d} If $p \in (1,2)$ and $\ep \in  (\max\{0,\frac{1}{p}-\frac{1}{d}\},\p )$, then
$$\begin{cases}
		[b,T_{\ep}]\in S^p &\Longrightarrow  \,\,\, b \in \dot{B}_p^{\frac1p-\ep}(\mu),\\
		[b,\widetilde{T_{\alpha}}]\in S^p &\Longrightarrow \,\,\, b \in \widetilde{B}_p^{\alpha}(\mu).
	\end{cases}$$

\item\label{C:const} If $p\in(0,d)$ and $\ep \in (0,1-\frac{1}{d}) \cap (0,\frac{1}{p}-\frac{1}{d}]$ and $X$ satisfies the $(1,\frac{d}{1+d\ep})$-Poincar\'e inequality,
then $[b,T_{\ep}]\in S^p$ or $[b,\widetilde{T_{\alpha}}]\in S^p$ if and only if $b$ is constant.

\end{enumerate}
\end{corollary}
%The different parameter ranges of Corollary \ref{C} are illustrated in Figure \ref{fig}.

%In order to show the parametric relationship in Corollary \ref{C} more clearly, we make the following picture, where the red symbols (1)-(3)  represent the parametric areas of Corollary \ref{C} (1)-(3), respectively and the exponent $\ep$ will not take values on all blue lines of the following picture.
\begin{figure}[H]
\centering

% ---- 左图 ----
\begin{minipage}[t]{0.47\textwidth}
\centering
\begin{tikzpicture}[scale=2.5]

% Axes
\draw[->, thick] (-0.35,0) -- (1.3,0) node[right] {${\tfrac{1}{p}}$};
\draw[->, thick] (0,-0.8) -- (0,1.3) node[above] {${\ep}$};

% Main diagonal blue line

% Black lines
\draw[thick] (0,0) -- (1,1) ;
\draw[thick] (0.7,0) -- (1,0.3) ;

% Dashed lines
\draw[dashed] (0,1) -- (1,1);
\draw[dashed] (0,0.5) -- (0.5,0.5);
\draw[dashed] (0,0.3) -- (1,0.3);
\draw[dashed] (0,0) -- (-0.3,-0.3);
\draw[dashed] (1,1) -- (1.2,1.2) node[right] {${\ep = \tfrac{1}{p}}$};
\draw[dashed] (1,0.3) -- (1.3,0.6) node[right] {$\ep = \tfrac{1}{p} - \tfrac{1}{d}$};
\draw[dashed] (-0.1,-0.8) -- (0.7,0);

% Axis ticks
\draw[thick] (0.7,0) -- (0.7,0.03) node[below=3pt] {$\tfrac{1}{d}$};
\draw[thick] (0.5,0) -- (0.5,0.03) node[below=3pt] {$\tfrac{1}{2}$};
\draw[thick] (1,0) -- (1,0.03) node[below=3pt] {1};
\draw[thick] (0,0.5) -- (0.03,0.5) node[left=7pt] {$\tfrac{1}{2}$};
\draw[thick] (0,-0.7) -- (0.02,-0.7); % node[right=-1pt] {$-\tfrac{1}{d}$};
\draw[thick] (0,0.3) -- (0.03,0.3) node[left=2pt] {$1-\tfrac{1}{d}$};
\draw[thick] (0,1) -- (0.03,1) node[left=3pt] {1};

% Red section labels (use plain text here to avoid eqref issues)
\node at (0.3,0.13) {\footnotesize{\eqref{C:p>d}}}; % {\footnotesize{(1)}}
\node at (0.71,0.4) {\footnotesize{\eqref{C:p<d}}};
\node at (1.1,0.13) {\footnotesize{\eqref{C:const}}};

% Blue step lines
\draw[black, thick] (1,0.3) -- (1,1);
\draw[black, thick] (1,0.3) -- (1.3,0.3);
\draw[black, thick] (0.5,0.5) -- (0.5,0);
\end{tikzpicture}

\caption*{\textbf{(a)} Case: $0<d\leq 2$}
\end{minipage}
\hfill
% ---- 右图 ----
\begin{minipage}[t]{0.47\textwidth}
\centering
\begin{tikzpicture}[scale=2.5]

% Axes
\draw[->, thick] (-0.35,0) -- (1.3,0) node[right] {${\tfrac{1}{p}}$};
\draw[->, thick] (0,-0.8) -- (0,1.3) node[above] {${\ep}$};

% Black lines
\draw[dashed] (-0.3,-0.3) -- (1.2,1.2) node[right] {${\ep = \tfrac{1}{p}}$};
\draw[dashed] (-0.2,-0.5) -- (1.2,0.9) node[right] {$\ep = \tfrac{1}{p} - \tfrac{1}{d}$};

% Dashed lines
\draw[dashed] (0,1) -- (1,1);
\draw[dashed] (0,0.5) -- (0.5,0.5);
\draw[dashed] (0,0.7) -- (1,0.7);
\draw[dashed] (0,0.2) -- (1.3,0.2);

% Axis ticks
\draw[thick] (0.3,0) -- (0.3,0.03) node[below=3pt] {$\tfrac{1}{d}$};
\draw[thick] (0.5,0) -- (0.5,0.03) node[below=3pt] {$\tfrac{1}{2}$};
\draw[thick] (1,0) -- (1,0.03) node[below=3pt] {1};
\draw[thick] (0,0.5) -- (0.03,0.5) node[left=7pt] {$\tfrac{1}{2}$};
\draw[thick] (0,0.2) -- (0.03,0.2) node[left=1pt] {$\tfrac{1}{2}-\tfrac{1}{d}$};
\draw[thick] (0,-0.3) -- (0.02,-0.3); % node[right=-1pt] {$-\tfrac{1}{d}$};
\draw[thick] (0,0.7) -- (0.03,0.7) node[left=2pt] {$1-\tfrac{1}{d}$};
\draw[thick] (0,1) -- (0.03,1) node[left=3pt] {1};
\draw[thick] (0.3,0) -- (1,0.7);
\draw[thick] (1,0.7) -- (1.3,0.7);

% Red section labels
\node at (0.28,0.12) {\footnotesize{\eqref{C:p>d}}}; % {\footnotesize{(1)}}
\node at (0.7,0.54) {\footnotesize{\eqref{C:p<d}}};
\node at (1,0.3) {\footnotesize{\eqref{C:const}}};

% Blue step lines
\draw[thick] (1,0.7) -- (1,1);
\draw[black, thick] (0.5,0.5) -- (0.5,0.2);
\draw[thick] (0,0) -- (1,1);
\end{tikzpicture}

\caption*{\textbf{(b)} Case: $d> 2$}
\end{minipage}

% ---- 总 caption ----
\caption{
The different parameter ranges for $(\ep,p,d)$ in Corollary \ref{C}.
In cases \eqref{C:p>d} and \eqref{C:const}, we have a characterization,
and in \eqref{C:p<d}, a necessary condition for
$[b,T_{\ep}]\in S^p$ and $[b,\widetilde{T_{\alpha}}]\in S^p$
(with $\alpha=d\ep$). The region below the dashed line $\ep = \tfrac{1}{2} - \tfrac{1}{d}$ in case $d> 2$ corresponds to the parameter ranges in Corollary \ref{C'}, where only cases \eqref{C:p>d} and \eqref{C:const} appear.
}
\label{fig'}
\end{figure}

\begin{remark}\label{R1.14}
We point out that the method of estimating the upper Schatten bounds of the fractional commutator $[b, T] $ in Corollary \ref{C1} is limited to the parameter $p\in [2,\infty)$. The reverse of Corollary \ref{C} \eqref{C:p<d}  for $p\in (1,2)$ will be addressed in a forthcoming work of the first author with L. Zacchini.
\end{remark}

%\begin{remark}
%Here we account for some extra restrictions on $\ep$. A step of proving Theorem \ref{T} (2) is Proposition \ref{P3.11}, which requires that $\ep \in (0,\frac{1}{p})$ to make sure the second sum in (\ref{eq3.14}) converges. That is why $\ep$ is smaller than $\p$ in Theorem \ref{T} (2). The other restriction: $p\in(0,d)$ and $\ep \in (0,1-\frac{1}{d}) \cap (0,\frac{1}{p}-\frac{1}{d}]$, is to ensure that $\frac{d}{1+d\ep}>1 $ meets the Poincar\'e inequality's exponent requirement. 
%\end{remark}

\section{Preliminaries}\label{sec:prelim}

This section has two subsections, \ref{sec:lemmas} on basic lemmas related to fractional integrals, and \ref{sec:dyadic} on dyadic cubes.

\subsection{Basic lemmas about fractional integrals}\label{sec:lemmas}

Here, we provide the proofs of the lemmas stated in Section \ref{sec:def}, plus some additional ones to make those proofs more streamlined.

We begin with the sufficient condition for strong non-degeneracy stated in Lemma \ref{lem:K>0}:

\begin{proof}[Proof of Lemma \ref{lem:K>0}]
 The non-empty annulus assumption implies that, for every point $x_0\in X$, radius $r>0$, and parameter $A>0$ yet to be chosen, there is a point $y_0\in B(x_0,\overline CAr)\setminus B(x,Ar)$. We will show that any such point satisfies properties \eqref{eq:sndg1} and \eqref{eq:sndg2} provided that $A>0$ is large enough (independently of $x_0$ and $r$).
 
Assumption \eqref{eq:Ksim} implies in particular that $K(x,y)>0$ for all $x\neq y$, and hence the bound concerning the argument in \eqref{eq:sndg1} and \eqref{eq:sndg2} is trivial with $v=1$ and $\eta=0$. As for the size bounds in \eqref{eq:sndg1} and \eqref{eq:sndg2}, we note that both $\rho(x,y)\sim\rho(x_0,y_0)\approx Ar$ and $V(x,y)\sim V(x_0,y_0)\sim V(x_0,Ar)$ for all $x,x_0,y,y_0$ as in Definition \ref{D1'} of strong non-degeneracy, as soon as $A$ is large enough. Fixing such an $A$, we then have $Ar\sim r$ and $V(x_0,Ar)\approx V(x_0,r)$, which implies that
\begin{equation*}
  K(x,y)\sim\frac{\phi(x,y)}{V(x,y)}\sim\frac{\phi(x_0,y_0)}{V(x_0,y_0)}\sim\frac{\phi(x_0,r)}{V(x_0,r)},
\end{equation*}
and the case of $K(y,x)$ is entirely analogous.
\end{proof}

To support the results about the well-definedness of $Tf$ and $[b,T]f$, we first give:

\begin{lemma}\label{lem:intK}
 Let $\phi$ and $K$ satisfy \eqref{eq:phixy} and \eqref{1.1}.
If $f\in L^1_{\operatorname{loc}}(X)$, then
\begin{equation}\label{eq:intK<phi}
  \int_{B(x,r)}|K(x,y) f(y)|d\mu(y)\lesssim\phi(x,r)Mf(x),
\end{equation}
where $M$ is the Hardy--Littlewood maximal operator.
 \end{lemma}
 
\begin{proof}
In the metric case $\phi(x,y)=\rho(x,y)^\alpha$, we have
\begin{equation*}
\begin{split}
   \int_{B(x,r)}|K(x,y) f(y)|d\mu(y)
   &\sim \sum_{k=0}^\infty \int_{B(x,2^{-k}r)\setminus B(x,2^{-k-1}r)}\frac{ (2^{-k} r)^\alpha}{V(x,2^{-k} r)}|f(y)|d\mu(y) \\
   &\leq  \sum_{k=0}^\infty (2^{-k} r)^\alpha Mf(x) \sim r^\alpha Mf(x).
\end{split}
\end{equation*}
In the volumic case $\phi(x,y)=V(x,y)^\eps$, we choose a decreasing sequence $(r_k)_{k=0}^K$, where $K\in\N\cup\{\infty\}$, recursively as follows: Let $r_0:=r$. Given $r_k$, we look for a $j\in\N$ such that $V(x,2^{-j}r_k)<\frac12 V(x,r_k)$. If no such $j\in\N$ exists, it means that $V(x,2^{-j}r_k)\geq\frac12 V(x,r_k)$ for all $j\in\N$, and hence $\mu(\{x\})=\lim_{k\to\infty}V(x,2^{-j}r_k)\geq\frac12 V(x,r_k)$. In this case, the process stops at this finite $k=:K$; note that this only happens if $\mu(\{x\})>0$. Otherwise, we let $r_{k+1}:=2^{-j}r_k$. Then, by construction,
\begin{equation*}
 V(x,r_{k+1})<\frac12 V(x,r_k)
 \leq V(x,2r_{k+1})\leq CV(x,r_{k+1})
\end{equation*}
for all $k<K$, and $\mu( \{ x\} )\leq V(x,r_K)\leq 2 \mu(\{ x\})$. Thus we have the following, where the term involving $B(x,r_K)$ is omitted if $K=\infty$:
\begin{equation*}
\begin{split}
  &\int_{B(x,r)} |K(x,y) f(y)|d\mu(y) \\
  &=\sum_{k=0}^{K-1}\int_{B(x,r_k)\setminus B(x,r_{k+1})} V(x,y)^{\eps-1}d\mu(y)+\int_{B(x,r_K)}V(x,y)^{\eps-1} |f(y)|d\mu(y) \\
  &\lesssim\sum_{k=0}^{K-1} V(x,r_{k+1})^{\eps-1}\int_{B(x,r_k)}|f(y)|d\mu(y)+\mu(\{x\})^{\eps-1}\int_{B(x,r_K)}|f(y)|d\mu(y) \\
  &\sim\Big(\sum_{k=0}^{K-1} V(x,r_k)^{\eps}+V(x,r_K)^{\eps}\Big)Mf(x)
  \leq\sum_{k=0}^K (2^{-k}V(x,r))^{\eps}Mf(x)\sim V(x,r)^\eps Mf(x).
\end{split}
\end{equation*}
Thus, in both cases, we obtain \eqref{eq:intK<phi}.
\end{proof}

We can now provide the proof of Lemma \ref{lem:Tf} that we already stated in Section \ref{sec:def}.

\begin{proof}[Proof of Lemma \ref{lem:Tf}]
\eqref{it:Tf(x)}: Given $x\in X$, the assumption implies that $f$ is supported in $B(x,r)$ for some $r$. Hence the integral in \eqref{1.3} can be restricted to $B(x,r)$, and Lemma \ref{lem:intK} guarantees that this integral exists provided that $Mf(x)<\infty$. By the weak $(1,1)$ inequality of $M$, this happens at almost every $x\in X$.

\eqref{it:Tf,g}: We can choose some $B=B(x_0,r)$ such that both $f$ and $g$ are supported in $B$. Moreover, if $x,y\in B(x_0,r)$, then $y\in B(x,2A_0 r)$. 
It then follows from Lemma \ref{lem:intK} that $|Tf(x)g(x)|\lesssim\phi(x,2A_0 r) Mf(x) |g(x)|\lesssim\phi(x_0,r)Mf(x)|g(x)|$. If $p>1$, then $Mf\in L^p(X)$, and hence this product is integrable. On the other hand, writing out the double integral defining $\langle Tf,g\rangle$, we also have the alternative bound
\begin{equation*}
\begin{split}
  \int_X |f(y)|\int_{B(y,2A_0 r)}|K(x,y)| |g(x)|d\mu(x) d\mu(y)
  &\lesssim\int_X |f(y)| \phi(y,2A_0 r) Mg(y) d\mu(y) \\
  &\lesssim \phi(x_0,r)\int_X |f(y)|  Mg(y) d\mu(y),
\end{split}
\end{equation*}
noting that the conditions \eqref{eq:phixy} and \eqref{1.1} are, up to constants, symmetric in $x$ and $y$, so that Lemma \ref{lem:intK} also applies with the roles of these variables interchanged. If $p'>1$, then $Mg\in L^{p'}(X)$, and the product above is integrable. Since at least one of $p,p'$ is greater than~$1$, we can always apply at least one of these alternative bounds, both of which lead to the same result \eqref{eq:Tf,g}.
\end{proof}

\subsection{Dyadic cubes}\label{sec:dyadic}

%We begin this section by recalling the dyadic systems.
Some of our results will make use of the notion of systems of dyadic cubes in a space of homogeneous type $(X,\rho,\mu)$.
Recall that the standard system of dyadic cubes on the Euclidean space $\mathbb{R}^d$ is defined as
$$\mathscr{D}:=\{2^{-k}([0,1)^d+m):k\in \mathbb{Z},m\in \mathbb{Z}^d\}.$$
The fundamental properties of these cubes are that any two of them are either disjoint or one is contained in the
other, and that the cubes of a given size partition all space. As for general spaces of homogeneous type, a more general construction was first provided by Christ \cite{C} and elaborated by Hyt\"{o}nen and Kairema \cite{HKa}, as follows.
\begin{definition}\label{De2.4}
A system of dyadic cubes $\mathscr{D}$, on the space of homogeneous type $(X,\rho,\mu)$ is a collection
 $$\mathscr{D}=\bigcup_{k\in \mathbb{Z}}\mathscr{D}_k,$$
 where	
 
 (1) for each $k\in \mathbb{Z}$, there is a disjoint union $X=\bigcup_{Q \in \mathscr{D}_k}Q$;
 
 (2) each $\mathscr{D}_{k+1}$ refines the previous $\mathscr{D}_k$;
 
 (3) for parameters $\delta \in (0,1)$ and $0<\widetilde{c_0}\leq \widetilde{C_0} <\infty$, each $Q\in \mathscr{D}_k$ is essentially a ball of size ${\delta}^k$, in the sense that, for some ``centre'' $z_Q \in X$,
 \begin{equation}\label{dyadic system}
 B(z_Q, \widetilde{c_0}\delta^k)\subseteq Q \subseteq B(z_Q, \widetilde{C_0}\delta^k):=B_Q. 
 \end{equation}
% The dyadic system is said to be connected if, in addition.
% 
% (4) for any cubes $P,Q\in \mathscr{D}$, there is some $R\in \mathscr{D}$ such that $P\cup Q \subseteq R$.  
\end{definition}
%There is a narural conclusion about dyadic cubes: there exists a constant $C\ge 1$ such that for each dyadic cube $Q\in \mathscr{D}$, there is a ball $B_Q \supseteq Q$ satisfying $\mu(B_Q)\leq C\mu(Q)$, where the ball $B_Q$ is consistent with the definition given in (\ref{dyadic system}).  

We denote by $l(Q):=\delta^k$ the ``side length'' of $Q\in\mathscr D_k$.
%The symbol \(l(Q)\) denotes the side length of the cube \(Q\).

Next, we will use the following notations from \cite{H2} to describe the relationship of different levels of dyadic cubes in Section 4. 
For $Q\in \mathscr{D}$, we denote by $Q^{[1]}$
 the minimal $R \in \mathscr{D}$ such that $Q\subsetneq R$. We refer to $Q^{[1]}$ as the strict parent of $Q$ and $Q^{[1]}$ exists unless $Q=X$. Denote $Q^{[0]}:=Q$ and $Q^{[j]}:=(Q^{[j-1]})^{[1]}$ for each $j\in \mathbb{Z}$.
By the equivalent size between one cube and its strict parent in Lemma 6.7, \cite{H2}: there exists constants $1<c\leq C<\infty$ such that for all $Q\in \mathscr{D}\setminus \{X\} $,
$$c\mu(Q)\leq \mu(Q^{[1]})\leq C\mu(Q).$$
By iterating the above estimate, it implies the following useful result: 
For every $\gamma >0$, 
\begin{equation}\label{eq3.15}
\sum_{k=0}^{\infty}\frac{1}{\mu(Q^{[k]})^{\gamma}}
\leq \sum_{k=0}^{\infty}\frac{1}{c^{\gamma k}\mu(Q)^{\gamma}}=\frac{c^{\gamma}}{c^{\gamma}-1}\frac{1}{\mu(Q)^{\gamma}}.
\end{equation}
If the space of homogeneous type $(X,\rho,\mu)$ has a lower dimension $d>0$, 
then 
\begin{equation*}
l(Q^{[k]}) \sim \delta^{-k} l(Q),
\,\,\, \mu(Q^{[k]})\gtrsim \delta^{-kd}\mu(Q),
\end{equation*}
and hence for any $s,t>0$ satisfying $s<dt$,
\begin{equation} \label{eq4.10}
\sum_{k=0}^{\infty}\frac{l(Q^{[k]})^{s}}{\mu(Q^{[k]})^{t}}
\lesssim \frac{l(Q)^{s}}{\mu(Q)^{t}}
\sum_{k=0}^{\infty} \delta^{k(dt-s )}
\lesssim \frac{l(Q)^{s}}{\mu(Q)^{t}}.
\end{equation}

%\begin{remark}
%It should be noticed that Definition \ref{De2.4} is still valid on more general metric spaces with geometric doubling property in \cite{HKa}. Meanwhile, it is well known that every space of homogeneous type satissfies the following geometric doubling property:
%there exists a positive integer $A_1\in \mathbb{N}$ such that for any $x \in X$ and any $r>0$, the ball $B(x,r):=\{y\in X: \rho(y,x)<r\}$ can be recovered by at most $A_1$ balls $B(x_i,r/2)$. Our paper just focuses on spaces of homogeneous type (see \cite{CW} for more details). That is why we give the above definition of dyadic systems only on spaces of homogeneous type.  	
%\end{remark}

\section{The upper bounds for the fractional commutators}\label{sec:upper}

For $p\in(2,\infty)$,  Hyt\"{o}nen \cite{H2} obtained $S^p$ estimates for a class of general integral operators on spaces of homogeneous type, by extending an idea from Janson and Wolff \cite{JW} on the Euclidean space $\mathbb{R}^d$. For $p=2$ (Hilbert-Schmidt operators), it is classical.
 %and refers the reader to \cite{JW} on Euclidean spaces and \cite{R} on some metric spaces.
 Based on these results, we give the following $S^p$ estimates for all $p \in [2,\infty)$.
Let $(X,\rho,\mu)$ be a $\sigma$-finite measure space.
%For $1\leq p,q<\infty$, the mixed norm  space $L^p(L^q)$ (cf. \cite{BP}) are defined by 
%$$\{f:X\times X\rightarrow \mathbb{C}\,\,\textup{is}\,\,\textup{measurable}: \|f\|_{L^p(L^q)}=\Big(\int_X \Big(\int_X |f(x,y)|^q d\mu(y)\Big)^{\frac{p}{q}}d\mu(x) \Big)^{\frac{1}{p}}<\infty\}.$$

\begin{proposition}\label{P11.8}
Let $(X,\rho,\mu)$ be a space of homogeneous type and $L$ be a kernel satisfying the size condition for any $x,y \in X$,
\begin{equation}\label{3.14}
|L(x,y)|\lesssim\frac{1}{V(x,y)}.
\end{equation}
Then for any measurable function $B$ on $X \times X$ and exponent $p\in [2,\infty)$, the integral operator $\mathcal{I}_{BL}$:
$$\mathcal{I}_{BL}f(x)=\int_X B(x,y)L(x,y)f(y)d\mu(y),\,\,\,\forall x \in X,$$
satisfies
\begin{equation}\label{upper-bound-eq1}
\| \mathcal{I}_{BL}\|_{S^p}\lesssim \|B\|_{L^p(V^{-2})}:=\left(\iint_{X\times X}\frac{|B(x,y)|^p}{V(x,y)^2}d\mu(x)d\mu(y) \right)	^{\frac{1}{p}}.
\end{equation}
\end{proposition}

\begin{proof}
For $p\in (2,\infty)$, this is \cite[Proposition 5.8]{H2}. For $p=2$,
 an explicit formula of the Hilbert-Schmidt norm of integral operators (see e.g.\ \cite[Theorem 2.11]{Simon}) and (\ref{3.14}) give that
\begin{equation*}
  	\| \mathcal{I}_{BL}\|_{S^2}=\|BL\|_{L^2}
	\lesssim \|BV^{-1}\|_{L^2}= \|B\|_{L^2(V^{-2})}. 
\end{equation*}
%\begin{equation}\label{eq3.16}
%\|BL\|_{L^2(L^2)}\lesssim \|BV^{-1}\|_{L^2(L^2)}= \|B\|_{L^2(V^{-2})}. 
%\end{equation}
%For the Hilbert-Schmidt norm of integral operators, there is the explicit formula 
%\begin{equation}\label{eq3.17}
%	\| \mathcal{I}_{BL}\|_{S^2}=\|BL\|_{L^2(L^2)}.
%\end{equation}
%Combining (\ref{eq3.16}) and (\ref{eq3.17}), we get $\| \mathcal{I}_{BL}\|_{S^2} \lesssim \|B\|_{L^2(V^{-2})} $.
Thus, we complete the proof of (\ref{3.14}).  
\end{proof}

From Proposition \ref{P11.8}, one immediately has the following two corollaries involving the Schatten class of the fractional commutators.

\begin{corollary}\label{C1}
Let $p\in [2,\infty)$, $\ep \in (0,1)$ and $\alpha \in (0,\infty)$. Let $(X,\rho,\mu)$ be a space of homogeneous type.
%Suppose that $K$ is a kernel satisfying \eqref{1.1} and $T$ is the (not necessarily bounded) integral operator with kernel $K$.
Let $T\in\{T_\ep,\widetilde T_\alpha\}$ be a $\phi$-fractional integral with kernel $K$.
Then the corresponding fractional commutator 
$$[b,T]f(x)=\int_X [b(x)-b(y)]K(x,y)f(y)d\mu(y)$$
belongs to the Schatten class $S^{p}$, provided that the respective right-hand side below is finite, and we have the estimates
\begin{equation}\label{upper-bound-eq2}
	\|[b,T_{\ep}]\|_{S^p} 
	\leq 
	\|b\|_{\dot{B}_{p}^{\p-\ep}(\mu)}=\left(\int_X \int_X 
	\frac{|b(x)-b(y)|^p}{V(x,y)^{1-p\ep}}\cdot \frac{d\mu(x)d\mu(y)}{V(x,y)}\right)^{\frac{1}{p}},
\end{equation}
and
\begin{equation}\label{upper-bound-eq2 variant}
	\|[b,\widetilde{T_{\alpha}}]\|_{S^p} 
	\leq 
	\|b\|_{\widetilde{B}_{p}^{\alpha}(\mu)}=\left(\int_X \int_X \frac{|b(x)-b(y)|^p \rho(x,y)^{p\alpha}}{V(x,y)}\cdot \frac{d\mu(x)d\mu(y)}{V(x,y)}\right)^{\frac{1}{p}}.	
\end{equation}

\end{corollary}
\begin{proof}
Taking $L(x,y)=\frac{1}{V(x,y)}$ and $B(x,y)=(b(x)-b(y)) K(x,y)\cdot V(x,y)$ in (\ref{upper-bound-eq1}), it follows from Proposition \ref{P11.8} that 
\begin{equation*}
\Norm{[b,T]}_{S^p}  = \| \mathcal{I}_{BL}\|_{S^p}
\lesssim 
\left(\iint_{X\times X}
|b(x)-b(y)|^p \frac{\phi(x,y)^p}{V(x,y)^p}V(x,y)^{p-2}d\mu(x)d\mu(y) \right)	^{\frac{1}{p}},
\end{equation*}
where
\begin{equation*}
  \frac{\phi(x,y)^p}{V(x,y)^p}V(x,y)^{p-2}
  =\frac{\phi(x,y)^p}{V(x,y)^2}
  =\begin{cases}
    V(x,y)^{p\ep-2}, & \text{if }\phi(x,y)=V(x,y)^\ep, \\
    \rho(x,y)^{p\alpha}V(x,y)^{-2}, & \text{if }\phi(x,y)=\rho(x,y)^\alpha.
  \end{cases}
\end{equation*}
Then we get (\ref{upper-bound-eq2}) and \eqref{upper-bound-eq2 variant}, as desired.
\end{proof}

\section{Characterization for fractional oscillation spaces and Besov spaces}\label{sec:Besov}

If $p\in (1,\infty)$, Hyt\"{o}nen \cite{H2}
 described Besov spaces $\dot{B}_{p}(\mu) $ by oscillation spaces $\mathrm{Osc}^{p}(\mu)$ (see \cite{H2}, Proposition 1.29), namely
 \begin{align*}
 	& \|b\|_{\dot{B}_{p}(\mu)} :=\left(\int_X \int_X \frac{|b(x)-b(y)|^p}{V(x,y)^{2}}d\mu(x)d\mu(y)\right)^{\frac{1}{p}}\\
& \thickapprox
 \|b\|_{\mathrm{Osc}^p(\mu)}:=\Big\| \Big\{  \fint_{B_Q}|b-\langle b\rangle_{B_Q} |d\mu\Big\}_{Q\in \mathscr{D}} \Big\|_{\ell^p}.
 \end{align*}
For the fractional version, an analogous is stated below in Propositions \ref{P2}. First, we give a useful characterisation from \cite[Lemma 11.4]{H2}:

\begin{lemma}[\cite{H2}, Lemma 11.4]
For any $p\in (0,\infty)$,
\begin{equation}\label{eq3.13}
\begin{aligned}
 \inf_c \Big(\fint_{B}|b- c |^p d\mu\Big)^{\frac{1}{p}}
 & \sim    \Big(\fint_{B} \fint_{B} |b(x)- b(y)|^p d\mu(x) d \mu(y)\Big)^{\frac{1}{p}}\\
 & \sim \Big(\fint_{B}|b(x)-\langle b\rangle_{B} |^p d\mu(x)\Big)^{\frac{1}{p}},\,\, \text{if} \,\,p\in [1,\infty).
\end{aligned}	
\end{equation}
\end{lemma}

Consistently with \eqref{eq:phixy}, for any cube $Q$, we denote
\begin{equation}\label{phi}
\phi(Q)\in \{\mu(B_Q)^{\ep} ,l(Q)^{\alpha}\}.	
\end{equation}
For $p\in (1,\infty)$, we denote   
\begin{equation}\label{Osc}
		 \|b\|_{{\mathrm{Osc}}_{\phi}^p(\mu)}:=\Big\| \Big\{ {\phi(Q)} \fint_{B_Q}|b-\langle b\rangle_{B_Q} |d\mu\Big\}_{Q\in \mathscr{D}} \Big\|_{\ell^p} .
		 \end{equation}

\begin{proposition}\label{P2}
Let $(X,\rho,\mu)$ be a space of homogeneous type with a lower dimension $d>0$ and $b\in L_{\mathrm{loc}}^1(\mu)$. Let $\mathscr{D}$ be a system of dyadic cubes on $(X,\rho,\mu)$ in the sense of Definition \ref{De2.4} and, for each $Q\in \mathscr{D}$, let $B_Q$ be a ball centred at $Q$ and of radius $c\cdot l(Q)$ for a constant $c$ that only depends on the space $X$.
Let $p\in(1,\infty)$, and $\phi$ be as in \eqref{eq:phixy} and \eqref{phi} with $\ep\in(0,1/p)$ and $\alpha\in(0,d/p)$. Then
\begin{equation*}
  \Norm{b}_{\B_p(\phi,\mu)}
  \sim \|b\|_{{\mathrm{Osc}}_{\phi}^p(\mu)}.
\end{equation*}
%	\begin{equation}
%		\|b\|_{\dot{B}_{p}^{\p-\ep}(\mu)}
%		\thickapprox
%		 \|b\|_{\mathrm{Osc}_{\ep}^p(\mu)}:=\Big\| \Big\{ \mu(B_Q)^{\ep} \fint_{B_Q}|b-\langle b\rangle_{B_Q} |d\mu\Big\}_{Q\in \mathscr{D}} \Big\|_{\ell^p},
%	\end{equation}
%	and
%	\begin{equation}
%		\|b\|_{\widetilde{B}_{p}^{\alpha}(\mu)}
%		\thickapprox
%		 \|b\|_{\widetilde{\mathrm{Osc}}_{\alpha}^p(\mu)}:=\Big\| \Big\{ {l(Q)}^{\alpha} \fint_{B_Q}|b-\langle b\rangle_{B_Q} |d\mu\Big\}_{Q\in \mathscr{D}} \Big\|_{\ell^p}.
%	\end{equation}
%
\end{proposition}

%The following Proposition \ref{P3.10} and \ref{P3.11} immediately imply Proposition \ref{P2}. 

Proposition \ref{P2} will be a consequence of the following Propositions \ref{P3.10} and \ref{P3.11}. We note that the range of the admissible parameters is larger in Proposition \ref{P3.10}.

\begin{proposition}\label{P3.10}
Suppose that the space of homogeneous type $(X,\rho,\mu)$ has a lower dimension $d\in (0,\infty)$.
Let $p\in(1,\infty)$ and let $\phi$ be as in \eqref{eq:phixy} and \eqref{phi} with $\ep\in(0,\frac{2}{p})$ and $\alpha \in (0,\frac{2d}{p})$. For all measurable $b$ on $X$, we have 
\begin{equation*}
  \Norm{b}_{\B_p(\phi,\mu)}\sim 
  \Big\| \Big\{ \phi(Q)^{\ep} 
	 \inf_c \Big(\fint_{B_Q}|b- c |^p d\mu\Big)^{\frac{1}{p}}\Big\}_{Q\in \mathscr{D}} \Big\|_{\ell^p}.
\end{equation*}
%
%\begin{equation}
%	\|b\|_{\dot{B}_{p}^{\p-\ep}(\mu)}
%	\thickapprox
%	 \Big\| \Big\{ \mu(B_Q)^{\ep} 
%	 \inf_c \Big(\fint_{B_Q}|b- c |^p d\mu\Big)^{\frac{1}{p}}\Big\}_{Q\in \mathscr{D}} \Big\|_{\ell^p},
%\end{equation}
%and 
%\begin{equation}
%	\|b\|_{\widetilde{B}_{p}^{\alpha}(\mu)}
%	\thickapprox
%	 \Big\| \Big\{ {l(Q)}^{\alpha} 
%	 \inf_c \Big(\fint_{B_Q}|b- c |^p d\mu\Big)^{\frac{1}{p}}\Big\}_{Q\in \mathscr{D}} \Big\|_{\ell^p}.
%\end{equation}
\end{proposition}

\begin{proof}
%Let $\phi(Q)$ be defined as in \eqref{phi} for any cube $Q\subset X$. 
By \eqref{eq3.13}, it follows that
\begin{equation}\label{eq4.8}
\begin{aligned}
 & \Big\| \Big\{ \phi(Q)
	 \inf_c \Big(\fint_{B_Q}|b- c |^p d\mu\Big)^{\frac{1}{p}}\Big\}_{Q\in \mathscr{D}} \Big\|_{\ell^p}^p\\
	 & \sim 	 \Big\| \Big\{ \phi(Q)
	 \Big(\fint_{B_Q} \fint_{B_Q} |b(x)- b(y)|^p d\mu(x) d \mu(y)\Big)^{\frac{1}{p}}
	 \Big\}_{Q\in \mathscr{D}} \Big\|_{\ell^p}^p\\
	 &= \sum_{Q\in \mathscr{D}}{\phi(Q) ^{p}} \fint_{B_Q} \fint_{B_Q} |b(x)- b(y)|^p d\mu(x) d \mu(y)\\
	 & = \int_X \int_X |b(x)-b(y)|^p \sum_{Q\in \mathscr{D}}
	 1_{B_Q}(x) 1_{B_Q}(y)\cdot \frac{ \phi(Q)^p}{{\mu(B_Q)}^{2}} d\mu(x) d \mu(y).
\end{aligned}
\end{equation}

Next, we consider the sum in the above formula. 
On the one hand,
there is a dyadic cube $Q$ which contains $x$ and $\ell(Q)\sim \rho(x,y)$. Under this condition, we obtain that the corresponding $B_Q$ contains both $x$ and $y$, satisfying $\mu(B_Q)\sim V(x,y)$. Hence, for this $Q$,
\begin{equation*}
\phi(Q) =
	\left.\begin{cases}
		\mu(B_Q)^{\ep} &\sim\ V(x,y)^{\ep} \\
		l(Q)^{\alpha} &\sim\ \rho(x,y)^\alpha
	\end{cases}\right\}
	=\phi(x,y),
	\qquad
	\frac{\phi(Q)^p}{\mu(B_Q)^2}\sim \frac{\phi(x,y)^p}{V(x,y)^2}
\end{equation*}
according to the two cases in \eqref{eq:phixy} and \eqref{phi}, 
and thus 
\begin{equation}\label{eq4.9}
\begin{aligned}
\sum_{Q\in \mathscr{D}}
	 1_{B_Q}(x) 1_{B_Q}(y)\cdot \frac{ \phi(Q)^p}{{\mu(B_Q)}^{2}} 
	 \gtrsim  \frac{\phi(x,y)^p}{V(x,y)^2}.
	 \end{aligned}
	 \end{equation}
%	 where the functions $H$ in \eqref{DH} and $\phi$ in \eqref{phi} are chosen consistently. More precisely, we state that if $\phi(Q)=\mu(B_Q)^{\ep}$ and $H(x,y)=V(x,y)^{\ep-1}$,
%	 \begin{align}\label{4.10}
%	 \textup{The LHS of } \eqref{eq4.9}\gtrsim
%	 	 V(x,y)^{p(\ep-1)}\cdot V(x,y)^{p-2}=\frac{1}{V(x,y)^{2-p\ep}};
%	 \end{align}
%and  if $\phi(Q)=l(Q)^{\alpha}$ and $H(x,y)=\frac{\rho(x,y)^{\alpha}}{V(x,y)}$,
%	 \begin{align}\label{4.11}
%	 \textup{The LHS of } \eqref{eq4.9}\gtrsim
%	 	 \frac{\rho(x,y)^{p\alpha}}{V(x,y)^p}\cdot V(x,y)^{p-2}=\frac{\rho(x,y)^{\alpha}}{V(x,y)^2}.
%	 \end{align}
%	 
	 
On the other hand, for any $x,y\in B_Q$ with $x\neq y$, it follows that $\ell(Q)\gtrsim \rho(x,y)$. For this type of dyadic cube $Q$, it contains a minimal cube belonging to the same type and there are at most boundedly many such minimal cubes thanks to the doubling property. Then every $Q$ appearing in the sum is a dyadic ancestor $P^{[k]}$ of some minimal cube $P$. 
Now we consider the following two cases.
 
	 {\bf{Case 1:}} $\phi(Q)=\mu(B_Q)^{\ep}$ with  $\ep \in (0,\frac{2}{p})$. Taking $Q=P$ and $\gamma=2-p\ep>0$ in (\ref{eq3.15}), we get 
$$\sum_{k=0}^{\infty}\frac{1}{\mu(P^{[k]})^{2-p\ep}}
\lesssim \frac{1}{\mu(P)^{2-p\ep}}\sim \frac{1}{V(x,y)^{2-p\ep}}.$$	 
Summing over the boundedly many minimal $P$, we conclude that 
\begin{equation}\label{eq4.12}
\sum_{Q\in \mathscr{D}}
	 \frac{1_{B_Q}(x) 1_{B_Q}(y)}{{\mu(B_Q)}^{2-p\ep}} \lesssim \frac{1}{V(x,y)^{2-p\ep}}=\frac{\phi(x,y)^p}{V(x,y)^2}.
	 \end{equation}
	 
	 {\bf{Case 2:}} $\phi(Q)=l(Q)^{\alpha}$ with $\alpha \in (0,\frac{2d}{p})$. Taking $Q=P$, $s=p\alpha$ and $t=2$ in \eqref{eq4.10}, 
	 noting that $p\alpha<p\frac{2d}{p}=2d$,
	 we get 
\begin{equation*}
\sum_{k=0}^{\infty}\frac{l(P^{[k]})^{p\alpha}}{\mu(P^{[k]})^{2}}
\lesssim \frac{l(P)^{p\alpha}}{\mu(P)^{2}}
\thickapprox \frac{\rho(x,y)^{p\alpha}}{V(x,y)^{2}}.
\end{equation*}
Summing over the boundedly many minimal $P$, we conclude that for any $\alpha \in (0,\frac{2d}{p})$,
\begin{equation} \label{4.13}
\sum_{Q\in \mathscr{D}}
	   1_{B_Q}(x) 1_{B_Q}(y) \cdot  \frac{{l(Q)}^{p\alpha}}{{\mu(B_Q)}^{2}} 
	 \lesssim \frac{\rho(x,y)^{p\alpha}}{V(x,y)^{2}}=\frac{\phi(x,y)^p}{V(x,y)^2} 
	 \end{equation}
By \eqref{eq4.9} through \eqref{4.13}, we find that 
\begin{align*}
	\sum_{Q\in \mathscr{D}}
	 1_{B_Q}(x) 1_{B_Q}(y)\cdot \frac{ \phi(Q)^p}{{\mu(B_Q)}^{2}} 
	  \sim \frac{\phi(x,y)^p}{V(x,y)}.
\end{align*}
Substituting this into \eqref{eq4.8}, we finish the proof of Proposition \ref{P3.10}.  
\end{proof}

\begin{proposition}\label{P3.11}
Suppose that the space of homogeneous type $(X,\rho,\mu)$ has a lower dimension $d\in (0,\infty)$.
Let $p\in(1,\infty)$ and $\phi$ be as in \eqref{eq:phixy} and \eqref{phi} with $\ep\in(0,\frac{1}{p})$ and $\alpha\in (0,\frac{d}{p})$. For the oscillatory norm $\Norm{\cdot }_{\mathrm{Osc}_{\phi}^p(\mu)}$ defined in \eqref{Osc} and any $b\in  L_{\mathrm{loc}}^1(\mu) $, we have 
\begin{equation}\label{eqP3.11}
	\begin{aligned}
		 \Big\| \Big\{ \phi(Q)
	 \inf_c \Big( \fint_{B_Q}|b- c |^p d\mu\Big)^{\frac{1}{p}}\Big\}_{Q\in \mathscr{D}} \Big\|_{\ell^p}
	 \sim 
	 		 \|b\|_{\mathrm{Osc}_{\phi}^p(\mu)}.
	\end{aligned}
\end{equation}
\end{proposition}

\begin{proof}
For a given $Q \in \mathscr{D}$ and any $r \in (0,\infty)$ and $\ep \in (0,1)$, we define
\begin{align*}
\lambda_{r}^{\phi}(Q)
&:= \inf_c \phi(Q)  \Big( \fint_{Q}|b- c |^r d\mu\Big)^{\frac{1}{r}}\\
&= \phi(Q) \inf_c  \Big( \fint_{Q}|b- c |^r d\mu\Big)^{\frac{1}{r}}
=: \phi(Q) \lambda_{r}(Q)
\end{align*}
For a sequence $\lambda=(\lambda_Q)_{Q\in \mathscr{D}}$ of numbers, we define the new sequence Car$\lambda$ by 
$$\mathrm{Car}\lambda(P):=\frac{1}{\mu(P)}\sum_{Q\subseteq P}|\lambda_Q|\mu(Q).$$
Taking $r=p$ and $s=1$ in \cite[Proposition 11.3]{H2},  we get for any given $Q \in \mathscr{D}$,
\begin{equation}\label{eqCar}
\lambda_{p}(Q) 
\lesssim \big(\mathrm{Car}(\lambda_1)^p\big)^{\frac{1}{p}}(Q)
=\Big(\frac{1}{\mu(Q)}\sum_{S \subseteq Q} \lambda_1(S)^p \mu(S)\Big)^{\frac{1}{p}}.
\end{equation}
It implies that %for any $p\in(1,\infty)$, $d\in (0,\infty)$, $\ep\in(0,1/p)$ and $\alpha\in (0,d/p)$, 
\begin{equation}\label{eq3.14}
\begin{aligned}
\|\{\lambda_{p}^{\phi}(Q)\} & _{Q \in \mathscr{D}}\|_{\ell^p}^p 
\lesssim \| \{\phi(Q) \big(\mathrm{Car}(\lambda_1)^p\big)^{\frac{1}{p}}(Q) \}_{Q \in \mathscr{D}}\|_{\ell^p}^p
=\sum_{Q\in \mathscr{D}}\phi(Q)^{p} \mathrm{Car}(\lambda_1)^p(Q)\\
& = \sum_{Q\in \mathscr{D}}\phi(Q)^{p} \sum_{S \subseteq Q} \frac{\mu(S)}{\mu(Q)}\lambda_1(S)^p 
=\sum_{S\in \mathscr{D}} \lambda_1(S)^p\sum_{Q \supseteq S} \frac{\mu(S)}{\mu(Q)}\cdot \phi(Q)^p\\
&=\sum_{S\in \mathscr{D}} \lambda_1(S)^p\sum_{k=0}^{\infty} \frac{\mu(S)}{\mu(S^{[k]})}\cdot \phi(S^{[k]})^p,
\end{aligned}
\end{equation}
where $S^{[k]}$ is the $k$-th strict dyadic ancestor.
Similarly as in the proof of Proposition \ref{P3.10}, we consider the two cases.

{\bf{Case 1:}} $\phi(Q)=\mu(B_Q)^{\ep}$ with $\ep \in (0,\frac{1}{p})$. 
Taking $Q=S$ and $\gamma=1-p\ep>0 $ in (\ref{eq3.15}), we get
\begin{equation}\label{case1'}
	\sum_{k=0}^{\infty}\frac{\mu(S)}{\mu(S^{[k]})}\cdot \phi(S^{[k]})^p
	=\sum_{k=0}^{\infty} \frac{\mu(S)}{\mu(S^{[k]})^{1-p\ep}}
	\lesssim \frac{\mu(S)}{\mu(S)^{1-p\ep}}
	=\mu(S)^{p\ep}\sim\phi(S)^p.
\end{equation}

{\bf{Case 2:}} $\phi(Q)=l(Q)^{\alpha}$ with $\alpha \in (0,\frac{d}{p})$. Taking $Q=S$, $s=p\alpha$ and $t=1$ in \eqref{eq4.10}, noting that $p\alpha<p\frac{d}{p}=d\cdot 1$ as required to apply \eqref{eq4.10}, it follows that
\begin{equation}\label{case2'}
	\sum_{k=0}^{\infty}\frac{\mu(S)}{\mu(S^{[k]})}\cdot \phi(S^{[k]})^p=	
\sum_{k=0}^{\infty} \frac{\mu(S)}{\mu(S^{[k]})}\cdot l(S^{[k]})^{p\alpha}
\lesssim \mu(S) \frac{l(S)^{p\alpha}}{\mu(S)}
=l(S)^{p\alpha}=\phi(S)^p.
\end{equation}
Thus, \eqref{case1'} and \eqref{case2'} yield that 
$$
\sum_{k=0}^{\infty} \frac{\mu(S)}{\mu(S^{[k]})}\cdot \phi(S^{[k]})^p
\lesssim \phi(S)^p.
$$
Substituting this into \eqref{eq3.14}, we get 
$$\|\{\lambda_{p}^{\phi}(Q)\}  _{Q \in \mathscr{D}}\|_{\ell^p}^p 
\lesssim \sum_{S\in \mathscr{D}} \lambda_1(S)^p \phi(S)^p=
\|\{\lambda_{1}^{\phi}(Q)\}  _{Q \in \mathscr{D}}\|_{\ell^p}^p 	 \sim 
	 		 \|b\|_{\mathrm{Osc}_{\phi}^p(\mu)},$$
where we used \eqref{eq3.13} in the last equality ``$\sim$''. 	 		 

This proves ``$\lesssim$'' in \eqref{eqP3.11}, while ``$\gtrsim$'' is clear.
\end{proof}

\section{The complex median method of Wei--Zhang revisited}\label{sec:median}

Two methods for proving lower bounds for commutators in the recent literature are the ``approximate weak factorisation'' from \cite{H}, and versions of the so-called median method. Until recently, this second approach was restricted to real-valued functions, for which the median is conventionally defined. However, Wei and Zhang \cite{WZ} recently demonstrated that, with suitable modifications, the median method can also be extended to work with complex-valued functions. A key to that extension is the existence of a ``two-dimensional median'' in a suitable sense. This existence is recently due to Baringhaus and Gr\"ubel \cite{BG}, and a slightly weaker variant (but sufficient for commutators) was independently rediscovered by \cite[Theorem 1.5]{WZ}. (See Remark \ref{rem:median} below for details.) In Proposition \ref{prop:median} below, we present yet another variant of the existence of a two-dimensional median, with a simpler proof and roughly the same scope of applicability as far as commutator applications are concerned.

%A typical way of obtaining lower bounds for commutators in several recent contributions is by means of some version of the so-called median method.

For $u\in\C$ with $|u|=1$ and $\theta\in[0,2\pi]$, we denote by
\begin{equation*}
\begin{split}
  \Gamma(u,\theta) &:=\{z\in\C:\Re(\bar uz)\geq\cos(\tfrac12\theta)|z|\}, \\
  \Gamma^\circ(u,\theta)&:=\{z\in\C:\Re(\bar uz)>\cos(\tfrac12\theta)|z|\}
\end{split}
\end{equation*}
the closed and open cones in direction $u$ of total angle $\theta$ (thus $\frac12\theta$ on either side of~$u$).

\begin{proposition}\label{prop:median}
Let $(N,\eps)=(3,\frac14)$.
%The following statement holds for each of the cases
%\begin{equation*}
%  (N,\eps)\in\{(1,1),(2,\tfrac12),(3,\tfrac14),(4,\tfrac{1}{16})\}.
%\end{equation*}
Given a Borel probability $\nu$ on $\C$, there exist $m\in\C$ and $N$ closed cones $\{\Gamma_j\}_{j=1}^N$ of angle $\frac{2\pi}{N}$ each, together covering all $\C$, such that
\begin{equation*}
  \nu(m+\Gamma_j)\geq\eps
\end{equation*}
for every $j\in\{1,\ldots,N\}$.
\end{proposition}

\begin{remark}\label{rem:median}
Recall that the median of a probability $\nu$ on $\R$ is a value $m\in\R$ such that both $\nu((-\infty,m])\geq\frac12$ and $\nu([m,\infty))\geq\frac12$, i.e., the mass of $\nu$ is essentially equally divided on two sides of $m$, to the extent allowed by a possible point mass at $m$.

The value $m\in\C$ in Proposition \ref{prop:median} is a ``quasi-median'' of $\nu$ in the sense that the measure of $\nu$ is divided among each of the $N$ conical regions around $m$, with a fair (if not quite equal) share to each. We do not know whether Proposition \ref{prop:median} holds with $(N,\eps)=(3,\frac13)$. This would be interesting in itself, but it makes no difference to our applications to commutators further below.

Proposition \ref{prop:median} also holds in each of the cases $(N,\eps)\in\{(1,1),(2,\tfrac12),(4,\tfrac{1}{4})\}$, where
case $(N,\eps)=(1,1)$ is the triviality $\nu(\C)\geq 1$, while case $N=2$ and $\eps=\frac12$ is a simple extension of the existence of the median for the probability $\nu_1(A):=\nu(A\times\R)$ on $\R$: if $m$ is a median of $\nu_1$, the required property is satisfied with the cones $\Gamma(\pm 1,\pi)$.

The nontrivial case $(N,\eps)=(4,\frac{1}{4})$ is a recent result of Baringhaus and Gr\"ubel \cite[Theorem 1]{BG}, while Wei and Zhang \cite[Theorem 1.5]{WZ} independently obtain the weaker variant with $(N,\eps)=(4,\frac{1}{16})$. The proof of \cite[Theorem 1.5]{WZ} is ``elementary'' but rather tedious, taking about 10 pages in \cite[Section 6]{WZ}. The proof of \cite[Theorem 1]{BG} in \cite[Section 4.1]{BG} takes less than a page, but it uses as input the special case of absolutely continuous measures, which is quoted from elsewhere.

We will show that case $(N,\eps)=(3,\frac14)$ holds with a much simpler proof, while the result has roughly the same scope of applicability as \cite[Theorem 1.5]{WZ} for commutator estimates, as we will see further below.

The proof below shows that the cones  can be chosen to be in the standard orientation $\Gamma_j=\Gamma(e^{i j\frac{2\pi}{3}},\frac{2\pi}{3})$ (or, by a rotation, in any other prescribed orientation that we like). This is in contrast to \cite[Theorem 1]{BG} and \cite[Theorem 1.5]{WZ}, where the orientation of the cones will also depend on $\nu$.
\end{remark}

\begin{proof}[Proof of Proposition \ref{prop:median}]
We will identify $\C\simeq\R^2$ without explicit mention whenever convenient.
Let $t$ be a median of $\nu_2:E\subset\R\mapsto\nu(\R\times E)$. Hence both
\begin{equation*}
  \nu(z:\Im z\geq t)\geq\frac12,\qquad  \nu(z:\Im z\leq t)\geq\frac12.
\end{equation*}
Let
\begin{equation*}
  \sigma_0:=\sup\{\sigma:\nu(\sigma+it+\Gamma(1,\tfrac{2\pi}{3}))\geq\tfrac14\}.
\end{equation*}
We claim that $m:=\sigma_0+it$ satisfies the required property with the cones $\Gamma_0:=\Gamma(1,\frac{2\pi}{3})$ and $\Gamma_\pm:=\Gamma(u_\pm,\frac{2\pi}{3})$, where $u_\pm:=e^{\pm i\frac{2\pi}{3}}$. Since
\begin{equation*}
  m+\Gamma_0=\bigcap_{\sigma<\sigma_0}(\sigma+it+\Gamma_0),\qquad
  m+\Gamma_0^\circ=\bigcup_{\sigma>\sigma_0}(\sigma+it+\Gamma_0),
\end{equation*}
it follows from the continuity of measure and the definition of supremum that
\begin{equation*}
  \nu(m+\Gamma_0)\geq\tfrac14,\qquad
  \nu(m+\Gamma^\circ_0)\leq\tfrac14.
\end{equation*}

Noting that
\begin{equation*}
  \C\setminus\Gamma_0^\circ
  \subset \Gamma_+\cup\{z:\Im z<0\},
\end{equation*}
it follows that
\begin{equation*}
  \tfrac34
  \leq  \nu( m+\C\setminus\Gamma^\circ_0)  
  \leq \nu(m+\Gamma_+)+\nu(\{z:\Im z<t\})
  \leq \nu(m+\Gamma_+)+\tfrac12
\end{equation*}
recalling that $t$ is a median of $\nu_2$. Hence
$\nu(m+\Gamma_+)\geq\tfrac 34-\tfrac12=\tfrac14$,
and the proof of the similar estimate with $\Gamma_-$ in place of $\Gamma_+$ is entirely analogous.
This completes the proof.
\end{proof}

The following result implements the ``median method'' in the present setting: With the help of the ``median'' from Proposition \ref{prop:median}, we dominate the oscillations of $b$ by the action of a commutator $[b,T]$ on suitable test functions.

\begin{proposition}\label{prop:mm}
Let $K:X\times X\to\C$ be a kernel, and let $B$ and $\tilde B$ be balls such that
\begin{equation}\label{eq:K(x,y)>}
   \abs{K(x,y)}\gtrsim \frac{\phi(B)}{\mu(B)},\qquad
   \abs{\arg(\bar v K(x,y))}\leq\eta\leq\frac{\pi}{9}
\end{equation}
for some $v\in\C$ with $\abs{v}=1$ and for all $x\in B$ and $y\in\tilde B$. If $b\in L^1_{\operatorname{loc}}(X)$, then there are measurable subsets $E\subset B$ and $F\subset \tilde B$ such that
\begin{equation*}
  \inf_c\phi(B)\fint_B\abs{b(x)-c}d\mu(x)
  \lesssim\frac{1}{\mu(\tilde B)}\abs{\langle 1_E, [b,T]1_F\rangle},
\end{equation*}
where $T$ is the integral operator with kernel $K$.
\end{proposition}

\begin{proof}
We apply Proposition \ref{prop:median} to $\nu(A):=\mu(\tilde B)^{-1}\mu(\tilde B\cap\{b\in A\})$. For convenience, let us denote the resulting cones by $-\Gamma_i$ instead of $\Gamma_i$. Thus, we find some $m\in\C$ and cones $\Gamma_i$ such that $\mu(\tilde B\cap\{b\in m-\Gamma_i\})\geq\eps\mu(\tilde B)$ for all $i=1,\ldots,N$. (Proposition \ref{prop:median} gives $N=3$, but we write the proof with a generic $N$, showing that one could equally well use $N=4$ from \cite[Theorem 1.5]{WZ}.) We will prove the claimed estimate with $m$ in place of $c$.

Since $\bigcup_{j=1}^N\Gamma_j=\C$, it follows that
\begin{equation}\label{eq:mStart}
  \phi(B)\fint_B\abs{b(x)-m}d\mu(x)
  \leq\sum_{j=1}^N\int_{B\cap\{b\in m+\Gamma_j\}}\frac{\phi(B)}{\mu(B)}\abs{b(x)-m}d\mu(x).
\end{equation}
Now consider a fixed $\Gamma_j=\Gamma(u_j,\frac{2\pi}{N})$. We claim that, for every $x\in B\cap\{b\in m+\Gamma_j\}$ and $y\in\tilde B\cap\{b\in m-\Gamma_j\}$, we have
\begin{equation}\label{eq:mClaim}
  \frac{\phi(B)}{\mu(B)}\abs{b(x)-m}
  \lesssim\Re[\overline{u_jv}(b(x)-b(y))K(x,y)],
\end{equation}
where $v\in\C$ comes from the assumption \eqref{eq:K(x,y)>}

To justify \eqref{eq:mClaim}, let $b(x)=m+u_jt_1  e^{i\phi_1}$ and $b(y)=m-u_jt_2 e^{i\phi_2}$, where $t_1=\abs{b(x)-m}$, $t_2=\abs{b(y)-m}$, and $\abs{\phi_j}\leq\theta=\frac{\pi}{N}$. By \eqref{eq:K(x,y)>}, we can also write $K(x,y)=v t_0 e^{i\phi_0}$, where $t_0=\abs{K(x,y)}\gtrsim \phi(B)\mu(B)^{-1}$. Hence
\begin{equation*}
\begin{split}
  \Re[\overline{u_jv}(b(x)-b(y))K(x,y)]
  &=\Re[\bar u_j((m+u_jt_1e^{i\phi_1})-(m-u_jt_2 e^{i\phi_2}))\bar v v t_0 e^{i\phi_0}] \\
  &=\Re(t_1 t_0 e^{i(\phi_1+\phi_0)}+t_2 t_0 e^{i(\phi_2+\phi_0)}) \\
  &=t_1 t_0 \cos(\phi_1+\phi_0)+t_2 t_0 \cos(\phi_2+\phi_0),
\end{split}
\end{equation*}
where
\begin{equation*}
  \abs{\phi_j+\phi_0}\leq\theta+\eta=\frac{\pi}{N}+\frac{\pi}{9}
  \leq\frac{\pi}{3}+\frac{\pi}{9}=\frac{4}{9}\pi<\frac{\pi}{2}.
\end{equation*}
Since $\cos \phi=\cos\abs{\phi}$ is a decreasing function of $\abs{\phi}\in[0,\frac{\pi}{2}]$, it follows that
\begin{equation*}
\begin{split}
  \Re[\overline{u_jv}(b(x)-b(y))K(x,y)]
  &\geq (t_1+t_2) t_0 \cos(\tfrac49\pi) \\
  &\gtrsim t_1 t_0
  \gtrsim\abs{b(x)-m}\frac{\phi(B)}{\mu(B)}.
\end{split}
\end{equation*}
This proves \eqref{eq:mClaim}.

We can now take the average of \eqref{eq:mClaim} over $y\in\tilde B\cap\{b\in m-\Gamma_j\}$. This gives
\begin{equation*}
\begin{split}
   \frac{\phi(B)}{\mu(B)}\abs{b(x)-m}
   &\lesssim\fint_{\tilde B\cap\{b\in m-\Gamma_j\}}
     \Re[\overline{u_jv}(b(x)-b(y))K(x,y)]d\mu(y) \\
    &\lesssim \Re\Big[\overline{u_jv}\fint_{\tilde B}
     (b(x)-b(y))K(x,y)
     1_{\tilde B\cap\{b\in m-\Gamma_j\}}(y)     d\mu(y)\Big],
\end{split}
\end{equation*}
where in the last step we used $\mu(\tilde B\cap\{b\in m-\Gamma_j\})\gtrsim\mu(\tilde B)$. Substituting back to \eqref{eq:mStart}, we obtain
\begin{equation*}
\begin{split}
  &\phi(B)\fint_B\abs{b(x)-m}d\mu(x) \\
  &\lesssim\sum_{j=1}^N \Re\Big[\overline{u_jv}\int_B 1_{B\cap\{b\in m+\Gamma_j\}}(x)
    \fint_{\tilde B} (b(x)-b(y))K(x,y)1_{\tilde B\cap\{b\in m-\Gamma_j\}}(y)d\mu(y) d\mu(x)\Big] \\
  &=\mu(\tilde B)^{-1}\sum_{j=1}^N \Re\Big[\overline{u_jv}\left\langle 1_{B\cap\{b\in m+\Gamma_j\}},
  [b,T]1_{\tilde B\cap\{b\in m-\Gamma_j\}} \right\rangle \Big].
\end{split}
\end{equation*}
The proof is completed by estimating the real part by the absolute value, and taking for $(E,F)$ the $(B\cap\{b\in m+\Gamma_j\},\tilde B\cap\{b\in m-\Gamma_j\})$ for which the largest value is achieved.
\end{proof}

\begin{corollary}\label{cor:mm}
Let $T$ be a strongly non-degenerate $\phi$-fractional integral. Then for all balls $B\subset X$, there exist subsets $E,F\subset B^*$, where $B^*= C\cdot B$ with a fixed dilation factor $C$, such that
\begin{equation*}
  \phi(B)\int_B\abs{b-\ave{b}_B}d\mu
  \lesssim \abs{\langle 1_E, [b,T]1_F\rangle}.
\end{equation*}
\end{corollary}

\begin{proof}
Consider a fixed ball $B=B(x_0,r)$. By assumption (recall Definition \ref{D1'}), for some $y_0\in B(x_0,\overline{C}Ar)\setminus B(x_0,Ar)$, the kernel $K$ of $T$ satisfies either \eqref{eq:sndg1} or \eqref{eq:sndg2}. Let us first assume \eqref{eq:sndg1} and denote $\tilde B:=B(y_0,r)$. If $y\in\tilde B$, then
\begin{equation*}
  \rho(y,x_0)\leq A_0\rho(y,y_0)+A_0\rho(y_0,x_0)<A_0(1+\overline{C}A)r,
\end{equation*}
thus $\tilde B\subset B^*:=B(x_0,A_0(1+\overline{C}A)r)=C\cdot B$ with $C=A_0(1+\overline{C}A)$, and symmetrically $B\subset C\cdot \tilde B$.

Note that \eqref{eq:sndg1} coincides with assumption \eqref{eq:K(x,y)>} of Proposition \ref{prop:mm}; hence the conclusion of Proposition \ref{prop:mm} gives us $E\subset B\subset B^*$ and $F\subset \tilde B\subset B^*$ such that
\begin{equation*}
  \phi(B)\int_B\abs{b-\ave{b}_B}d\mu
  \lesssim\inf_c\phi(B)\int_B\abs{b-c}d\mu
  \lesssim\frac{\mu(B)}{\mu(\tilde B)}\abs{\langle 1_E, [b,T]1_F\rangle}.
\end{equation*}
Since $B\subset C\cdot\tilde B$, it follows from doubling that
\begin{equation*}
  \mu(B)\leq\mu(C\cdot\tilde B)\lesssim\mu(\tilde B),
\end{equation*}
and we get the claimed estimate in this first case.

Suppose then that  \eqref{eq:sndg2} holds instead. Let $K^*(x,y)=K(y,x)$, and let $T^*g(u):=\int_X K^*(u,v)f(v)d\mu(v)$. Then $T^*$ is also a $\phi$-fractional integral operator, whose kernel satisfies the assumptions of the previous case in the ball $B$. Thus, by the previous case, we find sets $E,F\subset B^*$ such that
\begin{equation*}
  \phi(B)\int_B\abs{b-\ave{b}_B}\lesssim\abs{\langle 1_E,[b,T^*]1_F\rangle}.
\end{equation*}
But
\begin{equation*}
  \langle 1_E,[b,T^*]1_F\rangle
  =\langle 1_E,bT^*1_F-T^*(b1_F)\rangle
  =\langle T(b1_E),1_F\rangle-\langle bT1_E,1_F\rangle
  =-\langle 1_F,[b,T]1_E\rangle,
\end{equation*}
and hence we get the claimed estimate, only with the names of $E$ and $F$ interchanged. This completes the proof.
\end{proof}

\section{The lower bounds for the fractional commutators}\label{sec:lower}

%The main results of this section are the following lower bounds for Schatten $S^{p,q}$ norm of fractional commutators $[b,T]$, by means of a fractional oscillation space $\mathrm{Osc}_{\phi}^{p,q}$ which generalizes the case $\mathrm{Osc}_{\phi}^{p} = \mathrm{Osc}_{\phi}^{p,p} $ introduced in Section 4. 

The main results of this section are lower bounds for the Schatten norm of fractional commutators $[b,T]$, by means of fractional oscillation norms that generalize those introduced in Section \ref{sec:Besov}. Although our main concern in this paper is Schatten $S^p$ norms, we state and prove the following estimate for the more general Schatten--Lorentz $S^{p,q}$ norms: on the one hand, this added generality comes essentially for free by the same argument; on the other hand, it may be useful elsewhere. We denote by $\ell^{p,q}$ the usual Lorentz sequence space (see e.g.\ \cite{BL}) and by $S^{p,q}=S^{p,q}(L^2(\mu))$ the space of compact operators on $L^2(\mu)$ with
\begin{equation*}
  \Norm{R}_{S^{p,q}}:=
  \Norm{ \{a_n(R)\}_{n=0}^\infty }_{\ell^{p,q}},
\end{equation*}
where $a_n(R)$ are the approximation numbers (or singular values) as in \eqref{eq:Sp}.

\begin{proposition}\label{P3}
Let $(X,\rho,\mu)$ be a space of homogeneous type with a system of dyadic cubes $\mathscr{D}$ in the sense of Definition \ref{De2.4} and, for each $Q\in \mathscr{D}$, let $B_Q$ be a ball centred at $Q$ and of radius $c\cdot l(Q)$ for a constant $c$ only depending on the space $X$. 
Let $q\in [1,\infty]$, $p \in (1,\infty)$, and $\phi$ be as in \eqref{eq:phixy} and \eqref{phi} with $\ep \in (0,1)$ and $\alpha\in (0,\infty)$.
Let $T$ be a strongly non-degenerate $\phi$-fractional integral operator. Then for all $b \in L_{\mathrm{loc}}^1(\mu)$, we have
\begin{equation}\label{eq3.3}
	\|b\|_{\mathrm{Osc}_{\phi}^{p,q}(\mu)}
	: =\Big\| \Big\{ \phi(Q) \fint_{B_Q}|b-\langle b\rangle_{B_Q} |d\mu\Big\}_{Q\in \mathscr{D}} \Big\|_{\ell^{p,q}}.
	\lesssim \|[b,T]\|_{S^{p,q}},
\end{equation}
In particular, when $p=q\in(1,\infty)$, we have
\begin{equation}
\|b\|_{\mathrm{Osc}_{\phi}^{p}(\mu)}=\|b\|_{\mathrm{Osc}_{\phi}^{p,p}(\mu)} \lesssim \|[b,T]\|_{S^{p,p}}= \|[b,T]\|_{S^{p}}.
\end{equation}
\end{proposition}

To prove Proposition \ref{P3}, motivated by the ideas of \cite{H2} which go back to \cite{RS}, we consider the bi-sublinear maximal operator $\mathcal{M}$ for a sequence $\{(e_Q,h_Q)\}_{Q \in \mathscr{D}}$ of pairs of functions,  
\begin{equation}\label{eq:Mdef}
  \mathcal{M}:(f,g) \mapsto \sup_{Q \in \mathscr{D}} 1_Q 
      \frac{|\langle f,e_Q \rangle \langle g,h_Q \rangle |}{\mu(Q)},
\end{equation}
 and two related maximal operators 
 $$\mathcal{M}^1f:= \sup_{Q \in \mathscr{D}} 1_Q 
      \frac{|\langle f,e_Q \rangle |}{\mu(Q)^{\frac{1}{2}}}, 
      \quad  
      \mathcal{M}^2g:= \sup_{Q \in \mathscr{D}} 1_Q 
      \frac{|\langle g,h_Q \rangle |}{\mu(Q)^{\frac{1}{2}}} .$$
Obviously, $\mathcal{M}(f,g) \leq \mathcal{M}^1 f\cdot \mathcal{M}^2 g$. If the sequence of functions $\{(e_Q,h_Q)\}_{Q \in \mathscr{D}}$ satisfies 
\begin{equation}\label{3.8}
|e_Q|+|h_Q| \lesssim \mu(Q)^{-\frac{1}{2}}1_{B_Q^*},      	
\end{equation}
where $B_Q^*=c\cdot B_Q$ is a concentric extension of $B_Q$ for any fixed constant $c$, then 
$$\mathcal{M}^1f \lesssim Mf,\qquad \mathcal{M}^2g \lesssim Mg,$$  
where $M$ is the Hardy-Littlewood maximal operator. 
Under the assumption (\ref{3.8}), by H\"older's inequality and the $(L^2(\mu),L^2(\mu))$ boundedness of the maximal operator $M$, we obtain 
\begin{equation*}
  \|\mathcal{M}(f,g)\|_{L^1(\mu)}
  \leq \| \mathcal{M}^1f\|_{L^2(\mu)} \| \mathcal{M}^2g\|_{L^2(\mu)} \lesssim  \|f\|_{L^2(\mu)}\|g\|_{L^2(\mu)}.
\end{equation*}
            Thus, the bi-sublinear maximal operator $\mathcal{M}$ for a sequence of functions $\{(e_Q,h_Q)\}_{Q \in \mathscr{D}}$ satisfying (\ref{3.8}), is bounded from $L^2(\mu)\times L^2(\mu)$ to $L^1(\mu)$.
            
We will also need the following result, which is \cite[Corollary 7.7]{H2}:

 \begin{proposition}\label{P3.7}
 	Let $p\in (1,\infty)$ and $q \in [1,\infty]$. For all sequences $\{(e_Q,h_Q)\}_{Q \in \mathscr{D}}$ and their related maximal operator \eqref{eq:Mdef}, we have the following estimate for operators $A\in S^{p,q}(L^2(\mu))$:
 \begin{equation*}
\|\{\langle Ae_Q, h_Q\rangle\}_{Q\in \mathscr{D}} \|_{l^{p,q}(\mathscr{D})}
\lesssim \|\mathcal{M}\|_{L^2(\mu)\times L^2(\mu)\rightarrow L^1(\mu)}
\|A\|_{S^{p,q}(L^2(\mu))}.
 \end{equation*}	 
 \end{proposition}

 We are now prepared to give:
 
 \begin{proof}[Proof of Proposition \ref{P3}]
For each $Q \in \mathscr{D}$, we apply Corollary \ref{cor:mm} with $B=B_Q$.
This provides us with subsets $E_Q,F_Q\subset B_Q^*$ such that
\begin{equation*}
m_b^{\phi}(B_Q)
:=\phi(Q) \fint_{B_Q}|b-\langle b\rangle_{B_Q} |d\mu
\lesssim |\langle [b,T](\frac{1_{E_Q}}{\mu(B_Q)^{\frac{1}{2}}}),\frac{1_{F_Q}}{\mu(B_Q)^{\frac{1}{2}}}\rangle |.
\end{equation*} 
Letting
 $$e_Q : = \frac{1_{E_Q}}{\mu(B_Q)^{\frac{1}{2}}}, \qquad
 h_Q : = \frac{1_{F_Q}}{\mu(B_Q)^{\frac{1}{2}}},$$  
the estimate of $m_b^{\phi}$ can be rewritten as 
\begin{equation*}
m_b^{\phi}(B_Q)
\lesssim 	|\langle [b,T](e_Q), h_Q\rangle |,\quad
 \mathrm{with}
\quad
|e_Q|+| h_Q |\lesssim \frac{1_{B_Q^*}}{\mu(B_Q)^{\frac{1}{2}}}.
\end{equation*}

For $q\in [1,\infty]$, $p\in (1,\infty)$, $\ep \in (0,1)$ and $\alpha\in (0,\infty)$, we obtain, from the estimates right above followed by an application of Proposition \ref{P3.7}, that
\begin{align*}
\|b\|_{\mathrm{Osc}_{\phi}^{p,q}}(\mu)
&=\|\{m_b^{\phi}(B_Q)\}_{Q\in \mathscr{D}}\|_{\ell^{p,q}} \\
& \lesssim	\big\|\big\{|\langle [b,T](e_Q), h_Q\rangle |\big\}_{Q\in \mathscr{D}}\big\|_{\ell^{p,q}}\\
& \lesssim \|\mathcal{M}\|_{L^2\times L^2\rightarrow L^1}
\|[b,T]\|_{S^{p,q}}
 \lesssim \|[b,T]\|_{S^{p,q}}.
\end{align*}
as desired.                                                                           
\end{proof}

\section{Characterisation of constants}\label{sec:const}

Motivated by the result from \cite[Proposition 4.1]{HK}, 
we give a characterisation of constants via Besov-type conditions, with the assumption of the Poincar\'e inequality.

\begin{proposition}[\cite{HK}, Proposition 4.1]\label{lemma4.1}                                                                                                                                                                                                                                                                                                                                                                                                                                                                                                                                                                                                                                                                                                                                                                                                                                                                                                                                                                          Let $p \in [1,\infty)$ and $(X,\rho,\mu)$ be a doubling metric measure space supporting the $(1,p)$-Poincar\'e inequality. If $b\in L_{\mathrm{loc}}^1(\mu)$ satisfies                                                                                                                                                                                                                                                                                                                                                                                                                                                                                                                                                                                                                                                                                                                                                                                                                                                                                                                                                                           \begin{equation}\label{4.5}                                                                                                                                                                                                                                                                                                                                                                                                                                                                                                                                                                                                                                                                                                                                                                                                                                                                                                                                                     \lbrack (x,y)\mapsto \frac{|b(x)-b(y)|^p}{\rho(x,y)^p}\frac{1}{V(x,y)}\rbrack \in L_{\mathrm{loc}}^1(\mu\times \mu),	                                                                                                                                                                                                                                                                                                                                                                                                                                                                                                                                                                                                                                                                                                                                                                                                                                                                                                                              \end{equation}
then $f$ is equal to a constant almost everywhere. 
\end{proposition}

\begin{corollary}\label{corHK}
Let $d\in[1,\infty)$, $\eps\in[0,1-\frac1d]$, and $\alpha=\eps d\in[0,d-1]$.
%Let $\eps\in[0,1)$, $d\geq\frac{1}{1-\eps}$, and $\alpha=\eps d$.
%Let $p\in (1,\infty)$ and $\ep\in [0,\frac{1}{p})$. 
Let $(X,\rho,\mu)$ be a space of homogeneous type with lower dimension $d$
%$\widetilde{d}=\frac{p}{1-p\ep}$
and supporting the $(1,p)$-Poincar\'e inequality with $p=\frac{d}{1+\ep d}=\frac{d}{1+\alpha}$.
Let $\phi$ be as in \eqref{eq:phixy} and \eqref{phi}
If $b\in L_{\mathrm{loc}}^1(\mu)$ satisfies                                                                                                                                                                                                                                                                                                                                                                                                                                                                                                                                                                                                                                                                                                                                                                                                                                                                                                                       \begin{equation}\label{locBesov}                                                                                                                                                                                                                                                                                                                                                                                                                                                                                                                                                                                                                                                                                                                                                                                                                                                                                                                                             \lbrack (x,y)\mapsto \frac{|b(x)-b(y)|^p}{V(x,y)^2}\phi(x,y)^p\rbrack \in L_{\mathrm{loc}}^1(\mu\times \mu),	                                                                                                                                                                                                                                                                                                                                                                                                                                                                                                                                                                                                                                                                                                                                                                                                                                                                                                                    \end{equation}
then $b$ is equal to a constant almost everywhere. In particular,
\begin{equation}\label{BisConst}
		\B_{\frac{d}{1+\eps d}}(\phi,\mu)
		= \dot B^{\frac 1d}_{\frac{d}{1+\eps d}}(\mu)
		= \widetilde B^\alpha_{\frac{d}{1+\alpha}}(\mu)
		= \{constants\}.
\end{equation}
\end{corollary}

 \begin{proof}
Suppose that $b$ satisfies \eqref{locBesov}. 

 Fix a ball $B_0:=B(x_0,R)$. For all $x,y \in B_0$, we have 
 $$V(x,y)\lesssim \left(\frac{\rho(x,y)}{R}\right)^d\mu(B(x,R)) \lesssim \left(\frac{\rho(x,y)}{R}\right)^d \mu(B_0)\lesssim_{B_0}\rho(x,y)^d.$$	

{\bf Case 1:} $\phi(x,y)=V(x,y)^\eps$. Then
\begin{equation*}
  \frac{1}{\rho(x,y)^p V(x,y)}\lesssim\frac{1}{V(x,y)^{\frac pd+1}}
  =\frac{V(x,y)^{p(\frac1p-\frac 1d)}}{V(x,y)^2}
  =\frac{V(x,y)^{p\ep}}{V(x,y)^2}=\frac{\phi(x,y)^p}{V(x,y)^2}.
\end{equation*}

{\bf Case 2:} $\phi(x,y)=\rho(x,y)^\alpha$. Then
\begin{equation*}
  \frac{1}{\rho(x,y)^p V(x,y)}
  =\frac{\rho(x,y)^{\alpha p}}{\rho(x,y)^{(1+\alpha)p} V(x,y)}
  =\frac{\rho(x,y)^{\alpha p}}{\rho(x,y)^{d} V(x,y)}
  \lesssim\frac{\rho(x,y)^{\alpha p}}{V(x,y)^2}=\frac{\phi(x,y)^p}{V(x,y)^2}.
\end{equation*}

Hence, in either case,
\begin{equation*}
   \frac{|b(x)-b(y)|^p}{\rho(x,y)^p}\frac{1}{V(x,y)}
   \lesssim
   \frac{|b(x)-b(y)|^p}{V(x,y)^2}\phi(x,y)^p,
\end{equation*}
which is integrable over $B_0\times B_0$ by \eqref{locBesov}. Since this holds for every ball $B_0$, we see that \eqref{4.5} is satisfied. Then Proposition \ref{lemma4.1} shows that $b$ is constant.

Clearly, every $b\in \B_p(\phi,\mu)$ satisfies \eqref{locBesov}, and hence is constant by what we just proved. It follows from the definitions that
\begin{equation*}
  \B_p(V^\eps,\mu)=\dot B_p^{\frac1p-\eps}(\mu)=\dot B_{\frac{d}{1+\eps d}}^{\frac1d}(\mu),\qquad
  \B_p(\rho^\alpha,\mu)=\widetilde B_p^{\alpha}(\mu)=\widetilde B^{\alpha}_{\frac{d}{1+\alpha}}(\mu),
\end{equation*}
and hence both these spaces consists only of constants.
 \end{proof}

\section{Proof of the main results}\label{sec:synth}

\begin{proof}[Proof of Theorem \ref{T}]
The conclusion of Theorem \ref{T} (\ref{T:p>2}) is directly deduced from Corollary \ref{C1}. By combination of Proposition \ref{P2} (equivalence of the Besov norm and the oscillation norm) and Proposition \ref{P3} (domination of the oscillation norm by the Schatten norm), we obtain the result of Theorem \ref{T} (\ref{T:p<1}). Finally, Theorem \ref{T} (\ref{T:1<p<d}) follows from Corollary \ref{corHK}.
\end{proof}

\begin{proof}[Proof of Corollary \ref{C} \eqref{C:p>d}:]
Since $p\ge 2$, $\ep \in (0,1)$ and $\alpha \in (0,\infty)$, applying Theorem \ref{T} \eqref{T:p>2}, we get
$$\Norm{[b,T_\ep]}_{S^p}\lesssim\Norm{b}_{\dot{B}_p^{\frac1p-\ep}(\mu)}
\qquad\mathrm{and}\qquad
\Norm{[b,\widetilde{T_{\alpha}}]}_{S^p}\lesssim\Norm{b}_{\widetilde{B}_{p}^{\alpha}(\mu)}.$$ 
Since $p >1$ and $\ep \in (0,\frac{1}{p})\subset (0,1)$, applying Theorem \ref{T} \eqref{T:p<1}, we get 
$$\Norm{b}_{\dot{B}_p^{\frac1p-\ep}(\mu)}\lesssim \Norm{[b,T_\ep]}_{S^p}
\qquad\mathrm{and}\qquad
\Norm{b}_{\widetilde{B}_{p}^{\alpha}(\mu)}\lesssim \Norm{[b,\widetilde{T_{\alpha}}]}_{S^p}.$$ 
Combining the above two inequalities, we complete the proof of Corollary  \ref{C} \eqref{C:p>d}.
\end{proof}

\begin{proof}[Proof of Corollary \ref{C} \eqref{C:p<d}:]
Since $p\in (1,2)\subset (1,\infty)$ and $\ep \in (\max\{0,\frac{1}{p}-\frac{1}{d}\},\frac1p)\subset (0,\frac{1}{p})$, applying \ref{T} \eqref{T:p<1}, we obtain 
$$\Norm{b}_{\dot{B}_p^{\frac1p-\ep}(\mu)}\lesssim \Norm{[b,T_\ep]}_{S^p}
\qquad\mathrm{and}\qquad
\Norm{b}_{\widetilde{B}_{p}^{\alpha}(\mu)}\lesssim \Norm{[b,\widetilde{T_{\alpha}}]}_{S^p},$$
which complete the proof of Corollary  \ref{C} \eqref{C:p<d}. 
\end{proof}

\begin{proof}[Proof of Corollary \ref{C} \eqref{C:const}]
``$\Leftarrow$'': It is obvious that if $b$ is constants, then $[b,T_{\ep}],[b,\widetilde{T_{\alpha}}]\in S^p$ for all relevant parameters values.
 
``$\Rightarrow$'': For $\eps\in(0,\frac1p-\frac1d]\cap(0,1-\frac1d)$,
%$0<\ep \leq \frac{1}{p}-\frac{1}{d}$,
we have $p\leq \frac{d}{1+d\ep}=:q$,
where $\frac1q=\frac1d+\eps<1$ and hence $q>1$.
We apply Theorem \ref{T} \eqref{T:p<1} to $q$ in place of $p$ and \eqref{4.4} to get 
\begin{align*}
[b,T_{\ep}]\in S^p \subset S^q=S^{\frac{d}{1+d\ep}}\, \Rightarrow \,b\in 	\dot{B}_{\frac{d}{1+d\ep}}^{\frac{1}{d}}(\mu)\equiv \{constants\},
\end{align*}
and 
\begin{align*}
[b,\widetilde{T_{\alpha}}]\in S^p \subset S^q=S^{\frac{d}{1+\alpha}}\, \Rightarrow \,b\in 	\widetilde{B}_{\frac{d}{1+\alpha}}^{\alpha}(\mu)\equiv \{constants\}.
\end{align*}
\end{proof}

\begin{proof}[Proof of Corollary \ref{C'}]
Observe that for $0<\ep<\2-\frac{1}{d}$, we have $\frac{d}{1+d\ep}>2$.
By Remark \ref{R1.8}, the space $(X,\rho,\mu)$ also satisfies the $(1, \frac{d}{1+d\ep})$-Poincar\'e inequality.

\textit{Part \eqref{C:p>d'}:} If $p\in [\frac{d}{1+d\ep},\infty)\subset (2,\infty)$, then by Theorem \ref{T} \eqref{T:p>2} and \eqref{T:p<1}, we obtain 
\begin{equation}\label{eq7.6}
\begin{cases}
		[b,T_{\ep}]\in S^p &\iff  \,\,\, b \in \dot{B}_p^{\frac1p-\ep}(\mu),\\
		[b,\widetilde{T_{\alpha}}]\in S^p &\iff \,\,\, b \in \widetilde{B}_p^{\alpha}(\mu).
	\end{cases}
	\end{equation}
	In particular, this holds for $p \in (\frac{d}{1+d\ep},\infty)$.
	
\textit{Part \eqref{C:p<d'}:}
If $0<p\leq\frac{d}{1+d\ep}=:q$, then by taking $p=q$ in \eqref{eq7.6}
 and Theorem \ref{T} \eqref{T:1<p<d}, one has  
	\begin{align*}
[b,T_{\ep}]\in S^p \subset S^q=S^{\frac{d}{1+d\ep}}\, \Rightarrow \,b\in 	\dot{B}_{\frac{d}{1+d\ep}}^{\frac{1}{d}}(\mu)\equiv \{constants\},
\end{align*}
and 
\begin{align*}
[b,\widetilde{T_{\alpha}}]\in S^p \subset S^q=S^{\frac{d}{1+\alpha}}\, \Rightarrow \,b\in 	\widetilde{B}_{\frac{d}{1+\alpha}}^{\alpha}(\mu)\equiv \{constants\}.
\end{align*}
\end{proof}

\section{Fractional integrals with regular kernels}\label{sec:regular}

Our main results in this paper, whose proofs we have just completed, only depend on the size and not on the any regularity of the kernel $K$ of the fractional integral. Nevertheless, many examples of fractional integrals arising in applications also possess additional regularity, and we hence make some comments on this situation. In particular, we observe that, with some mild regularity of the kernel, the strong non-degeneracy that we have assumed (a uniform estimate over all points in certain balls) already follows from a simpler non-degeneracy condition involving a pair of points only.

\begin{definition}\label{D1reg}
A $\phi$-fractional integral kernel (Definition \ref{D1}) is called
\begin{enumerate}
 \item {\em $\omega$-regular} if it satisfies
\begin{equation}\label{1.2}
|K(x,y)-K(x',y)|+|K(y,x)-K(y,x')|\leq \omega\Big(\frac{\rho(x,x')}{\rho(x,y)}\Big)\frac{\phi(x,y)}{V(x,y)}
%H(x,y),	
\end{equation}
for $\rho(x,x')<(2A_0)^{-1}\rho(x,y)$, where $\omega$ is a bounded function with
 $\displaystyle\lim_{t\to 0}\omega(t)=0$;
     \item {\em non-degenerate} if there are positive constants $c_0$ and $\overline{C}$ such that for every $x \in X$ and $r>0$, there exists a point $y \in B(x,\overline{C}r) \setminus B(x,r)$ such that
\begin{equation}\label{Non-degenerate}
|K(x,y)|+|K(y,x)|\ge c_0\cdot \frac{\phi(x,y)}{V(x,y)},
%H(x,y),
\end{equation} 
 \end{enumerate}
\end{definition}

The following lemma clarifies the connection between the two versions of non-degeneracy.
 
 \begin{lemma}\label{lem:ndg}
 Let $K$ be a $\phi$-fractional integral kernel (Definition \ref{D1}). If $K$ is non-degenerate and $\omega$-regular (Definition \ref{D1reg}), then $K$ is strongly non-degenerate (Definition \ref{D1'}).
 \end{lemma}

\begin{proof}
Given $x_0\in X$ and $r>0$, we apply the condition of non-degeneracy with $x_0$ in place of $x$ and $Ar$ in place of $r$, where $A$ is yet to be chosen. This gives us a point $y_0\in B(x_0,\overline CAr)\setminus B(x_0,Ar)$ such that
\begin{equation*}
  \abs{K(x_0,y_0)}+\abs{K(y_0,x_0)}\geq c_0\frac{\phi(x_0,y_0)}{V(x_0,y_0)},
\end{equation*}
thus
\begin{equation*}
   \abs{K(x_0,y_0)}\geq \frac{c_0}{2}\frac{\phi(x_0,y_0)}{V(x_0,y_0)},\quad\text{or}\quad
   \abs{K(y_0,x_0)}\geq \frac{c_0}{2}\frac{\phi(x_0,y_0)}{V(x_0,y_0)}.
\end{equation*}
We assume the first case and proceed to prove \eqref{eq:sndg1}. In the second case, we obtain \eqref{eq:sndg2} analogously. Since $y_0\in B(x_0,\overline CAr)\setminus B(x_0,Ar)$, it follows that $\phi(x_0,y_0)\sim \phi(x_0,Ar)$ and $V(x_0,y_0)\sim V(x_0,Ar)$.

Let $x\in B(x_0,r)$, $y\in B(y_0,r)$. When $A$ is large, it is clear that we can apply the $\omega$-regularity below to estimate
\begin{equation}\label{eq:Kdiff}
\begin{split}
  \abs{K(x,y)-K(x_0,y_0)}
  &\leq\abs{K(x,y)-K(x_0,y)}+\abs{K(x_0,y)-K(x_0,y_0)} \\
  &\leq \omega\Big(\frac{\rho(x,x_0)}{\rho(x_0,y)}\Big)\frac{\phi(x_0,y)}{V(x_0,y)}
  +  \omega\Big(\frac{\rho(y,y_0)}{\rho(x_0,y_0)}\Big)\frac{\phi(x_0,y_0)}{V(x_0,y_0)} \\
  &\lesssim \Big[\omega\Big(\frac{r}{\rho(x_0,y)}\Big)
  +  \omega\Big(\frac{r}{\rho(x_0,y_0)}\Big)\Big]\frac{\phi(x_0,y_0)}{V(x_0,y_0)}.
\end{split}
\end{equation}
Here $Ar\leq\rho(x_0,y_0)\leq A_0(\rho(x_0,y)+\rho(y,y_0))< A_0\rho(x_0,y)+A_0 r$, and hence
\begin{equation*}
  \rho(x_0,y)>\frac{A-A_0}{A_0}r\geq 2A_0 r %\geq 2A_0\rho(x,x_0)
\end{equation*}
provided that $A\geq A_0+2A_0^2$.
%Then further
%\begin{equation*}
%  \rho(x_0,y_0)\leq A_0\rho(x_0,y)+A_0 r
%  \leq A_0\rho(x_0,y)+\frac12\rho(x_0,y).
%\end{equation*}
%Hence 
%\begin{equation*}
%  V(x_0,y_0)=V(x_0,\rho(x_0,y_0))
%  \lesssim V(x_0,\rho(x_0,y))
%  =V(x_0,y).
%\end{equation*}
%and thus $V(x_0,y)^{\eps-1}\lesssim V(x_0,y_0)^{\eps-1}$. Also $\rho(x_0,y)\leq A_0\rho(x_0,y_0)+A_0 r\lesssim \rho(x_0,y_0)$, and hence $\rho(x_0,y)^\alpha V(x_0,y)^{-1}\lesssim\rho(x_0,y_0)^\alpha V(x_0,y_0)^{-1}$. This means that $\phi(x_0,y)V(x_0,y)^{-1}\lesssim\phi(x_0,y_0)V(x_0,y_0)^{-1}$ in both cases fo $\phi$.

Substituting back to \eqref{eq:Kdiff}, we obtain
\begin{equation*}
  \abs{K(x,y)-K(x_0,y_0)}
  \lesssim \Big[\omega\Big(\frac{A_0}{A-A_0}\Big)
  +\omega\Big(\frac{1}{A}\Big)\Big]\frac{\phi(x_0,y_0)}{V(x_0,y_0)}.
\end{equation*}
Recalling that $\displaystyle\lim_{t\to 0}\omega(t)=0$, choosing $A=A(\eps)$ large enough, we can guarantee that
\begin{equation*}
  \abs{K(x,y)-K(x_0,y_0)}\leq \eps\frac{\phi(x_0,y_0)}{V(x_0,y_0)}
\end{equation*}
for any given $\eps>0$. Then clearly $\abs{K(x,y)}\geq\abs{K(x_0,y_0)}-\abs{K(x,y)-K(x_0,y_0)}$ has the required lower bound.

Moreover, denoting by $v$ the complex unit in the direction of $K(x,y)$, we have
\begin{equation*}
  \bar v K(x,y)
  =\bar v K(x_0,y_0)+\bar v [K(x,y)-K(x_0,y_0)]=:\sigma+\zeta
  =\sigma(1+\frac{\zeta}{\sigma}),
\end{equation*}
where $\sigma=\abs{K(x_0,y_0)}\gtrsim \phi(x_0,y_0) V(x_0,y_0)^{-1}=:\tau$, and $\abs{\zeta}\leq\eps\tau$; hence $\abs{\zeta/\sigma}\lesssim\eps$. Then
\begin{equation*}
  \abs{\arg(\bar v K(x,y))}=\abs{\arg(1+\frac{\zeta}{\sigma})}\lesssim\eps
\end{equation*}
can also made as small as desired with sufficiently small $\eps$, i.e., sufficiently large $A$.

Noting that, for $A$ large but fixed,
\begin{equation*}
  \frac{\phi(x_0,y_0)}{V(x_0,y_0)}\sim\frac{\phi(x_0,Ar)}{V(x_0,Ar)}\sim\frac{\phi(x_0,r)}{V(x_0,r)},
\end{equation*}
this completes the proof.
\end{proof}

\begin{example}\label{ex:basic2} 
The basic fractional integrals $I_\eps$ and $\tilde I_\alpha$ from \eqref{def} are strongly non-degenerate $\phi$-fractional integrals with $\phi(x,y)=V(x,y)^\eps$ and $\phi(x,y)=\rho(x,y)^\alpha$, respectively, but not necessarily $\omega$-regular.
\end{example}

\begin{proof}
The fact that $I_\eps$ and $\tilde I_\alpha$ are strongly non-degenerate $\phi$-fractional integrals was already observed in Example \ref{ex:basic}.

In a general space of homogeneous type, $V(x,y)$ need not be continuous, and hence $\omega$-regularity may easily fail. For example, let $n\geq 2$ and
\begin{equation*}
  X=\{x=(x_i)_{i=1}^n\in\R^n\mid \exists i:x_i\in\Z\}
\end{equation*}
with the $\ell^\infty$ metric and the $(n-1)$-dimensional Lebesgue measure. This is an Ahlors $(n-1)$-regular space of homogeneous type, but $V(x,y)$ has jumps at the points where $x$ or $y$ has more than one integer coordinate.
\end{proof}

\begin{remark}
 By Lemma \ref{lem:ndg}, $\omega$-regularity together with non-degeneracy is a sufficient condition for strong non-degeneracy, and this is often convenient in applications, but Example \ref{ex:basic} shows that this condition is not necessary; instead, strong non-degeneracy without any regularity is a strictly more general property. Hence we have formulated the main theorem below in terms of this latter condition.
 \end{remark}

\section{Fractional integrals arising from heat kernels}\label{heat kernel}

There is an extensive literature on heat kernels on manifolds and more general metric measure spaces, see e.g. \cite{GY,G,LY}. In this section, we show that negative fractional powers of the generator of a semigroup with a heat kernel are metric fractional integrals in the sense of this paper. In particular, their commutators will be in the scope of our results. We will detail this application in the specific setting of fractional Bessel operators in Section~\ref{sec:Bessel}, but we first deal with a more general setting here.

We begin by giving the definition of heat kernels and recalling related facts from \cite{G}.

\begin{definition}[Heat kernel]\label{def:heat-kernel}
Let $(X,\rho,\mu)$ be a metric measure space. A family $\{p_t\}_{t>0}$ of measurable functions on $X\times X$ 
is called a \emph{heat kernel} if for all $s,t>0$ and almost all $x,y \in X $, it satisfies
\begin{flalign*}
\quad \quad \text{(i)}\quad & p_t(x,y) \geq 0; &&\\
\text{(ii)}\quad & \int_{X} p_t(x,y)\, d\mu(y) = 1; &&\\
\text{(iii)}\quad & p_t(x,y) = p_t(y,x); &&\\
\text{(iv)}\quad & p_{s+t}(x,y) = \int_{X} p_s(x,z)p_t(z,y)\, d\mu(z); &&\\
\text{(v)}\quad & \lim_{t\to 0^+}\int_{X} p_t(x,y) f(y)\, d\mu(y)
= f(x)\quad\text{in the $L^2$-sense for all } f \in L^2(X,\mu). &&
\end{flalign*}
\end{definition}
For any $t>0$, the heat kernel $p_t(\cdot,\cdot)$ is the kernel of an operator $P_t$, which we write as $e^{-t\mathcal{L}}$.
The generator $\mathcal{L}$ of the semigroup $\{e^{-t\mathcal{L}}\}_{t>0}$ is defined by
\[
\mathcal{L}f := \lim_{t\to 0} \frac{f - e^{-t\mathcal{L}}(f)}{t},
\]
for those $f\in L^2(X,\mu)$ for which the limit exists in $L^2(X,\mu)$. 
The operator $\mathcal{L}$ is self-adjoint and positive definite. For any $s>0$, the corresponding fractional operator $\mathcal{L}^{-s}$ is defined by  
\begin{equation}\label{L(-s)}
\mathcal{L}^{-s} f(x)=\frac{1}{\Gamma(s)}\int_{0}^{\infty}e^{-t \mathcal{L}}(f)(x)\frac{dt}{t^{1-s}},\quad \forall x \in X.
\end{equation}

Next, we consider the following assumptions for the heat kernel $p_t(\cdot,\cdot)$.
	Let $\gamma>0$. Suppose that $\Phi_i:[0,\infty)\rightarrow [0,\infty)$ is a non-negative function for any $i=1,2,3$. The heat kernel $p_t(\cdot,\cdot) $ satisfies the following conditions: 
	\begin{enumerate}[\rm(i)]
		\item  \label{p_t 1}For all $x,y \in X $,
	\begin{equation}\label{p_t1}
p_t(x,y)\lesssim  \frac{1}{V(x,t^{\gamma})}
\Phi_1\big(\frac{\rho(x,y)}{t^{\gamma}}\big);
\end{equation} 

		\item  \label{p_t 1'}For all $x,y \in X $,
\begin{equation}\label{p_t1'}
p_t(x,y)\gtrsim  \frac{1}{V(x,t^{\gamma})}
\Phi_2\big(\frac{\rho(x,y)}{t^{\gamma}}\big);
\end{equation}

\item \label{p_t 2} For $\rho(x,x')<(2A_0)^{-1}\rho(x,y)$ and some $\ep>0$,
\begin{equation}\label{p_t2}
|p_t(x,y)-p_t(x',y)|+|p_t(y,x)-p_t(y,x')| \lesssim \frac{1}{V(x,t^{\gamma})}\big(\frac{\rho(x,x')}{t^{\gamma}}\big)^{\ep}
\Phi_3\big(\frac{\rho(x,y)}{t^{\gamma}}\big).
\end{equation}
 		\end{enumerate}
 		
Since the fundamental results of Li and Yau \cite{LY} proving bounds of this type for the heat kernel of the Laplace--Beltrami operator on a complete Riemannian manifold of nonnegative Ricci curvature (in which case $\gamma=\frac12$ and $\Phi_j(u)=e^{-c_j u^2}$), many further situations giving rise to such heat kernel bounds have been explored in the literature. A very general form of such bounds has been recently studied in \cite{GY}.

% There are many examples of heat kernels satisfying the conditions \eqref{p_t1} and \eqref{p_t2}, such as those associated with the Schr\"odinger operator on $\mathbb{R}^n$,  the sub-Laplacian on stratified Lie groups, the Laplace-Beltrami operator on complete Riemannian manifolds (see, e.g., \cite{HLLS}).
%Moreover, the heat kernels corresponding to the operator $\mathcal{L}=-\frac{1}{\omega(x)}\sum_{i,j}\partial_i(a_{ij}\partial_j)$ on $\mathbb{R}^n$ (cf. \cite{D}) and to the Bessel operators on $\mathbb{R}_+^{n+1}$ satisfies \eqref{p_t1}-\eqref{p_t2}. In particular, we will present more details about the heat kernels of Bessel operators on $\mathbb{R}_+^{n+1}$ in Section \ref{sec:Bessel}.   
  
 The following is the main result of this section:
 
\begin{proposition}\label{GP}
Let $\gamma,s,\ep>0$.
Let $(X,\rho,\mu)$ be a space of homogeneous type with a lower dimension $d>0$ and an upper dimension 	$D>0$. Suppose that the fractional operator $\mathcal{L}^{-s}$ is defined as in \eqref{L(-s)} and $p_t(\cdot,\cdot)$ is the heat kernel associated to $\mathcal{L}$. Let $K_{s}(\cdot,\cdot)$ be the kernel of the fractional operator $\mathcal{L}^{-s}$. 
Then the following statements hold: 
\begin{enumerate} [\rm(i)]
	\item 
If $p_t(\cdot,\cdot)$ satisfies \eqref{p_t1} and $\Phi_1$ satisfies 
\begin{equation}\label{Phi_1}
 \int_0^1    \Phi_1(w)\, w^{d-\frac{s}{\gamma}}\frac{dw}{w}+ 
 \int_1^{\infty}    \Phi_1(w)\, w^{D-\frac{s}{\gamma}}\frac{dw}{w}< \infty,
\end{equation}
then for all $x,y \in X$,
\begin{equation}\label{GP1}
 K_{s}(x,y)	\lesssim \frac{\rho(x,y)^{\frac{s}{\gamma}}}{V(x,y)};
\end{equation}

\item
If $p_t(\cdot,\cdot)$ satisfies \eqref{p_t1'} and $\Phi_2$ is non-zero in a set of positive measure,
then for all $x,y \in X$,
\begin{equation}\label{GP1'}
 K_{s}(x,y)	\gtrsim \frac{\rho(x,y)^{\frac{s}{\gamma}}}{V(x,y)};
\end{equation}

\item If  $p_t(\cdot,\cdot)$ satisfies \eqref{p_t2} and the function $\Phi_3$ satisfies 
 \begin{equation}\label{Phi_2}
 \int_0^1    \Phi_3(w)\, w^{d+\ep-\frac{s}{\gamma}}\frac{dw}{w}+ 
 \int_1^{\infty}    \Phi_3(w)\, w^{D+\ep-\frac{s}{\gamma}}\frac{dw}{w}< \infty,
\end{equation}
then for $\rho(x,x')<(2A_0)^{-1}\rho(x,y)$,
\begin{equation}\label{GP2}
|K_s(x,y)-K_s(x',y)|+|K_s(y,x)-K_s(y,x')| 	\lesssim \frac{\rho(x,y)^{\frac{s}{\gamma}}}{V(x,y)}\big( \frac{\rho(x,x')}{\rho(x,y)}\big)^{\ep}.
\end{equation}
 \end{enumerate}
\end{proposition}
\begin{proof}
By the doubling and reverse doubling property of the measure $V$, we consider the following two cases. For $0< u\leq  \rho(x,y)$,   
\begin{equation}\label{eq5.11}
		\big(\frac{\rho(x,y)}{u}\big)^{d}
		\lesssim \frac {V(x,y)}{V (x,u)} 
		=\frac {V(x,\rho(x,y))}{V (x,u)}
	\lesssim  
	\big(\frac{\rho(x,y)}{u}\big)^{D},
	\end{equation}
and for $u\ge \rho(x,y)$,  
\begin{equation}\label{eq5.11'}
		\big(\frac{\rho(x,y)}{u}\big)^{D}
		\lesssim \frac {V(x,y)}{V (x,u)} 
		=\frac {V(x,\rho(x,y))}{V (x,u)}
	\lesssim  
	\big(\frac{\rho(x,y)}{u}\big)^{d}.
	\end{equation}	 
By \eqref{p_t2} and a change of variable $u=t^{\gamma}$, we have for all $x,y \in X$,
\begin{equation}\label{8.15-}
	\begin{aligned}
K_{s}(x,y)
&=\frac{1}{\Gamma(s)}\int_{0}^{\infty}p_t(x,y)\frac{dt}{t^{1-s}}
= \frac{1}{\gamma \cdot \Gamma(s)}  \int_0^{\infty}
p_{u^{1/\gamma}}(x,y) \frac{du}{u^{1-\frac{s}{\gamma}}}\\
& \sim \Big(\int_0^{\rho(x,y)}+ \int_{\rho(x,y)}^{\infty}\Big)
p_{u^{1/\gamma}}(x,y) \frac{du}{u^{1-\frac{s}{\gamma}}}\\
&=: \uppercase\expandafter{\romannumeral1}+ 
\uppercase\expandafter{\romannumeral2},
	\end{aligned}
\end{equation}
and for $\rho(x,x')<(2A_0)^{-1}\rho(x,y)$,
\begin{equation}
	\begin{aligned}
	&	|K_s(x,y)-K_s(x',y)|+|K_s(y,x)-K_s(y,x')| \\
	&\lesssim  \int_{0}^{\infty} |p_t(x,y)-p_t(x',y)|+|p_t(y,x)-p_t(y,x')| \frac{dt}{t^{1-s}}\\
	& \lesssim \rho(x,x')^{\ep} \int_0^{\infty}\frac{1}{V(x,t^{\gamma})}
\Phi_3\big(\frac{\rho(x,y)}{t^{\gamma}}\big) \frac{dt}{t^{1+\ep\gamma-s}}\\
& \sim  \rho(x,x')^{\ep}\int_0^{\infty}
\frac{1}{V(x,u)}
\Phi_3\big(\frac{\rho(x,y)}{u}\big) \frac{du}{u^{1+\ep-\frac{s}{\gamma}}}\\
& = \rho(x,x')^{\ep}\Big(\int_0^{\rho(x,y)}+ \int_{\rho(x,y)}^{\infty}\Big)
\frac{1}{V(x,u)}
\Phi_3\big(\frac{\rho(x,y)}{u}\big) \frac{du}{u^{1+\ep-\frac{s}{\gamma}}}\\
& =: \uppercase\expandafter{\romannumeral3}+ 
\uppercase\expandafter{\romannumeral4}.
	\end{aligned}
\end{equation}
To proceed, we estimate the above terms $ \uppercase\expandafter{\romannumeral1} $-$ \uppercase\expandafter{\romannumeral4} $. 

First,
\begin{equation}
\begin{aligned}	
\uppercase\expandafter{\romannumeral1}
& \lesssim \int_0^{\rho(x,y)} 
      \frac{1}{V(x,u)}\Phi_1\!\left(\frac{\rho(x,y)}{u}\right) 
      \frac{du}{u^{1-\frac{s}{\gamma}}}  \quad \quad \quad \quad \quad \text{by}\, \eqref{p_t1}\\
 &\lesssim \frac{{\rho(x,y)}^D}{V(x,y)}
   \int_0^{\rho(x,y)} 
      \Phi_1\!\left(\frac{\rho(x,y)}{u}\right) 
      \frac{du}{u^{1+D-\frac{s}{\gamma}}} \,\quad \quad \quad \text{by}\, \eqref{eq5.11}\\
& =\frac{\rho(x,y)^D}{V(x,y)} 
\int_1^{\infty} 
   \Phi_1(w)\,
   \bigl(w^{-1} \rho(x,y)\bigr)^{\frac{s}{\gamma}-D}\,
   \frac{dw}{w}\qquad\quad\text{by }w=\frac{\rho(x,y)}{u}\\
&= \frac{\rho(x,y)^{\frac{s}{\gamma}}}{V(x,y)}
   \int_1^{\infty} 
      \Phi_1(w)\,
      w^{D-\frac{s}{\gamma}}\,
      \frac{dw}{w}
      \lesssim \frac{\rho(x,y)^{\frac{s}{\gamma}}}{V(x,y)} \qquad\text{by \eqref{Phi_1}}.
\end{aligned}
\end{equation}

Similarly,
\begin{equation}
\begin{aligned}	
\uppercase\expandafter{\romannumeral2}
 &\lesssim 
   \int_{\rho(x,y)}^{\infty} 
      \frac{1}{V(x,u)}\Phi_1\!\left(\frac{\rho(x,y)}{u}\right) 
      \frac{du}{u^{1-\frac{s}{\gamma}}}  \qquad \qquad  \text{by}\, \eqref{p_t1}\\
 &\lesssim \frac{{\rho(x,y)}^d}{V(x,y)}
   \int_{\rho(x,y)}^{\infty} 
      \Phi_1\!\left(\frac{\rho(x,y)}{u}\right) 
      \frac{du}{u^{1+d-\frac{s}{\gamma}}}  \qquad \qquad \text{by}\, \eqref{eq5.11'}\\
& =\frac{\rho(x,y)^d}{V(x,y)} 
\int_0^1 
   \Phi_1(w)\,
   \bigl(w^{-1} \rho(x,y)\bigr)^{\frac{s}{\gamma}-d}\,
   \frac{dw}{w}\qquad\quad\text{by }w=\frac{\rho(x,y)}{u}
   \\
&= \frac{\rho(x,y)^{\frac{s}{\gamma}}}{V(x,y)}
   \int_0^1 
      \Phi_1(w)\,
      w^{d-\frac{s}{\gamma}}\,
      \frac{dw}{w}
\lesssim \frac{\rho(x,y)^{\frac{s}{\gamma}}}{V(x,y)}     \qquad\text{by \eqref{Phi_1}},
\end{aligned}
\end{equation}
and 
\begin{equation}\label{8.20+}
\begin{aligned}
	\uppercase\expandafter{\romannumeral3}
	&\lesssim \frac{\rho(x,y)^{\frac{s}{\gamma}}}{V(x,y)}\big( \frac{\rho(x,x')}{\rho(x,y)}\big)^{\ep}
   \int_1^{\infty} 
      \Phi_3(w)w^{D+\ep-\frac{s}{\gamma}}\,
      \frac{dw}{w}
       \lesssim \frac{\rho(x,y)^{\frac{s}{\gamma}}}{V(x,y)}\big( \frac{\rho(x,x')}{\rho(x,y)}\big)^{\ep},
      \end{aligned}
\end{equation}
and 
\begin{equation}\label{8.20-}
\begin{aligned}
	\uppercase\expandafter{\romannumeral4}
	&\lesssim \frac{\rho(x,y)^{\frac{s}{\gamma}}}{V(x,y)}
	\big( \frac{\rho(x,x')}{\rho(x,y)}\big)^{\ep}
   \int_0^1 \Phi_3(w)
      w^{d+\ep-\frac{s}{\gamma}}\,
      \frac{dw}{w}
   \lesssim \frac{\rho(x,y)^{\frac{s}{\gamma}}}{V(x,y)}
	\big( \frac{\rho(x,x')}{\rho(x,y)}\big)^{\ep},
      \end{aligned}
\end{equation}
by \eqref{Phi_2} in the last steps of both \eqref{8.20+} and \eqref{8.20-}. This completes the proof of \eqref{GP1} and \eqref{GP2}.

It remains to consider the lower bound in \eqref{GP1'}. First,
\begin{align*}	
\uppercase\expandafter{\romannumeral1}
   & \gtrsim \int_0^{\rho(x,y)} 
      \frac{1}{V(x,u)}\Phi_2\!\left(\frac{\rho(x,y)}{u}\right) 
      \frac{du}{u^{1-\frac{s}{\gamma}}}  \quad \quad \quad \quad \quad \text{by}\, \eqref{p_t1'}\\
 &\gtrsim \frac{{\rho(x,y)}^d}{V(x,y)}
   \int_0^{\rho(x,y)} 
      \Phi_2\!\left(\frac{\rho(x,y)}{u}\right) 
      \frac{du}{u^{1+d-\frac{s}{\gamma}}}  \, \quad \quad \quad \text{by}\, \eqref{eq5.11}\\
& =\frac{\rho(x,y)^d}{V(x,y)} 
\int_1^{\infty} 
   \Phi_2(w)\,
   \bigl(w^{-1} \rho(x,y)\bigr)^{\frac{s}{\gamma}-d}\,
   \frac{dw}{w}\qquad\text{by }w=\frac{\rho(x,y)}{u}\\
&= \frac{\rho(x,y)^{\frac{s}{\gamma}}}{V(x,y)}
   \int_1^{\infty} 
      \Phi_2(w)\,
      w^{d-\frac{s}{\gamma}}\,
      \frac{dw}{w}
      \gtrsim \frac{\rho(x,y)^{\frac{s}{\gamma}}}{V(x,y)},
\end{align*}     
assuming, in the last step, that $\Phi_2$ is non-zero in a subset of $[1,\infty)$ of positive measure.

Similarly,
\begin{align*}	
\uppercase\expandafter{\romannumeral2}
   & \gtrsim \int_{\rho(x,y)}^{\infty} 
      \frac{1}{V(x,u)}\Phi_2\!\left(\frac{\rho(x,y)}{u}\right) 
      \frac{du}{u^{1-\frac{s}{\gamma}}}  \quad \quad \quad \quad \quad \text{by}\, \eqref{p_t1'}\\
 &\gtrsim \frac{{\rho(x,y)}^D}{V(x,y)}
   \int_{\rho(x,y)}^{\infty} 
      \Phi_2\!\left(\frac{\rho(x,y)}{u}\right) 
      \frac{du}{u^{1+D-\frac{s}{\gamma}}}  \, \quad \quad \quad \text{by}\, \eqref{eq5.11'}\\
& =\frac{\rho(x,y)^D}{V(x,y)} 
\int_0^1 
   \Phi_2(w)\,
   \bigl(w^{-1} \rho(x,y)\bigr)^{\frac{s}{\gamma}-D}\,
   \frac{dw}{w}\qquad\text{by }w=\frac{\rho(x,y)}{u}\\
&= \frac{\rho(x,y)^{\frac{s}{\gamma}}}{V(x,y)}
   \int_0^1 
      \Phi_2(w)\,
      w^{D-\frac{s}{\gamma}}\,
      \frac{dw}{w}
       \gtrsim \frac{\rho(x,y)^{\frac{s}{\gamma}}}{V(x,y)}, 
\end{align*}
assuming, in the last step, that $\Phi_2$ is non-zero in a subset of $[0,1]$ of positive measure.

Since $\Phi_2$ is non-zero in a subset of $[0,\infty)$ of positive measure, at least one the two lower bounds above is valid. Summing up, and noting that both terms are certainly non-negative, it follows that
\begin{equation*}
	K_{s}(x,y) = \uppercase\expandafter{\romannumeral1}+ \uppercase\expandafter{\romannumeral2}\gtrsim \frac{\rho(x,y)^{\frac{s}{\gamma}}}{V(x,y)},
\end{equation*}
which complete the proof of \eqref{GP1'}.
\end{proof}

\begin{remark}\label{R8.8}
We now present some examples of functions to illustrate Proposition \ref{GP}. These examples are used to provide intuition for the proposition and will be applied in the next section.
If $\Phi_1$ is bounded on $[0,1]$, then for any $0<s<\gamma d$,  
$$\int_0^1 \Phi_1(w)\, w^{d-\frac{s}{\gamma}}\frac{dw}{w}
\lesssim
 \int_0^1 w^{d-\frac{s}{\gamma}}\frac{dw}{w}
\sim
 1,
 $$
 If the function $u^{a}\, \Phi_1(u)$ is bounded on $[1,\infty)$ for some $a>D-\frac{s}{\gamma}$, then
 $$ \int_1^{\infty}    \Phi_1(w)\, w^{D-\frac{s}{\gamma}}\frac{dw}{w} 
= \int_1^{\infty}    \Phi_1(w)w^a\cdot  w^{D-a-\frac{s}{\gamma}}\frac{dw}{w} 
\lesssim  \int_1^{\infty}  w^{D-a-\frac{s}{\gamma}}\frac{dw}{w} \sim 1.
 $$
Similarly, if the function $\Phi_3$ is bounded on $[0,1]$ and $u^{\widetilde{a}}\,\Phi_3(u)=0$ is bounded on $[1,\infty)$ for some $\widetilde{a}>D+\ep-\frac{s}{\gamma}$, then for any $0<s<\gamma (d+\ep)$, the function $\Phi_3$ satisfies \eqref{Phi_2}.

In particular, if $\Phi_j(u)=\exp(-c_j u^{d_j})$ with $c_j,d_j>0$, then the conclusions of Proposition \ref{GP} are valid for all $s\in(0,\gamma d)$.
\end{remark}

\section{Application to fractional Bessel operators}\label{sec:Bessel}

In this section, we apply our main results to the fractional Bessel operator. We first recall the (non-fractional) Bessel operator from Huber \cite{Hu}. For $n\ge 0$ and $\lambda>0$, the $(n+1)$-dimension Bessel operator $\Delta_{\lambda}^{(n+1)}$ on $\mathbb{R}^{n+1}_+  :=\mathbb{R}^n \times (0,\infty) $ is defined by \eqref{eq5.1}. When $n=0$, we write $\mathbb{R}_+:=(0,\infty)$.
The operator $-\Delta_{\lambda}^{(n+1)} $ is symmetric and non-negative in $L^2(\mathbb{R}^{n+1}_+,dm_{\lambda}^{{(n+1)}})$, where 
$$dm_{\lambda}^{{(n+1)}}(x) := x_{n+1}^{2\lambda}dx.$$ 
% \prod_{j=1}^n dx_j x_{n+1}^{2\lambda}dx_{n+1}.$$ 
For simplicity, we write $\Delta_{\lambda}:= \Delta_{\lambda}^{(n+1)} $ and $m_{\lambda}:= m_{\lambda}^{{(n+1)}} $ when the dimension is clear from the context.

In \cite{BDLL}, Betancor et al. gave the kernel estimates of the one-dimensional fractional Bessel operator $(-\Delta_{\lambda}^{(1)} )^{-\alpha/2}$ on $(0,\infty)$ with parameters $n=0$, $\lambda>0$ and $0<\alpha<1+2\lambda$.  
In this section, we extend their result to the corresponding fractional Bessel operator $(-\Delta_{\lambda}^{(n+1)})^{-\alpha/2}$ for all $n\ge 0$, $\lambda>0$ and $0<\alpha<n+1+2\lambda$. Taking $\mathcal{L}=-\Delta_{\lambda}$ and $s=\alpha/2$ in \eqref{L(-s)}, the fractional Bessel operator is given by
\begin{equation}\label{Fr-B}
(-\Delta_{\lambda})^{-\alpha/2} f(x)=\frac{1}{\Gamma(\alpha/2)}\int_{0}^{\infty}e^{t\Delta_{\lambda}}(f)(x)\frac{dt}{t^{1-\alpha/2}},\quad \forall x \in \mathbb{R}^{n+1}_+ .
\end{equation}

\begin{corollary}\label{PB2}
Let $n\ge 0$, $\lambda>0$ and $0<\alpha<n+1$. Let $(-\Delta_{\lambda})^{-\alpha/2} $ be the fractional Bessel operator in $L^2(\mathbb{R}^{n+1}_+,\abs{x-y},m_{\lambda}) $. Let $ S^p:=S^p(L^2(dm_\lambda^{(n+1)})) $ and $\widetilde{B}_{p}(m_{\lambda}):= \widetilde{B}_{p}^{\alpha}(m_{\lambda}^{(n+1)}) $ be defined as in \eqref{b2 variant} with $(X,\rho,\mu)=(\mathbb{R}^{n+1}_+,\abs{x-y},m_{\lambda})$.
Then the following conclusions hold for all $b \in L_{\mathrm{loc}}^1(\mathbb{R}^{n+1}_+)$:
\begin{enumerate}[\rm(1)]

\item\label{C8.2:p>d variant} If $p\in [2,\infty) $ and $\alpha \in (0,\frac{n+1}{p})$, then $[b, (-\Delta_{\lambda})^{-\alpha/2}]\in S^p$ if and only if $b \in \widetilde{B}_{p}^{\alpha}(m_{\lambda}) $.
  
\item\label{C8.3:p<d variant} If $p \in (1,2)$ and $\alpha \in  (\frac{n+1}{p}-1, \frac{n+1}{p} )$, then $[b, (-\Delta_{\lambda})^{-\alpha/2}]\in S^p$ only if $b \in \widetilde{B}_{p}^{\alpha}(m_{\lambda}) $.

\item\label{C8.3:const variant} If $p\in(0,n+1)$ and $\alpha \in (0,n) \cap (0,\frac{n+1}{p}-1]$,
then $[b, (-\Delta_{\lambda})^{-\alpha/2}]\in S^p$ if and only if $b$ is constant.
\end{enumerate}
\end{corollary}

We intend to apply Corollary \ref{C} to prove Corollary \ref{PB2}. (Similarly, Corollary \ref{C'} implies Corollary \ref{PB2'}, as we already sketched in the Introduction.) Therefore, we need to verify that the space $(X,\rho,\mu)=(\R_+^{n+1},|x-y|,m_\lambda)$ and the operator $\tilde T_\alpha=(-\Delta_\lambda)^{-\frac{\alpha}{2}}$ satisfy the assumptions of Corollary \ref{C}. Concerning the space, we can quote the following results:

\begin{lemma}[\cite{FLL}, Eq. (2.2)]\label{L8.2}
	Let $n\ge 0$ and $\lambda > 0$. For any $x\in \mathbb{R}^{n+1}_+ $ and $r>0$, let $B_{\mathbb{R}^{n+1}_+}(x,r)=B(x,r)\,\cap \mathbb{R}^{n+1}_+ $. Then for every $x=(x_1,\dots,x_{n+1})\in \mathbb{R}^{n+1}_+ $ and $r>0$,
\begin{equation}\label{5.2}
	V_{\lambda}(x,r):=m_{\lambda}(B_{\mathbb{R}^{n+1}_+}(x,r))\sim r^{n+1}x_{n+1}^{2\lambda}+r^{n+1+2\lambda}.
\end{equation}
\end{lemma}

\begin{proposition}[\cite{H2}, Proposition 4.2]\label{Bessel}
For every $n\ge 0$ and $\lambda>0$, the space $(\mathbb{R}^{n+1}_+,|x-y|,m_{\lambda}) $ is a space of homogeneous type with lower dimension $d=n+1$ and upper dimension $D=n+1+2\lambda$. Moreover, $(\mathbb{R}^{n+1}_+,|x-y|,m_{\lambda}) $ satisfies the $(1,1)$-Poincar\'e inequality.
\end{proposition}

By Remark \ref{R1.8}, $(\mathbb{R}^{n+1}_+,|x-y|,m_{\lambda}) $ also satisfies the $(1,\frac{n+1}{\alpha+1})$-Poincar\'e inequality for all $\alpha \in (0,n)$.

It remains to show the relevant conditions for the kernel of the fractional Bessel operators. We will derive these from bounds for the heat kernels associated to Bessel operators from \cite[Section 7.5]{DGK}.
\begin{enumerate}[\rm(1)]

\item If $n=0$, then $\Delta_{\lambda}= \Delta_{\lambda}^{(1)} $ and the heat kernel associated to $ \Delta_{\lambda}^{(1)} $ is 
\begin{equation}\label{eq5.2}
W_{t}^{\lambda}(x,y)=
\frac{(xy)^{-\lambda+\frac{1}{2}}}{2t}e^{-\frac{x^2+y^2}{4t}}I_{\lambda-\frac{1}{2}}(\frac{xy}{2t}),	\quad \forall x, y \in \mathbb{R}_+,
\end{equation}
 where $I_{\nu}$ is the modified Bessel function of the first kind with order $\nu>-1$ (see \cite{BHV} for more details).
  The kernel of the operator $(\Delta_{\lambda}^{(1)})^{-\alpha/2} $ is obtained by 
\begin{equation}\label{eq5.5'}
K_{\lambda,\alpha}(x,y)= \frac{1}{\Gamma(\alpha/2)}\int_{0}^{\infty} W_{t}^{\lambda}(x,y)\frac{dt}{t^{1-\alpha/2}}	.
\end{equation}
for any $x, y \in \mathbb{R}_+ $ and $x\neq y$.  

\item If $n\ge 1$, then the operator $\Delta_{\lambda} =\Delta_{\lambda}^{(n+1)}$ can be written as $\Delta_{\lambda}= \Delta^{(n)}+ \Delta_{\lambda}^{(1)} $, where $\Delta^{(n)} $ denotes the standard Laplacian on $\mathbb{R}^n$ and $\Delta_{\lambda}^{(1)} $ denotes the Bessel operator on $\mathbb{R}_+$. Hence, it is clear that $e^{-t\Delta_{\lambda}} = e^{-t\Delta^{(n)}} \cdot e^{-t\Delta_{\lambda}^{(1)}} $ and the heat kernel associated to $\Delta_{\lambda} $ is
\begin{equation}\label{eq:BesselHeat}
  K_{e^{-t\Delta_{\lambda}}}(x,y)= K_{e^{-t\Delta^{(n)}}}(x',y')\cdot W_{t}^{\lambda} (x_{n+1},y_{n+1}),
\end{equation}
for $x=(x',x_{n+1}),y=(y',y_{n+1})\in \mathbb{R}^n \times (0,\infty)$, where 
$$K_{e^{-t\Delta^{(n)}}}(x',y'):=e^{-\frac{|x'-y'|^2}{4t}} (4\pi t)^{-n/2}, $$ 
is the heat kernel of standard Laplacian and $W_{t}^{\lambda} $ is the heat kernel of $\Delta_{\lambda}^{(1)} $ as in (\ref{eq5.2}). The kernel of the operator $(\Delta_{\lambda}^{(n+1)})^{-\alpha/2} $ is obtained by \eqref{Fr-B},
\begin{equation}\label{eq5.5}
K_{\lambda,\alpha}(x,y)= \frac{1}{\Gamma(\alpha/2)}\int_{0}^{\infty} K_{e^{-t\Delta^{(n)}}}(x',y') \cdot W_{t}^{\lambda}(x_{n+1},y_{n+1})\frac{dt}{t^{1-\alpha/2}}	.
\end{equation}
for any $x=(x',x_{n+1}),y=(y',y_{n+1})\in \mathbb{R}^{n}\times (0,\infty)$ and $x\neq y$.
\end{enumerate}

Letting $n=0$ in \eqref{eq5.5}, we get \eqref{eq5.5'} by interpreting $K_{e^{-t\Delta^{(0)}}}\equiv 1$. Hence, \eqref{eq5.5} is well defined for all $n\ge 0$.

\begin{remark}\label{R8.3}
Since $ K_{e^{-t\Delta^{(n)}}}(x',y') = K_{e^{-t\Delta^{(n)}}}(y',x') $ and, by \eqref{eq5.2}, also $$W_{t}^{\lambda}(x_{n+1},y_{n+1}) = W_{t}^{\lambda}(y_{n+1},x_{n+1}),$$ it is immediate that $K_{\lambda,\alpha}(x,y) = K_{\lambda,\alpha}(y,x) $ 
for any $n\ge 0$, $x,y\in \mathbb{R}^{n+1}_+$ with $x\neq y$.  
\end{remark}

As for the case $n=0$, the size estimate and the smoothness estimate of the kernel $K_{{\lambda,\alpha}} $ associated with the fractional Bessel operator have been proved by Betancor et.al. in \cite[Lemmas 5.1, 5.2]{BDLL}. The following Proposition \ref{P5.5} extends their result to all $n\ge 0$ and gives the lower bound of the kernel $K_{{\lambda,\alpha}} $.    

\begin{proposition}\label{P5.5}
Let $n\ge 0$, $\lambda>0$ and $0<\alpha<n+1$. 
%There exists a positive constant $C_{K_{{\lambda,\alpha}}}$ depending on these parameters, such that 
Then for all $x,y \in \mathbb{R}_+^{n+1} $ with $x \neq  y$,
\begin{equation}\label{5.9}
 K_{\lambda,\alpha}(x,y)
\sim  \frac{|x-y|^{\alpha}}{V_{\lambda}(x,y)}
:= \frac{|x-y|^{\alpha}}{m_{\lambda}(B_{\mathbb{R}^{n+1}_+}(x,|x-y|))},
\end{equation}
and for $|x-x'|<|x-y|/2$,
\begin{equation}\label{5.10}
|K_{\lambda,\alpha}(x,y)-K_{\lambda,\alpha}(x',y)|+|K_{\lambda,\alpha}(y,x)-K_{\lambda,\alpha}(y,x')|\lesssim \frac{|x-y|^{\alpha}}{V_{\lambda}(x,y)}\cdot \frac{|x-x'|}{|x-y|}.
\end{equation}
In particular, the fractional Bessel operator $(-\Delta_{\lambda})^{-\alpha/2} $ is a strongly non-degenerate metric fractional integral operator with kernel $K_{\lambda,\alpha}$ satisfying \eqref{1.1}, \eqref{1.2}, and the non-degenerate condition \eqref{Non-degenerate}, where
$\phi(x,y)=|x-y|^{\alpha}$, $V(x,y)=V_{\lambda}(x,y)$, and $\omega(t)=t$.
\end{proposition}

We will derive Proposition \ref{P5.5} from related heat kernel bounds using Proposition \ref{GP}, which holds in particular for $(X,\rho,\mu)=(\mathbb{R}^{n+1}_+,|x-y|,m_{\lambda})$. This requires checking the relevant assumptions for the heat kernel associated to Bessel operators $\Delta_{\lambda}$. The relevant upper bounds have already been verified in \cite{FLL}:

\begin{lemma}[\cite{FLL}, Lemma 2.4]\label{P5.3}
Let $n\ge 0$. For all multi-indices $\beta\in \mathbb{N}^{n+1}$, there exists positive constants $C_{\beta} ,c>0$ such that for all $x,y \in \mathbb{R}_+^{n+1} $,
\begin{equation}\label{5.7}
|\partial_x^{\beta} K_{e^{-t\Delta_{\lambda}}}(x,y)| + |\partial_y^{\beta} K_{e^{-t\Delta_{\lambda}}}(x,y)| 
\leq \frac{C_{\beta}}{t^{\frac{|\beta|}{2}}V_{\lambda}(x,\sqrt t)}\exp(-c\frac{|x-y|^2}{t}).	
\end{equation}
\end{lemma}

To obtain the required lower bound, we first present the following lemma on the modified Bessel functions $I_{\nu}$:

\begin{lemma}
For any $u>0$ and $\nu>-1$, %there exists constants $0<\ep <1$ and $M>1$ such that  
the modified Bessel functions satisfy
\begin{equation}\label{eq8.36}
  I_{\nu}(u) \sim 
	\begin{cases}
	u^{\nu} &\textup{if } u\in (0,1], \\
%	u^{\nu} \,\sim  u^{-\2}e^{u}	&\textup{if } u\in [\ep,M],\\
	u^{-\2}e^{u}	 &\textup{if }  u\in [1,\infty).
\end{cases}
	\end{equation}
\end{lemma}

\begin{proof}
We recall the asymptotic properties of the modified Bessel function $I_{\nu}$ from \cite{BHV}, p.109. For any $\nu>-1$,
\begin{equation*}
	\lim_{u\rightarrow 0^+}u^{-\nu}I_{\nu}(u)=\frac{1}{2^{\nu}\Gamma(\nu+1)},
\end{equation*}
and 
\begin{equation*} 
	\lim_{u\rightarrow +\infty} \frac{\sqrt{2\pi u}}{e^u} \cdot I_{\nu}(u)=1.
\end{equation*}
This implies that there exist constants $0<\ep <1$ and $M>1$ such that
\begin{equation*}
  I_{\nu}(u) \sim 
	\begin{cases}
	u^{\nu} &\textup{if } u\in (0,\ep], \\
	u^{-\2}e^{u}	 &\textup{if }  u\in [M,\infty).
\end{cases}
	\end{equation*}
On the other hand, one of definitions of $I_{\nu}(u)$ for any $u>0$ and $\nu>-1$ is  
$$
I_{\nu}(u)= \sum_{n=0}^{\infty}\frac{(u/2)^{2n+\nu}}{\Gamma(n+1)\Gamma(n+\nu+1)}.
$$
This shows that the function $I_{\nu}$ is positive and continuous. Consequently, 
\begin{equation*}
	I_{\nu}(u)  \sim 1 \sim
	u^{\nu} \sim
	u^{-\2}e^{u}	\quad \quad 
	 \textup{if }u \in [\ep,M].
\end{equation*}
Hence, we get \eqref{eq8.36}.
\end{proof}

We establish the following lower bound of the heat kernel $ K_{e^{-t\Delta_{\lambda}}} $.

\begin{lemma}\label{L8.10}
Let $n \ge 0$ and $\delta>0$. For all $x,y \in \mathbb{R}_+^{n+1} $,	\begin{equation}\label{5.7'}
 K_{e^{-t\Delta_{\lambda}}}(x,y) 
\gtrsim  \frac{1}{V_{\lambda}(x,\sqrt t)} \exp\Big({-(\frac{1}{4}+\delta)\cdot\frac{|x-y|^2}{t}}\Big).	
\end{equation}
\end{lemma}

\begin{proof}
We now take $r=\sqrt t$ in \eqref{5.2}, which implies that
\begin{equation}\label{8.24}
V_{\lambda}(x,\sqrt t)
\sim 	t^{\frac{n+1}{2}}(x_{n+1}^{2 \lambda}+t^{\lambda}).
\end{equation}
 By using the identity
\begin{equation*}
e^{-\frac{|x-y|^2}{4t}}= e^{-\frac{|x'-y'|^2}{4t}}\cdot e^{-\frac{|x_{n+1}-y_{n+1}|^2}{4t}} ,
\end{equation*}
we rewrite the heat kernel \eqref{eq:BesselHeat} (where $W_t^\lambda$ is given in \eqref{eq5.2}) as 
\begin{equation}\label{8.25}
\begin{aligned}
		 K_{e^{-t\Delta_{\lambda}}}(x,y) &=
		 \frac{1}{(4\pi t)^{\frac{n}{2}}}	 e^{-\frac{|x'-y'|^2}{4t}}
\frac{(x_{n+1}y_{n+1})^{-\lambda+\frac{1}{2}}}{2t}
 I_{\lambda-\frac{1}{2}}(\frac{x_{n+1}y_{n+1}}{2t})
 e^{-\frac{x_{n+1}^2+y_{n+1}^2}{4t}}\\
&\sim t^{-\frac{n+1+2\lambda}{2}} e^{-\frac{|x-y|^2}{4t}}\cdot e^{-\frac{x_{n+1}y_{n+1}}{2t}}
\big(\frac{x_{n+1}y_{n+1}}{2t}\big)^{-\lambda+\2} 
I_{\lambda-\frac{1}{2}}(\frac{x_{n+1}y_{n+1}}{2t})\\
&=: t^{-\frac{n+1+2\lambda}{2}} e^{-\frac{|x-y|^2}{4t}} D_t(x_{n+1},y_{n+1}).
\end{aligned}	
\end{equation}

{\bf{Case 1:} $\frac{x_{n+1}y_{n+1}}{2t} <1$.}
Taking $\nu=\lambda-\2 $ in \eqref{eq8.36}, we conclude that 
\begin{equation}\label{8.26}
\begin{aligned}
D_t(x_{n+1},y_{n+1})
&\sim e^{-\frac{x_{n+1}y_{n+1}}{2t}}
\big(\frac{x_{n+1}y_{n+1}}{2t}\big)^{-\lambda+\2} \,\big(\frac{x_{n+1}y_{n+1}}{2t}\big)^{\lambda-\2} \\
& = e^{-\frac{x_{n+1}y_{n+1}}{2t}} \sim 1.
\end{aligned}	
\end{equation}
Hence
\begin{equation}\label{case1}
\begin{aligned}
			 K_{e^{-t\Delta_{\lambda}}}(x,y) 
			 & \sim
			 t^{-\frac{n+1+2\lambda}{2}} e^{-\frac{|x-y|^2}{4t}}
			 \quad \quad \quad \quad \quad \text{by}\,\eqref{8.25},\eqref{8.26}\\
			 			  & \gtrsim  \frac{1}{V_{\lambda}(x,\sqrt t)} e^{-\frac{|x-y|^2}{4t}}  \quad \quad \quad \quad \,\,\, \text{by}\,\eqref{8.24}\\
			 			 &\geq \frac{1}{V_{\lambda}(x,\sqrt t)} 
			 e^{-(\frac{1}{4}+\delta)\cdot\frac{|x-y|^2}{t}}. 
\end{aligned}	
\end{equation}

{\bf{Case 2:} $\frac{x_{n+1}y_{n+1}}{2t} >1$.}
We note that
\begin{align*}
		x_{n+1}y_{n+1}
		&\leq x_{n+1}(x_{n+1}+|x_{n+1}-y_{n+1}|)
		= x_{n+1}^2+x_{n+1}|x_{n+1}-y_{n+1}|\\
		& \leq 2x_{n+1}^2+|x_{n+1}-y_{n+1}|^2\lesssim x_{n+1}^2+|x-y|^2.
\end{align*}
Hence
\begin{equation}\label{8.28'}
\begin{aligned}
	(x_{n+1}y_{n+1})^\lambda
	&\lesssim x_{n+1}^{2\lambda}+|x-y|^{2\lambda} 
	= x_{n+1}^{2\lambda}+t^\lambda\frac{|x-y|^{2\lambda}}{t^\lambda}
	\lesssim (x_{n+1}^{2\lambda}+t^\lambda) e^{\delta\frac{|x-y|^2}{t}},
\end{aligned}	
\end{equation}
using $1\leq e^x$ and $x^a\lesssim e^x$ for all $x,a>0$ in the last inequality.

Taking $\nu=\lambda-\2 $ in \eqref{eq8.36} and recalling that $\frac{x_{n+1}y_{n+1}}{2t}>1$, we conclude that 
\begin{equation}\label{8.28}
\begin{aligned}
D_t(x_{n+1},y_{n+1})
&\sim e^{-\frac{x_{n+1}y_{n+1}}{2t}}
\big(\frac{x_{n+1}y_{n+1}}{2t}\big)^{-\lambda+\2} 
\cdot e^{\frac{x_{n+1}y_{n+1}}{2t}}
\big(\frac{x_{n+1}y_{n+1}}{2t}\big)^{-\2} \\
& = \big(\frac{x_{n+1}y_{n+1}}{2t}\big)^{-\lambda}
 \gtrsim t^{\lambda}(x_{n+1}^{2\lambda}+t^\lambda)^{-1}e^{-\delta\frac{\abs{x-y}^2}{t}}.
\end{aligned}	
\end{equation}

Thus we obtain 
\begin{equation}\label{case2}
\begin{split}
			 K_{e^{-t\Delta_{\lambda}}}(x,y) 
			 &\sim t^{-\frac{n+1+2\lambda}{2}} e^{-\frac{|x-y|^2}{4t}} D_t(x_{n+1},y_{n+1})
			 \qquad\text{by \eqref{8.25}} \\
			 & \gtrsim t^{-\frac{n+1}{2}} 
			 (x_{n+1}^{2 \lambda}+t^{\lambda})^{-1} e^{-(\frac{1}{4}+\delta)\frac{|x-y|^2}{t}}\qquad\text{by \eqref{8.28}} \\
%			 & \gtrsim t^{-\frac{n+1}{2}} e^{-\frac{|x-y|^2}{4t}} 
%			 \cdot e^{-\frac{|x-y|^2}{8t}}
%			 (x_{n+1}^{2 \lambda}+t^{ \lambda})^{-1} 
%			 \quad \quad \text{by}\,\eqref{8.28'}\\
			 &\sim \frac{1}{V_{\lambda}(x,\sqrt t)} 
			 e^{-(\frac14+\delta)\cdot\frac{|x-y|^2}{t}}. 
			 \qquad\text{by}\,\eqref{8.24}
\end{split}	
\end{equation}
Hence, combining \eqref{case1} with \eqref{case2}, we obtain \eqref{5.7'}. 
\end{proof}

\textit{Proof of Proposition \ref{P5.5}:}
Choose parameters  $$s=\frac{\alpha}{2}, \, \gamma=\frac{1}{2}, \, d=n+1, \,\ep=1,\, D=n+1+2\lambda$$ and
$$\mathcal{L}=-\Delta_{\lambda},\,p_t=K_{e^{-t\Delta_{\lambda}}},\,K_s=K_{\lambda,\alpha},\, V(\cdot,\cdot):=V_{\lambda}(\cdot,\cdot)$$ in Proposition \ref{GP}. 
Taking $|\beta|=0,1$ in Lemma \ref{P5.3} and by Lemma \ref{L8.10}, we obtain that $K_{e^{-t\Delta_{\lambda}}}$ satisfies \eqref{p_t1}-\eqref{p_t2} with $\Phi_j(u)=e^{-c_ju^2} $, where $c_1=c_3=c$ is the same as in Lemma \ref{P5.3}, and $c_2=\frac14+\delta$ for some $\delta>0$.

Obviously, $\Phi_2$ is positive in a set of positive measure. By Remark \ref{R8.8}, it is easy to verify that the functions $\Phi_1,\Phi_3 $ satisfy conditions \eqref{Phi_1} and \eqref{Phi_2}, respectively. 
Then, by Proposition \ref{GP}, the kernel $K_s=K_{\lambda,\alpha}$ satisfies \eqref{GP1}, \eqref{GP1'} and \eqref{GP2}. Thus, we obtain the desired estimates \eqref{5.9} and \eqref{5.10}.

\end{document}